\documentclass[]{article}

\usepackage[english]{babel}
\usepackage{amsmath,amssymb,amsthm}
\usepackage{xfrac,enumitem,bm}

\usepackage{tikz}
\usetikzlibrary{shapes.geometric}
\usepackage{pgfplots}
\usepackage{soul,xcolor}
\setstcolor{red}
\setul{}{2pt}

\usepackage{graphicx}
\usepackage{mathtools}
\usepackage{booktabs}
\usepackage{xspace}
\usepackage[hidelinks]{hyperref}

\usepackage{chngcntr}
\usepackage{cases}
\usepackage{cleveref}

\usepackage{adjustbox}
\usepackage{comment}

\usepackage{subfig}

\newenvironment{cpf}{\begin{trivlist} \item[] {\em Proof of Claim.}}{\hspace*{\stretch{1}} $\diamond$ \end{trivlist}}

\DeclareMathOperator{\cone}{cone} 
\DeclareMathOperator{\conv}{conv} 
\DeclareMathOperator{\cl}{cl} 
\DeclareMathOperator{\inte}{int} 
\DeclareMathOperator{\relint}{relint}

\newcommand{\mbb}[1]{\mathbb{#1}}
\newcommand{\mbfs}[1]{\boldsymbol{#1}}
\newcommand{\mcf}[1]{\mathcal{#1}}

\definecolor{MediumRed}{rgb}{0.925, 0.345, 0.345}
\definecolor{MediumGreen}{rgb}{0.37, 0.7, 0.66}
\definecolor{MediumBlue}{rgb}{0.015, 0.315, 0.45}
\definecolor{MediumPurple}{RGB}{153, 102, 203}
\definecolor{JungleGreen}{rgb}{0.16, 0.67, 0.53}

\newtheorem{rmk}{Remark}
\newtheorem{cm}{Claim}
\newtheorem{theorem}{Theorem}
\newtheorem{lemma}{Lemma}
\newtheorem{definition}{Definition}

\title{A characterization of maximal inhomogeneous-quadratic-free sets}
\author{
  Gonzalo Mu\~noz\thanks{Department of Industrial Engineering, Universidad de Chile, Santiago, Chile (gonzalo.m@uchile.cl)}
  \and Joseph Paat\thanks{Sauder School of Business, University of British Columbia, Vancouver
BC, Canada (joseph.paat@ubc.ca)}
  \and Felipe Serrano\thanks{COPT GmbH, Berlin, Germany (serrano@copt.de)}
}

\date{}

\begin{document}
\maketitle

\begin{abstract}
The intersection cut framework is a versatile tool for generating valid inequalities in optimization.
Its main ingredients are so-called $S$-free sets: convex sets whose interiors do not intersect a given set $S$.
Among these, inclusion-wise maximal $S$-free sets are particularly important, as they yield the strongest intersection cuts.
In the integer programming setting, maximal lattice-free sets are well studied and admit explicit characterizations.
In the quadratic optimization context, Mu\~noz, Paat, and Serrano (2025) characterized maximal $S$-free sets when $S$ is defined by a homogeneous quadratic inequality.
In this work, we characterize maximal $S$-free sets when $S$ is defined by an inhomogeneous quadratic inequality.
As in the homogeneous case, our characterization is built using non-expansive functions.
Together with the results in the homogeneous case, our results complete a characterization of \emph{every} maximal quadratic-free set via non-expansive functions.
\end{abstract}%

\section{Introduction}

Let $S \subseteq \mbb{R}^d$ be a closed set.
 A closed convex set $C \subseteq \mbb{R}^d$ is {\bf $\bm S$-free} if the interior of $C$ is disjoint from $S$.
An $S$-free set is {\bf maximal} $S$-free if it is not a strict subset of another $S$-free set. 

A major application of maximal $S$-free sets is the generation of cutting planes in optimization. Such cuts are created using the intersection cut framework, which we describe next. Consider an optimization problem
\[
\min \{\mbfs{c}{}^\top \mbfs{s}:\ \mbfs{s} \in P \cap S\},
\]
where $\mbfs{c} \in \mbb{R}^d$ and $P \subseteq \mbb{R}^d$ is a closed set that models constraints aside from inclusion in $S$.
Let $\overline{\mbfs{s}} \not\in S$ be a vertex of a polyhedral relaxation of $P \cap S$.
The intersection cut framework constructs a valid linear inequality for $P \cap S$ that strictly separates $\overline{\mbfs{s}}$, using an $S$-free set $C$ such that $\overline{\mbfs{s}}$ is in the interior of $C$.
Among all intersection cuts, those derived from maximal $S$-free sets yield the strongest cuts. 
For a general reference on intersection cuts including their precise derivation from an $S$-free set, we refer the reader to~\cite[Chapter 6]{CCZ2014}.

In this article, we study maximal $Q_g$-free sets, where
\[
Q_g := Q_g(\mbfs{a}, \mbfs{d}) := \{ (\mbfs{x},\mbfs{y}) \in \mbb{R}^n \times \mbb{R}^m :\ \|\mbfs{x}\| \le \|\mbfs{y}\|~\text{and}~ \mbfs{a}{}^\top\mbfs{x} + \mbfs{d}{}^\top\mbfs{y} = -1\}
\]
for $(\mbfs{a}, \mbfs{d}) \in \mbb{R}^n \times \mbb{R}^m$ with $n,m\geq 1$ and $\max\{\|\mbfs{a}\|, \|\mbfs{d}\|\} = 1$. 
Here, $\|\cdot\|$ denotes the $\ell_2$-norm.
Together with maximal $Q_h$-free sets, where
\[
Q_h := \left\{ (\mbfs{x},\mbfs{y}) \in \mbb{R}^n \times \mbb{R}^m : \|\mbfs{x}\| \le \|\mbfs{y}\|\right\},
\]
the study of maximal $Q_g$-free sets is sufficient to describe \emph{every} maximal quadratic-free set, that is, $S$-free sets when $S$ has the form
\[
\{ \mbfs{s} \in \mbb{R}^{d}  :\ \mbfs{s}{}^\top \mbfs{A} \mbfs{s} +  \mbfs{b}{}^\top \mbfs{s} +  f \le0\}.
\]
The reduction from quadratic-free sets to $Q_h$- and $Q_g$-free sets is based on a homogenization and diagonalization argument; we describe this in detail in Section~\ref{sec:Reduction}.
A characterization of maximal $Q_h$-free sets was given in \cite{munoz2025characterization} using so-called non-expansive functions; this prior work characterizes all maximal {\it homogeneous}-quadratic-free sets.
In this article, we characterize maximal $Q_g$-free sets using non-expansive functions as well, consequently completing a characterization of all maximal {\it inhomogeneous}-quadratic-free sets.
Hence, our work completes a characterization of {\it all} maximal quadratic-free sets.

Our characterization is explicit and enables the construction of a wide range of maximal quadratic-free sets beyond what is currently known.
Prior work has implemented intersection cuts using only one family of such sets, with positive results~\cite{chmiela2022implementation,chmiela2025monoidal}.
A complete description of all maximal quadratic-free sets significantly expands known families of valid inequalities for non-convex QCQPs, thus offering a new avenue to strengthen cutting-plane methods in this setting.
%

\noindent{\bf Characterizing full-dimensional maximal $Q_g$-free sets.}

Mu\~noz et al. \cite{munoz2025characterization} focus on full-dimensional maximal $Q_h$-free sets.
This emphasis on full-dimensionality follows from the fact that lower dimensional sets have empty interiors and are vacuously $Q_h$-free.
Hence, all lower dimensional maximal $Q_h$-free sets are hyperplanes. 

Similarly, we search for sets that are not vacuously $Q_g$-free. 
Observe that $Q_g = Q_h \cap H$, where
\[
H := H(\mbfs{a}, \mbfs{d}) := \left\{ (\mbfs{x},\mbfs{y}) \in \mbb{R}^n \times \mbb{R}^m : \mbfs{a}{}^\top\mbfs{x} + \mbfs{d}{}^\top\mbfs{y} = -1\right\}.
\]
We say a set $K \subseteq H$ is {\bf full-dimensional} if it is $(n+m-1)$-dimensional.
A full-dimensional set $K \subseteq H$ is said to be {\bf $\bm{Q_g}$-free} if its relative interior is disjoint from $Q_g$.
Using these definitions, we search for full-dimensional maximal $Q_g$-free sets.
These definitions have the benefit of allowing us to analyze maximal $Q_g$-free sets in the same ambient space as $Q_h$.

Mu\~noz et al.\ begin their characterization by showing that every full-dimensional maximal $Q_h$-free set has the form $C_{\Gamma}$ for some $\Gamma:D^m \to D^n$, where $D^k := \{\mbfs{x} \in \mbb{R}^k: \|\mbfs{x}\| = 1\}$, and
\[
C_{\Gamma} := \left\{(\mbfs{x}, \mbfs{y}) \in \mbb{R}^n\times \mbb{R}^m:\ -\Gamma(\mbfs{\beta}){}^\top\mbfs{x} + \mbfs{\beta}{}^\top \mbfs{y} \le 0 ~~\forall~ \mbfs{\beta} \in D^m\right\}.
\]
Figure~\ref{fig0} shows an example of $Q_h$ and $C_{\Gamma}$.

\begin{theorem}[Mu\~{n}oz  et al.~{\cite[Theorem 1.1]{munoz2025characterization}}]\label{thmHCGamma}
Each full-dimensional maximal $Q_h$-free set has the form $C_{\Gamma}$ for some $\Gamma:D^m \to D^n$.
\end{theorem}

\begin{figure}[t]
\centering
\begin{tikzpicture}[scale = .4,
declare function={
PX(\x,\y,\z) = -\x + .45*\y ;
PY(\x,\y,\z) = 0*\x+.25*\y + \z;
} 
]

\draw[draw = blue!50]({PX(-1.5,1.5,-1.5)}, {PY(-1.5,1.5,-1.5)}) to ({PX(1.5,-1.5,1.5)}, {PY(1.5,-1.5,1.5)});
\draw[draw = blue!50, fill = blue!50, opacity = .25]({PX(-1.5,1.5,-1.5)}, {PY(-1.5,1.5,-1.5)}) 
to ({PX(-1.5+1*1,1.5+6*1,-1.5+1*1)}, {PY(-1.5+1*1,1.5+6*1,-1.5+1*1)})
to ({PX(-1.5+1*1+3,1.5+6*1-3,-1.5+1*1+3)}, {PY(-1.5+1*1+3,1.5+6*1-3,-1.5+1*1+3)})
to ({PX(-1.5+3,1.5-3,-1.5+3)}, {PY(-1.5+3,1.5-3,-1.5+3)})
to cycle;

\draw[draw = blue!75, fill = blue!75, opacity = .25]({PX(-1.5,1.5,-1.5)}, {PY(-1.5,1.5,-1.5)}) 
to ({PX(-1.5-0*1,1.5+1.5*1,-1.5-1.5*1)}, {PY(-1.5-0*1,1.5+1.5*1,-1.5-1.5*1)})
to ({PX(-1.5-0*1+3,1.5+1.5*1-3,-1.5-1.5*1+3)}, {PY(-1.5-0*1+3,1.5+1.5*1-3,-1.5-1.5*1+3)})
to ({PX(-1.5+3,1.5-3,-1.5+3)}, {PY(-1.5+3,1.5-3,-1.5+3)})
to cycle;

\draw[draw = blue!75, fill = blue!75, opacity = .25]({PX(-1.5,1.5,-1.5)}, {PY(-1.5,1.5,-1.5)}) 
to ({PX(-1.5-0*1,1.5+1.5*1,-1.5-1.5*1)}, {PY(-1.5-0*1,1.5+1.5*1,-1.5-1.5*1)})
to ({PX(-1.5+1*1,1.5+6*1,-1.5+1*1)}, {PY(-1.5+1*1,1.5+6*1,-1.5+1*1)})
to cycle;



\draw[opacity = .15]({PX(0,0,0)}, {PY(0,0,0)}) to({PX(0,0,1.4)}, {PY(0,0,1.4)});
\draw[->]({PX(0,0,0)}, {PY(0,0,0)}) to node[left, pos = 1]{\color{black}$x_1$} ({PX(3,0,0)}, {PY(3,0,0)});
\draw[->]({PX(0,0,0)}, {PY(0,0,0)}) to node[left, pos = 1]{\color{black}$-x_2$} ({PX(0,-5,0)}, {PY(0,-5,0)});

\begin{scope}%
\color{red!70}
\pgfsetstrokeopacity{.65}
\foreach \x in {2, 1.975, ..., 0.05}{
\pgfpathellipse{\pgfpointxy{0}{-1*\x}}{\pgfpointxy{-1*\x}{.15*\x}}{\pgfpointxy{-.75*\x}{-.25*\x}} 
}

\foreach \x in {0.05, 0.1, ..., 2}{
\pgfpathellipse{\pgfpointxy{0}{1*\x}}{\pgfpointxy{-1*\x}{.15*\x}}{\pgfpointxy{-.75*\x}{-.25*\x}} 
}
\pgfusepath{stroke}
\end{scope}%

\begin{scope}
\color{red!10}
\pgfsetfillopacity{.75}
\pgfpathellipse{\pgfpointxy{0}{1*1.95}}{\pgfpointxy{-1*1.95}{.15*1.95}}{\pgfpointxy{-.75*1.95}{-.25*1.95}} 
\pgfusepath{fill}
\end{scope}

\draw[->]({PX(0,0,1.4)}, {PY(0,0,1.4)}) to node[left, pos = 1]{\color{black}$y$} ({PX(0,0,3)}, {PY(0,0,3)});

\node[red] at ({PX(3,0,2.5)}, {PY(3,0,2.5)}){$Q_h$};
\node[blue] at ({PX(-4,0,2)}, {PY(-4,0,2)}){$C_{\Gamma}$};
\end{tikzpicture}
\caption{The sets $Q_h$ for $(n,m) = (2,1)$ and $C_{\Gamma}$ for $\Gamma(1) := (1,0)$ and $\Gamma(-1) := (0,1)$.
}\label{fig0}
\end{figure}
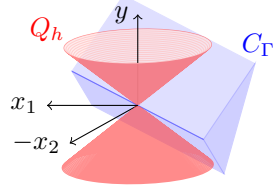

We give an analogue of Theorem~\ref{thmHCGamma} for maximal $Q_g$-free sets by showing that each maximal $Q_g$-free set can obtained starting from some set $C_{\Gamma} \cap H$ by possibly removing some inequalities.
More precisely, for a function $\Gamma: D^m \to D^n$ and a set $I\subseteq D^m$, define
\begin{equation}\label{eq:def_CIGamma}
C^I_\Gamma:= \left\{(\mbfs{x}, \mbfs{y}) \in \mathbb{R}^n \times \mathbb{R}^m:\ -\Gamma(\mbfs{\beta}){}^\top\mbfs{x} + \mbfs{\beta}{}^\top \mbfs{y} \le 0 ~~\forall~ \mbfs{\beta} \in I \right\}.
\end{equation}
Note that $C_{\Gamma} = C^{D^m}_{\Gamma}$.
The important inequalities for our maximality characterization are indexed by
\begin{equation}\label{defnG}
G := G(\Gamma, \mbfs{a}, \mbfs{d}) := \{\mbfs{\beta} \in D^m :\ \mbfs{a}{}^\top\Gamma(\mbfs{\beta}) + \mbfs{d}{}^\top\mbfs{\beta} \le 0 \}.
\end{equation}
Our first result shows we can restrict attention to sets of the form $C^G_{\Gamma}\cap H$.

\begin{theorem}\label{thm:Both}
Each full-dimensional maximal $Q_g$-free set equals $C^G_{\Gamma} \cap H$ for some $\Gamma:D^m \to D^n$ such that $C_{\Gamma}$ is a full-dimensional maximal $Q_h$-free set.
\end{theorem}

Mu\~{n}oz  et al.\ leverage Theorem~\ref{thmHCGamma} to characterize full-dimensional maximal $Q_h$-free sets using properties of $\Gamma$.
A function $\Gamma:D^m \to D^n$ is {\bf non-expansive} if $\|\Gamma(\mbfs{\beta}) - \Gamma(\overline{\mbfs{\beta}})\| \le \|\mbfs{\beta} - \overline{\mbfs{\beta}}\|$ for all $\mbfs{\beta}, \overline{\mbfs{\beta}} \in D^m$.

\begin{theorem}[Mu\~{n}oz  et al.~{\cite[Theorems 1.1 and 1.2]{munoz2025characterization}}]\label{thmHQuad}
Let $\Gamma:D^m \to D^n$.
The set $C_{\Gamma}$ is a full-dimensional maximal $Q_h$-free set if and only if $\Gamma$ is non-expansive and $\mbfs{0}\not\in \conv(\{\Gamma(\mbfs{\beta}) \,:\ \mbfs{\beta}\in D^m\})$.
\end{theorem}

The set $C_{\Gamma}$ in Figure~\ref{fig0} is maximal according to Theorem~\ref{thmHQuad}.
This result effectively characterizes all maximal homogeneous-quadratic-free sets, i.e., $S$-free sets when $S$ has the form $\{ \mbfs{s} \in \mbb{R}^{d}  : \mbfs{s}{}^\top \mbfs{A} \mbfs{s} \le0\}$.

Similar to Theorem~\ref{thmHQuad}, we leverage Theorem~\ref{thm:Both} to characterize maximal $Q_g$-free sets using properties of $\Gamma$.
Our characterization has two cases based on which is larger, $\|\mbfs{a}\|$ or $\|\mbfs{d}\|$; this case distinction was also used in \cite{MS2020,MS2022}.
The first case assumes $\|\mbfs{a}\| \le \|\mbfs{d}\| = 1$.
See Figure~\ref{fig1} for an example.
In this case, $C^G_{\Gamma} \cap H$ is a maximal $Q_g$-free set for almost any choice of $\Gamma$ such that $C_\Gamma$ is maximal $Q_h$-free.
The only exception is if $\mbfs{a} = \Gamma(-\mbfs{d})$.
In this case, $-\mbfs{a}{}^\top\mbfs{x}-\mbfs{d}{}^\top \mbfs{x} \le 0$ is valid for $C^G_{\Gamma}$ yet invalid for $H$, so $C^G_{\Gamma} \cap H = \emptyset$.

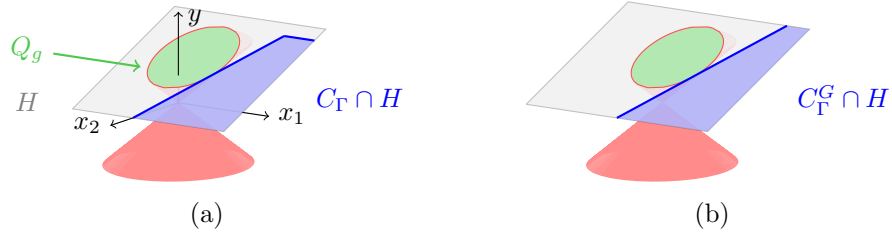
\begin{figure}[t]
\centering
\begin{tabular}{c@{\hskip 1.5 cm}c}
\begin{tikzpicture}[scale = .4,
declare function={
PX(\x,\y,\z) = \x - .75*\y ;
PY(\x,\y,\z) = -.15*\x-.25*\y + \z;
} 
]

\draw[opacity = .15]({PX(0,0,0)}, {PY(0,0,0)}) to ({PX(0,0,0.9)}, {PY(0,0,0.9)});
\draw[opacity = .15]({PX(0,0,0)}, {PY(0,0,0)}) to ({PX(3,0,0)}, {PY(3,0,0)});
\draw[opacity = .15]({PX(0,0,0)}, {PY(0,0,0)}) to ({PX(0,3,0)}, {PY(0,3,0)});

\draw[draw = none] ({PX(-2,2.15,-2)}, {PY(-2,2.15,-2)}) to ({PX(-2,2.15,2)}, {PY(-2,2.15,2)}) to ({PX(-2,-2.15,2)}, {PY(-2,-2.15,2)}) to ({PX(2.15,-2.15,2)}, {PY(2.15,-2.15,2)}) to ({PX(2.15,-2.15,-2)}, {PY(2.15,-2.15,-2)}) to ({PX(2.15,2.15,-2)}, {PY(2.15,2.15,-2)}) to cycle;

\begin{scope}%
\clip({PX(-2,2,0)}, {PY(-2,2,0)}) to ({PX(3,2,0)}, {PY(3,2,0)}) to ({PX(3,2,0)}, {PY(3,2,0)}) to ({PX(2.15,-2.15,-2)}, {PY(2.15,-2.15,-2)}) to ({PX(2.15,2.15,-2)}, {PY(2.15,2.15,-2)}) to ({PX(-2,2.15,-2)}, {PY(-2,2.15,-2)}) to cycle;
\color{red!70}
\pgfsetstrokeopacity{.65}
\foreach \x in {2, 1.975, ..., 0.05}{
\pgfpathellipse{\pgfpointxy{0}{-1*\x}}{\pgfpointxy{-1*\x}{.15*\x}}{\pgfpointxy{-.75*\x}{-.25*\x}} 
}
\pgfusepath{stroke}
\end{scope}%

\begin{scope}%
\clip
	({PX({-sqrt(8*.75 - 3*.75*.75 - 4)},2-2*.75,.75)}, {PY({-sqrt(8*.75 - 3*.75*.75 - 4)},2-2*.75,.75)}) 
	to ({PX(-2,2,0)}, {PY(-2,2,0)})
	to ({PX(3,2,0)}, {PY(3,2,0)})
	to ({PX(2,-2,2)}, {PY(2,-2,2)}) 
	to ({PX(0,-2,2)}, {PY(0,-2,2)}) 
	to cycle;

\color{red!70}
\pgfsetstrokeopacity{0.15}
\foreach \x in {0.05, 0.1, ..., 2}{
\pgfpathellipse{\pgfpointxy{0}{1*\x}}{\pgfpointxy{-1*\x}{.15*\x}}{\pgfpointxy{-.75*\x}{-.25*\x}} 
}
\foreach \x in {2, 1.975, ..., 0.05}{
\pgfpathellipse{\pgfpointxy{0}{-1*\x}}{\pgfpointxy{-1*\x}{.15*\x}}{\pgfpointxy{-.75*\x}{-.25*\x}} 
}
\pgfusepath{stroke}
\end{scope}%

\begin{scope}
\draw[draw =black, fill = black!20, opacity = .25]({PX(-2,2,0)}, {PY(-2,2,0)}) to ({PX(-2,-5/2,9/4)}, {PY(-2,-5/2,9/4)}) to ({PX(3,-5/2,9/4)}, {PY(3,-5/2,9/4)}) to ({PX(3,2,0)}, {PY(3,2,0)}) to cycle;
\end{scope}

\begin{scope}
\draw[fill=green!70!black!30, draw =  red!70]({PX({sqrt(8*2/3 - 3*2/3*2/3 - 4)},2/3,2/3)}, {PY({sqrt(8*2/3 - 3*2/3*2/3 - 4)},2/3,2/3)})
	\foreach \i in {.667, .72, ..., 2} {
	to ({PX({sqrt(8*\i - 3*\i*\i - 4)},2-2*\i,\i)}, {PY({sqrt(8*\i - 3*\i*\i - 4)},2-2*\i,\i)}) 
	}
	\foreach \i in {2, 1.95, ..., .667} {
	to ({PX({-sqrt(8*\i - 3*\i*\i - 4)},2-2*\i,\i)}, {PY({-sqrt(8*\i - 3*\i*\i - 4)},2-2*\i,\i)}) 
	}
	to cycle;
\end{scope}

\draw[->]({PX(0,0,.9)}, {PY(0,0,.9)}) to node[right, pos = .9]{\color{black}$y$} ({PX(0,0,3)}, {PY(0,0,3)});
\draw[->]({PX(2.1,0,0)}, {PY(2.1,0,0)}) to node[right, pos = 1]{\color{black}$x_1$} ({PX(3,0,0)}, {PY(3,0,0)});
\draw[->]({PX(0,2,0)}, {PY(0,2,0)}) to node[left, pos = 1]{\color{black}$x_2$} ({PX(0,3,0)}, {PY(0,3,0)});

\draw[draw = none, fill = blue!50, opacity = .5]({PX(0,2,0)}, {PY(0,2,0)}) to ({PX(3, 2,0)}, {PY(3, 2,0)}) to ({PX(3, -2,2)}, {PY(3, -2,2)}) to ({PX(2,-2,2)}, {PY(2,-2,2)}) to cycle;
\draw[blue, thick]({PX(2,-2,2)}, {PY(2,-2,2)}) to ({PX(3, -2,2)}, {PY(3, -2,2)});
\draw[blue, thick]({PX(2,-2,2)}, {PY(2,-2,2)}) to ({PX(0,2,0)}, {PY(0,2,0)});

\node[black!50] at ({PX(-3.5,2,0)}, {PY(-3.5,2,0)}) {$H$};
\node(qg)[green!70!black!70] at ({PX(-5,0,1)}, {PY(-5,0,1)}) {$Q_g$};
\node[blue] at ({PX(6,0,1)}, {PY(6,0,1)}) {$C_{\Gamma} \cap H$};
\draw[green!70!black!70, ->, thick] (qg) to ({PX(-1.25,0,1)}, {PY(-1.25,0,1)});

\end{tikzpicture}
&
\begin{tikzpicture}[scale = .4,
declare function={
PX(\x,\y,\z) = \x - .75*\y ;
PY(\x,\y,\z) = -.15*\x-.25*\y + \z;
} 
]

\draw[draw = none] ({PX(-2,2.15,-2)}, {PY(-2,2.15,-2)}) to ({PX(-2,2.15,2)}, {PY(-2,2.15,2)}) to ({PX(-2,-2.15,2)}, {PY(-2,-2.15,2)}) to ({PX(2.15,-2.15,2)}, {PY(2.15,-2.15,2)}) to ({PX(2.15,-2.15,-2)}, {PY(2.15,-2.15,-2)}) to ({PX(2.15,2.15,-2)}, {PY(2.15,2.15,-2)}) to cycle;

\begin{scope}%
\clip({PX(-2,2,0)}, {PY(-2,2,0)}) to ({PX(3,2,0)}, {PY(3,2,0)}) to ({PX(3,2,0)}, {PY(3,2,0)}) to ({PX(2.15,-2.15,-2)}, {PY(2.15,-2.15,-2)}) to ({PX(2.15,2.15,-2)}, {PY(2.15,2.15,-2)}) to ({PX(-2,2.15,-2)}, {PY(-2,2.15,-2)}) to cycle;
\color{red!70}
\pgfsetstrokeopacity{.65}
\foreach \x in {2, 1.975, ..., 0.05}{
\pgfpathellipse{\pgfpointxy{0}{-1*\x}}{\pgfpointxy{-1*\x}{.15*\x}}{\pgfpointxy{-.75*\x}{-.25*\x}} 
}
\pgfusepath{stroke}
\end{scope}%

\begin{scope}%
\clip
	({PX({-sqrt(8*.75 - 3*.75*.75 - 4)},2-2*.75,.75)}, {PY({-sqrt(8*.75 - 3*.75*.75 - 4)},2-2*.75,.75)}) 
	to ({PX(-2,2,0)}, {PY(-2,2,0)})
	to ({PX(3,2,0)}, {PY(3,2,0)})
	to ({PX(2,-2,2)}, {PY(2,-2,2)}) 
	to ({PX(0,-2,2)}, {PY(0,-2,2)}) 
	to cycle;

\color{red!70}
\pgfsetstrokeopacity{0.15}
\foreach \x in {0.05, 0.1, ..., 2}{
\pgfpathellipse{\pgfpointxy{0}{1*\x}}{\pgfpointxy{-1*\x}{.15*\x}}{\pgfpointxy{-.75*\x}{-.25*\x}} 
}
\foreach \x in {2, 1.975, ..., 0.05}{
\pgfpathellipse{\pgfpointxy{0}{-1*\x}}{\pgfpointxy{-1*\x}{.15*\x}}{\pgfpointxy{-.75*\x}{-.25*\x}} 
}
\pgfusepath{stroke}
\end{scope}%

\begin{scope}
\draw[draw =black, fill = black!20, opacity = .25]({PX(-3,2,0)}, {PY(-3,2,0)}) to ({PX(-3,-5/2,9/4)}, {PY(-3,-5/2,9/4)}) to ({PX(3,-5/2,9/4)}, {PY(3,-5/2,9/4)}) to ({PX(3,2,0)}, {PY(3,2,0)}) to cycle;
\end{scope}

\begin{scope}
\draw[fill=green!70!black!30, draw =  red!70]({PX({sqrt(8*2/3 - 3*2/3*2/3 - 4)},2/3,2/3)}, {PY({sqrt(8*2/3 - 3*2/3*2/3 - 4)},2/3,2/3)})
	\foreach \i in {.667, .72, ..., 2} {
	to ({PX({sqrt(8*\i - 3*\i*\i - 4)},2-2*\i,\i)}, {PY({sqrt(8*\i - 3*\i*\i - 4)},2-2*\i,\i)}) 
	}
	\foreach \i in {2, 1.95, ..., .667} {
	to ({PX({-sqrt(8*\i - 3*\i*\i - 4)},2-2*\i,\i)}, {PY({-sqrt(8*\i - 3*\i*\i - 4)},2-2*\i,\i)}) 
	}
	to cycle;
\end{scope}

\draw[draw = none, fill = blue!50, opacity = .5]({PX(2.25,-5/2,2.25)}, {PY(2.25,-5/2,2.25)}) to ({PX(0,2,0)}, {PY(0,2,0)}) to ({PX(3,2,0)}, {PY(3,2,0)}) to ({PX(3,-5/2,9/4)}, {PY(3,-5/2,9/4)}) to cycle;
\draw[blue, thick]({PX(2.25,-5/2,2.25)}, {PY(2.25,-5/2,2.25)}) to ({PX(0,2,0)}, {PY(0,2,0)});
\node[blue] at ({PX(6,0,1)}, {PY(6,0,1)}) {$C^G_{\Gamma} \cap H$};
\end{tikzpicture}

\\
(a) & (b)
\end{tabular}
\caption{
For $\mbfs{a} = (0,-1/2)$ and $\mbfs{d} = -1$, $H$ is added to Figure~\ref{fig0}. 
The axes are oriented to clarify the figures. 
The hyperplane $H$ is in gray, and $Q_g$ is in green.
(a) $C_{\Gamma} \cap H$ (in blue) is not maximal $Q_g$-free.
(b) Removing the inequality $-\Gamma(-1){}^\top \mbfs{x} - \mbfs{y} \le 0$ from the description of $C_{\Gamma}$ yields the maximal $Q_g$-free set $C^G_{\Gamma} \cap H$, shown in blue.
}\label{fig1}
\end{figure}

\begin{theorem}\label{thm:easycase}
Suppose $\|\mbfs{a}\|\leq \|\mbfs{d}\| = 1$.
Let $\Gamma: D^m \to D^n$ be such that $C_{\Gamma}$ is a full-dimensional maximal $Q_h$-free set.
Then $C^G_{\Gamma} \cap H$ is a full-dimensional maximal $Q_g$-free set if and only if $\mbfs{a} \neq \Gamma(-\mbfs{d})$.
\end{theorem}
Our proof of Theorem~\ref{thm:easycase} reveals another maximality characterization when $\|\mbfs{a}\|\leq \|\mbfs{d}\| = 1$; see Remark~\ref{rm:GcircG} for this similar characterization.

The second case of our characterization assumes $\|\mbfs{d}\| <  \|\mbfs{a}\| = 1$.
See Figure~\ref{fig2} for an example.
In this case, we need $G = D^m$ for maximality of $C^G_{\Gamma} \cap H$.
%

\begin{figure}[h]
\centering
\begin{tabular}{c@{\hskip 1 cm}c}
\begin{tikzpicture}[scale = .4,
declare function={
PX(\x,\y,\z) = \x - .75*\y ;
PY(\x,\y,\z) = -.15*\x-.25*\y + \z;
} 
]

\draw[opacity = .15]({PX(0,0,0)}, {PY(0,0,0)}) to ({PX(0,0,1.55)}, {PY(0,0,1.55)});
\draw[opacity = .15]({PX(0,0,0)}, {PY(0,0,0)}) to ({PX(1.25,0,0)}, {PY(1.25,0,0)});
\draw[opacity = .15]({PX(0,0,0)}, {PY(0,0,0)}) to ({PX(0,.7,0)}, {PY(0,.7,0)});

\draw[draw = none] ({PX(-2,2.15,-2)}, {PY(-2,2.15,-2)}) to ({PX(-2,2.15,2)}, {PY(-2,2.15,2)}) to ({PX(-2,-2.15,2)}, {PY(-2,-2.15,2)}) to ({PX(2.15,-2.15,2)}, {PY(2.15,-2.15,2)}) to ({PX(2.15,-2.15,-2)}, {PY(2.15,-2.15,-2)}) to ({PX(2.15,2.15,-2)}, {PY(2.15,2.15,-2)}) to cycle;

\begin{scope}%
\clip ({PX(1,-2,2)}, {PY(1,-2,2)}) to ({PX(-3,-2,2)}, {PY(-3,-2,2)}) to ({PX(-3,2,2)}, {PY(-3,2,2)}) to ({PX(-3,2,-2)}, {PY(-3,2,-2)})  to ({PX(1,2,-2)}, {PY(1,2,-2)}) to ({PX(1,2,2)}, {PY(1,2,2)}) to cycle;
\color{red!70}
\pgfsetstrokeopacity{.65}
\foreach \x in {2, 1.975, ..., 0.05}{
\pgfpathellipse{\pgfpointxy{0}{-1*\x}}{\pgfpointxy{-1*\x}{.15*\x}}{\pgfpointxy{-.75*\x}{-.25*\x}} 
}

\foreach \x in {0.05, 0.1, ..., 2}{
\pgfpathellipse{\pgfpointxy{0}{1*\x}}{\pgfpointxy{-1*\x}{.15*\x}}{\pgfpointxy{-.75*\x}{-.25*\x}} 
}
\pgfusepath{stroke}
\end{scope}%

\begin{scope}
\draw[draw =black, fill = black!20, opacity = .25]({PX(1,-2,2)}, {PY(1,-2,2)}) to ({PX(1,-2,-2)}, {PY(1,-2,-2)}) to ({PX(1,2,-2)}, {PY(1,2,-2)}) to ({PX(1,2,2)}, {PY(1,2,2)}) to cycle;

\draw[fill=green!70!black!30, draw =  none]({PX(1,0,1)}, {PY(1,0,1)})
	\foreach \i in {1, 1.1, ..., 2.1} {
	to ({PX(1,{sqrt(\i*\i - 1)},\i)}, {PY(1,{sqrt(\i*\i - 1)},\i)}) 
	}
	\foreach \i in {2, 1.9, ..., 1} {
	to ({PX(1,{-sqrt(\i*\i - 1)},\i)}, {PY(1,{-sqrt(\i*\i - 1)},\i)}) 
	}
	to cycle;
	
\draw[draw =  red!70]
	({PX(1,{-sqrt(2*2 - 1)},2)}, {PY(1,{-sqrt(2*2 - 1)},2)}) 
	\foreach \i in {2, 1.9, ..., 1} {
	to ({PX(1,{-sqrt(\i*\i - 1)},\i)}, {PY(1,{-sqrt(\i*\i - 1)},\i)}) 
	}
	to ({PX(1,0,1)}, {PY(1,0,1)})
	\foreach \i in {1, 1.1, ..., 2.1} {
	to ({PX(1,{sqrt(\i*\i - 1)},\i)}, {PY(1,{sqrt(\i*\i - 1)},\i)}) 
	};
	
\draw[fill=green!70!black!30, draw =  none]({PX(1,0,-1)}, {PY(1,0,-1)})
	\foreach \i in {1, 1.1, ..., 2.1} {
	to ({PX(1,{sqrt(\i*\i - 1)},-\i)}, {PY(1,{sqrt(\i*\i - 1)},-\i)}) 
	}
	\foreach \i in {2, 1.9, ..., 1} {
	to ({PX(1,{-sqrt(\i*\i - 1)},-\i)}, {PY(1,{-sqrt(\i*\i - 1)},-\i)}) 
	}
	to cycle;
	
\draw[draw =  red!70]
	({PX(1,{-sqrt(2*2 - 1)},-2)}, {PY(1,{-sqrt(2*2 - 1)},-2)}) 
	\foreach \i in {2, 1.9, ..., 1} {
	to ({PX(1,{-sqrt(\i*\i - 1)},-\i)}, {PY(1,{-sqrt(\i*\i - 1)},-\i)}) 
	}
	to ({PX(1,0,-1)}, {PY(1,0,-1)})
	\foreach \i in {1, 1.1, ..., 2.1} {
	to ({PX(1,{sqrt(\i*\i - 1)},-\i)}, {PY(1,{sqrt(\i*\i - 1)},-\i)}) 
	};

\end{scope}

\begin{scope}
\clip ({PX(1,-2,2)}, {PY(1,-2,2)}) to ({PX(1,-2,-2)}, {PY(2,-2,-2)}) to ({PX(1, 2,-2)}, {PY(1,2,-2)}) to ({PX(1,2,2)}, {PY(1,2,2)}) to cycle;

\draw[draw = none, fill = blue!50, opacity = .5]({PX(1,-1,1)}, {PY(1,-1,1)}) to ({PX(1,5,1)}, {PY(1,5,1)}) to ({PX(1,5,-5)}, {PY(1,5,-5)}) to cycle;
\end{scope}

\draw[blue, thick]({PX(1,-1,1)}, {PY(1,-1,1)}) to ({PX(1,2,1)}, {PY(1,2,1)});
\draw[blue, thick]({PX(1,-1,1)}, {PY(1,-1,1)}) to ({PX(1,2,-2)}, {PY(1,2,-2)});

\begin{scope}
\clip({PX(1,-2,2)}, {PY(1,-2,2)}) to ({PX(3,-2,2)}, {PY(3,-2,2)}) to ({PX(3,-2,-2)}, {PY(3,-2,-2)}) to ({PX(3,2,-2)}, {PY(3,2,-2)})  to ({PX(1,2,-2)}, {PY(1,2,-2)}) to ({PX(1,2,2)}, {PY(1,2,2)}) to cycle;
\color{red!20}
\pgfsetstrokeopacity{0.25}
\foreach \x in {2, 1.95, ..., 0.05}{
\pgfpathellipse{\pgfpointxy{0}{-1*\x}}{\pgfpointxy{-1*\x}{.15*\x}}{\pgfpointxy{-.75*\x}{-.25*\x}} 
}

\foreach \x in {0.05, 0.1, ..., 2}{
\pgfpathellipse{\pgfpointxy{0}{1*\x}}{\pgfpointxy{-1*\x}{.15*\x}}{\pgfpointxy{-.75*\x}{-.25*\x}} 
}
\pgfusepath{stroke}
\end{scope}

\draw[->]({PX(0,0,1.55)}, {PY(0,0,1.55)}) to node[pos = 1, left]{$y$} ({PX(0,0,3)}, {PY(0,0,3)});
\draw[->]({PX(1.25,0,0)}, {PY(1.25,0,0)}) to node[pos = 1, right]{$x_1$}({PX(3,0,0)}, {PY(3,0,0)});
\draw[->]({PX(0,.7,0)}, {PY(0,.7,0)}) to node[pos = 1, left]{$x_2$} ({PX(0,3,0)}, {PY(0,3,0)});

\end{tikzpicture}
&
\begin{tikzpicture}[scale = .4,
declare function={
PX(\x,\y,\z) = \x - .75*\y ;
PY(\x,\y,\z) = -.15*\x-.25*\y + \z;
} 
]

\draw[draw = none] ({PX(-2,2.15,-2)}, {PY(-2,2.15,-2)}) to ({PX(-2,2.15,2)}, {PY(-2,2.15,2)}) to ({PX(-2,-2.15,2)}, {PY(-2,-2.15,2)}) to ({PX(2.15,-2.15,2)}, {PY(2.15,-2.15,2)}) to ({PX(2.15,-2.15,-2)}, {PY(2.15,-2.15,-2)}) to ({PX(2.15,2.15,-2)}, {PY(2.15,2.15,-2)}) to cycle;

\begin{scope}%
\clip({PX(-2,-1,2)}, {PY(-2,-1,2)}) to ({PX(0,-2,2)}, {PY(0,-2,2)})  to ({PX(2,-2,2)}, {PY(2,-2,2)}) to ({PX(3,-1,2)}, {PY(3,-1,2)})  to cycle;
\color{red!70}
\pgfsetstrokeopacity{.65}
\foreach \x in {0.05, 0.1, ..., 2}{
\pgfpathellipse{\pgfpointxy{0}{1*\x}}{\pgfpointxy{-1*\x}{.15*\x}}{\pgfpointxy{-.75*\x}{-.25*\x}} 
}
\pgfusepath{stroke}
\end{scope}%

\begin{scope}
\draw[draw =black, fill = black!20, opacity = .25]({PX(-2,-1,2)}, {PY(-2,-1,2)}) to ({PX(3,-1,2)}, {PY(3,-1,2)}) to ({PX(3,-1,-2)}, {PY(3,-1,-2)}) to ({PX(-2,-1,-2)}, {PY(-2,-1,-2)}) to cycle;

\draw[fill=green!70!black!30, draw =  none]({PX(0,-1,1)}, {PY(0,-1,1)})
	\foreach \i in {1, 1.1, ..., 2.1} {
	to ({PX({sqrt(\i*\i - 1)},-1,\i)}, {PY({sqrt(\i*\i - 1)},-1,\i)}) 
	}
	\foreach \i in {2.0, 1.9, ..., 1} {
	to ({PX({-sqrt(\i*\i - 1)},-1,\i)}, {PY({-sqrt(\i*\i - 1)},-1,\i)}) 
	}
	to cycle;
	
\draw[draw =  red!70]
	({PX({-sqrt(2*2 - 1)},-1,2)}, {PY({-sqrt(2*2 - 1)},-1,2)}) 
	\foreach \i in {2.0, 1.9, ..., 1} {
	to ({PX({-sqrt(\i*\i - 1)},-1,\i)}, {PY({-sqrt(\i*\i - 1)},-1,\i)}) 
	}
	to ({PX(0,-1,1)}, {PY(0,-1,1)})
	\foreach \i in {1, 1.1, ..., 2.1} {
	to ({PX({sqrt(\i*\i - 1)},-1,\i)}, {PY({sqrt(\i*\i - 1)},-1,\i)}) 
	};
	
\draw[fill=green!70!black!30, draw =  none]({PX(0,-1,-1)}, {PY(0,-1,-1)})
	\foreach \i in {1, 1.05, ..., 2.05} {
	to ({PX({sqrt(\i*\i - 1)},-1,-\i)}, {PY({sqrt(\i*\i - 1)},-1,-\i)}) 
	}
	\foreach \i in {2.0, 1.95, ..., 1} {
	to ({PX({-sqrt(\i*\i - 1)},-1,-\i)}, {PY({-sqrt(\i*\i - 1)},-1,-\i)}) 
	}
	to cycle;

\draw[draw =  red!70]
	({PX({-sqrt(2*2 - 1)},-1,-2)}, {PY({-sqrt(2*2 - 1)},-1,-2)}) 
	\foreach \i in {2.0, 1.9, ..., 1} {
	to ({PX({-sqrt(\i*\i - 1)},-1,-\i)}, {PY({-sqrt(\i*\i - 1)},-1,-\i)}) 
	}
	to ({PX(0,-1,-1)}, {PY(0,-1,-1)})
	\foreach \i in {1, 1.1, ..., 2.1} {
	to ({PX({sqrt(\i*\i - 1)},-1,-\i)}, {PY({sqrt(\i*\i - 1)},-1,-\i)}) 
	};	
	
\end{scope}

\draw[draw = none, fill = blue!50, opacity = .5]({PX(1,-1,1)}, {PY(1,-1,1)}) to ({PX(3,-1,1)}, {PY(3,-1,1)}) to({PX(3,-1,2)}, {PY(3,-1,2)}) to ({PX(2,-1,2)}, {PY(2,-1,2)}) to cycle;
\draw[blue, thick]({PX(1,-1,1)}, {PY(1,-1,1)}) to ({PX(3,-1,1)}, {PY(3,-1,1)});
\draw[blue, thick]({PX(1,-1,1)}, {PY(1,-1,1)}) to ({PX(2,-1,2)}, {PY(2,-1,2)});

\end{tikzpicture}
\\
(a) $\mbfs{a} = (-1,0)$ and $\mbfs{d} = 0$  & (b) $\mbfs{a} = (0,1)$ and $\mbfs{d} = 0$
\end{tabular}
\caption{The set $H$ is added to Figure~\ref{fig0} for different choices of $(\mbfs{a},\mbfs{d})$.
The axes are oriented to clarify the figures.
The hyperplane $H$ is in gray, and $Q_g$ is in green.
(a) $C^G_{\Gamma} \cap H = C_{\Gamma} \cap H$ (in blue) is maximal $Q_g$-free.
(b) $C_{\Gamma} \cap H$ (in blue) is not maximal $Q_g$-free.
Furthermore, the set $C^G_{\Gamma}$ in (b) is obtained by dropping one inequality, which results in a set that is no longer $Q_g$-free.
}\label{fig2}
\end{figure}
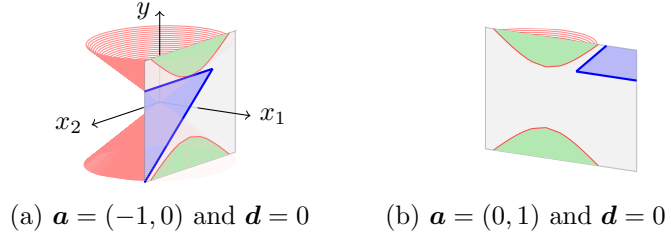

\begin{theorem}\label{thm:hardcase}
Suppose $\|\mbfs{d}\| < \|\mbfs{a}\| = 1$.
Let $\Gamma: D^m \to D^n$ be such that $C_{\Gamma}$ is a full-dimensional maximal $Q_h$-free set.
Then $C^G_{\Gamma} \cap H$ is a full-dimensional maximal $Q_g$-free set if and only if $G = D^m$. 
\end{theorem}

The role of $G$ in Theorems \ref{thm:easycase} and \ref{thm:hardcase} is fundamentally different.
Theorem \ref{thm:hardcase} gives an `all-or-nothing' condition $G = D^m$ if $\|\mbfs{d}\| < \|\mbfs{a}\|=1$.
In contrast, if $\|\mbfs{a}\| \leq \|\mbfs{d}\|=1$, then this all-or-nothing condition may not hold.
For instance, if $\|\mbfs{a}\| < \|\mbfs{d}\|$, then both $G$ and $D^m \setminus G$ are non-empty, as they contain $\mbfs{-d}$ and $\mbfs{d}$, respectively, by the Cauchy-Schwarz Inequality.
Theorem \ref{thm:easycase} implies that for each $\mbfs{\beta}\in D^m \setminus G$, the inequality $-\Gamma(\mbfs{\beta}){}^\top\mbfs{x}+\mbfs{\beta}{}^\top\mbfs{y} \le 0$ can be safely removed from the description of $C_{\Gamma} \cap H$ without destroying the $Q_g$-free property.
Such a removal may not work if $\|\mbfs{d}\| < \|\mbfs{a}\|=1$, as seen in Figure~\ref{fig2}.

As stated earlier, Theorems \ref{thmHQuad},~\ref{thm:easycase}, and~\ref{thm:hardcase} complete a characterization of all maximal $S$-free sets when $S$ is defined by a single quadratic inequality.
They also yield a constructive procedure for generating such sets. Namely, it suffices to construct a non-expansive function $\Gamma$ satisfying $\mbfs{0}\not\in \conv(\{\Gamma(\mbfs{\beta}) \,:\ \mbfs{\beta}\in D^m\})$, together with the case-dependent conditions:
\[
\begin{array}{rl@{\hskip .25cm}l}
 \text{If } \|\mbfs{a}\| \leq \|\mbfs{d}\| = 1,&\text{then}~\mbfs{a} \neq \Gamma(-\mbfs{d}).\\[.15cm]
\text{If } \|\mbfs{d}\| < \|\mbfs{a}\| = 1,&\text{then}~\mbfs{a}{}^\top\Gamma(\mbfs{\beta}) + \mbfs{d}{}^\top\mbfs{\beta} \le 0~ \text{for all}~\mbfs{\beta}\in D^m.
\end{array}
\]

\noindent{\bf Outline.}
The rest of this article is organized as follows.
Section~\ref{sec:Background} provides background results.
Theorem \ref{thm:Both} is then proved in Section~\ref{sec:Both}; our proof simultaneously proves one direction for each of Theorems~\ref{thm:easycase} and~\ref{thm:hardcase}.
In Section~\ref{secMaxH}, we tailor the maximality criterion~\cite[Theorem 1.4]{munoz2025characterization} from Mu\~noz et al.\ to the $Q_g$-free setting; we use this to prove the remaining directions of Theorems~\ref{thm:easycase} and~\ref{thm:hardcase}.
The remaining direction of Theorem~\ref{thm:easycase} and Theorem~\ref{thm:hardcase} are proved in Section~\ref{sec:Maximality_easycase} and Section~\ref{sec:Maximality_hardcase}, respectively.
%

\noindent{\bf Related literature.}
%
Outside of intersection cuts, $S$-free sets have been used as optimality certificates in integer optimization~\cite{averkov2013maximal,BOW2016,BCCWW2017,conforti2016maximal,PSS2022}.
Intersection cuts were introduced by Balas~\cite{B1971} when $S$ is a lattice and by Tuy~\cite{T1964} when $S$ is a reverse convex set.
Since then, intersection cuts have been well studied in several areas of optimization:
mixed-integer linear and conic optimization~\cite{AJ2013,ALWW2007,ABP2018,BDP2019,CCDLM2014},
mixed-integer nonlinear programming~\cite{fischetti2020branch,MKV2016,serrano2019intersection},
rank-one and polynomial optimization~\cite{BCM2019,BCM2020},
bilevel optimization~\cite{fischetti2016intersection},
non-convex quadratic optimization~\cite{chmiela2022implementation,chmiela2025monoidal,MS2020,MS2022},
signomial programming~\cite{xu2025cutting},
and submodular maximization~\cite{xu2025submodular}.

The first maximal quadratic-free sets were constructed in \cite{MS2020,MS2022}.  
Later, their corresponding intersection cut implementation, and their strengthening were studied in \cite{chmiela2022implementation,chmiela2025monoidal}.
These works consider one family of maximal quadratic-free sets, and they have produced positive
empirical results in optimization instances with non-convex quadratic constraints.

Complete characterizations of maximal $S$-free sets are known in only a few settings.
For the case when $S$ is a lattice,  Lov\'{a}sz~\cite{L1989} showed that full-dimensional maximal $S$-free sets are polyhedra with integer points in the relative interior of each facet; see also~\cite{A2013,BCCZ2010,basu2010minimal}. 
Mu\~noz et al.~\cite[Theorem 1.4]{munoz2025characterization} characterized maximality for general $S$, but this level of generality can be difficult to interpret for specific settings.
As previously stated, Mu\~noz et al.~\cite{munoz2025characterization} characterized maximal $Q_h$-free sets.
%

\noindent{\bf Preliminaries and notation.}
%
Throughout the paper, we fix $(\mbfs{a}, \mbfs{d}) \in \mbb{R}^n \times \mbb{R}^m$ with $n,m\geq 1$ and $\max\{\|\mbfs{a}\|, \|\mbfs{d}\|\} = 1$.
After fixing $(\mbfs{a}, \mbfs{d})$, we have consequently fixed $Q_g$ and $H$.
Since the domain and range of all functions $\Gamma$ we consider are unit spheres, we can equivalently define $\Gamma:D^m \to D^n$ to be non-expansive if $\mbfs{\beta}{}^\top\overline{\mbfs{\beta}} \le \Gamma(\mbfs{\beta}){}^\top \Gamma(\overline{\mbfs{\beta}}) $, for all $\mbfs{\beta}, \overline{\mbfs{\beta}} \in D^m$.
Whenever we refer to a set $G$, we refer to a set of the form \eqref{defnG}; we leave the dependence on $\Gamma, \mbfs{a}$, and $\mbfs{d}$ implicit to alleviate notation.

For background on convexity, see~\cite{R1970}.
We use $\cl(\cdot)$, $\inte(\cdot)$, $\relint(\cdot)$, $\conv(\cdot)$, and $\cone(\cdot)$ to denote the closure, interior, the relative interior, the convex hull, and convex conic hull, respectively.
We use a capital Greek letter, e.g., $\Gamma, \Lambda$, or $\Psi$, for a function whose domain is $D^m$.

\section{Background results}\label{sec:Background}

In this section, we collect known background results for completeness.
%

\subsection{Reduction from quadratic-free sets to $Q_h$- and $Q_g$-free sets}\label{sec:Reduction}

In this subsection, we reduce the characterization of maximal quadratic-free sets to maximal $Q_h$- and $Q_g$-free sets.
This discussion also appears in \cite{MS2022}.

Consider a set of the form 
\[
S := \left\{ \mbfs{s} \in \mbb{R}^{d}  :\ \mbfs{s}{}^\top \mbfs{A} \mbfs{s} +  \mbfs{b}{}^\top \mbfs{s} +  c \le0\right\},
\]
where $\mbfs{A} \in \mbb{R}^{d\times d}$ is symmetric, $\mbfs{b} \in \mbb{R}^d$ and $c \in \mbb{R}$.
If $\mbfs{A}$ is positive semi-definite, then $S$ is convex.
In this case, full-dimensional maximal $S$-free sets are half-spaces that define supporting hyperplanes of $S$.
If $\mbfs{A}$ is negative semi-definite, then $S$ is reverse convex, i.e., $\mbb{R}^d \setminus S$ is convex.
In this case, $\cl(\mathbb{R}^d \setminus S)$ is the unique full-dimensional maximal $S$-free set. 

Assume $\mbfs{A}$ is indefinite.
If the quadratic form is homogeneous, i.e., 
\[
S = \left\{ \mbfs{s} \in \mbb{R}^{d}  :\ \mbfs{s}{}^\top \mbfs{A} \mbfs{s}  \le0\right\},
\]
then we can diagonalize $\mbfs{A}$ to map $S$ to a set of the form $\{(\mbfs{x},\mbfs{y},\mbfs{z}) \in \mathbb{R}^{n}\times \mbb{R}^m\times \mbb{R}^{\ell} : (\mbfs{x},\mbfs{y}) \in Q_h\}$.
The $\mbfs{x}$ (respectively, $\mbfs{y}$ and $\mbfs{z}$) variables come from the positive (respectively, negative and zero) eigenvalues of $\mbfs{A}$.
Both $\mbfs{x}$ and $\mbfs{y}$ variables exist because $\mbfs{A}$ is indefinite.
The $\mbfs{z}$ variables are free, so after projecting them out, we deduce that maximal $Q_h$-free sets capture all maximal $S$-free sets.

If the quadratic form defining $S$ is inhomogeneous, we consider a homogenization of $S$: 
\[
S' := \left\{ (\mbfs{s},t) \in \mbb{R}^{d}\times \mbb{R}  :\ \mbfs{s}{}^\top \mbfs{A} \mbfs{s} +  (\mbfs{b}{}^\top \mbfs{s})\cdot t +  c\cdot t^2 \le0,\, t=1\right\}.
\]
There is a one-to-one correspondence between maximal $S'$-free sets contained in the hyperplane $ \{ (\mbfs{s},t) \in \mbb{R}^{d}\times \mbb{R}  :\, t=1\}$ and maximal $S$-free sets.
Then, we can proceed similarly to the homogeneous case. Via a diagonalization, map $S'$ to a set the form
\[
Q_g' := \{(\mbfs{x},\mbfs{y},\mbfs{z}) \in \mathbb{R}^{n}\times \mbb{R}^m\times \mbb{R}^{\ell} : \|\mbfs{x}\| \le \|\mbfs{y}\|,\; \mbfs{a}^\top \mbfs{x} + \mbfs{d}^\top \mbfs{y} + \mbfs{h}^\top \mbfs{z} = -1\},
\]
where $(\mbfs{a}, \mbfs{d}, \mbfs{h}) \in \mathbb{R}^{n}\times \mbb{R}^m\times \mbb{R}^{\ell}$.
Here, the $\mbfs{x}$ (respectively, $\mbfs{y}$ and $\mbfs{z}$) variables come from the positive (respectively, negative and zero) eigenvalues in the quadratic form defining $S'$.
This transformation is linear and one-to-one.
Thus, maximal $Q_g'$-free sets capture all maximal $S$-free sets in this case.

The next step is to remove the $\mbfs{z}$ variables. 
Mu\~noz and Serrano \cite[Remark 3.2]{MS2022} argue that if $\mbfs{h}\neq \mbfs{0}$, then all maximal $Q_g'$-free sets are of the form $C\times \mathbb{R}^\ell$, where $C$ is a maximal $Q_h$-free set.
Thus, we can further reduce to the case where $\mbfs{h}=\mbfs{0}$ in $Q_g'$, i.e., to $Q_g$.

The final step is to ensure $\max\{\|\mbfs{a}\|, \|\mbfs{d}\|\} = 1$.
The set $S$ (and thus $Q_g$) is not empty because $\mbfs{A}$ is indefinite. 
Hence, $\max\{\|\mbfs{a}\|, \|\mbfs{d}\|\} > 0$.
We can achieve $\max\{\|\mbfs{a}\|, \|\mbfs{d}\|\} = 1$ by rescaling the variables.

\subsection{Basic properties of full dimensional maximal $Q_g$-free sets}\label{sec:Structure}

Theorem~\ref{thm:Both} states maximal $Q_g$-free sets have the form $C^G_{\Gamma} \cap H$.
The next result is a simpler variation of Theorem~\ref{thm:Both}.
The result follows from Zorn's Lemma, using the proof strategy of~\cite[Theorem 4.4]{CCDLM2014}.
We use this simpler variation as a stepping stone to prove Theorem~\ref{thm:Both} later. 

\begin{lemma}\label{sliceofCgamma}
If $K\subseteq H$ is a full-dimensional maximal $Q_g$-free set, then there exists a function $\Gamma:D^m\to D^n$ such that $C_{\Gamma}$ is a full-dimensional maximal $Q_h$-free set and $K = C_{\Gamma}\cap H$.
\end{lemma}
\begin{proof}
The affine hull of $K$ does not contain the origin, so $\cone(K) \subseteq \mbb{R}^n \times \mbb{R}^m$ is $(n+m)$-dimensional.
Furthermore, if $(\mbfs{x},\mbfs{y})\in \inte(\cone(K))$, then $\lambda \cdot (\mbfs{x},\mbfs{y}) \in \relint(K)$ for some $\lambda > 0$.
As $K$ is $Q_g$-free, it follows that $\lambda \cdot (\mbfs{x},\mbfs{y}) \not \in Q_g$.
Hence, $\| \mbfs{x}\| > \|\mbfs{y}\|$, implying that $(\mbfs{x},\mbfs{y}) \not \in Q_h$.
Thus, $\cone(K)$ is $Q_h$-free.

Consider the family $\mcf{F}$ of $Q_h$-free sets containing $\cone(K)$, along with the partial-order defined by set containment.
Let us argue that each chain $(P_i)_{i=1}^{\infty}$ in $\mcf{F}$ has an upper bound in $\mcf{F}$, namely, $\cup_{i=1}^{\infty} P_i$.
If $\inte( \cup_{i=1}^{\infty} P_i) \cap Q_h \neq \emptyset$, then $\inte(P_i) \cap Q_h \neq \emptyset$ for some $i$.
The set $\cup_{i=1}^{\infty} P_i$ is convex because it is a nested union of convex sets. 
Hence, $ \cup_{i=1}^{\infty} P_i$ is $Q_h$-free.
By Zorn's Lemma, there is a maximal element in $\mcf{F}$.
By Theorem~\ref{thmHQuad}, this maximal element has the form $C_\Gamma$ for some non-expansive function $\Gamma : D^m \to D^n$.
 Thus, 
 \(
 K  = \cone(K) \cap H \subseteq C_\Gamma \cap H.
 \)
Given that $C_\Gamma $ is $Q_h$-free, the set $C_\Gamma \cap H$ is $Q_g$-free.
As $K$ is a maximal $Q_g$-free set, we conclude $K = C_\Gamma \cap H$.
\end{proof}

Our main theorems consider full-dimensional sets in $H$.
The following is a sufficient condition for $C_\Gamma\cap H$ to be full-dimensional

\begin{lemma}\label{lem:fulldim}
Let $\Gamma:D^m \to D^n$ be such that $C_\Gamma$ is full-dimensional. 
If $\mbfs{a}{}^\top \overline{\mbfs{x}} +\mbfs{d}{}^\top\overline{\mbfs{y}}< 0$ for some $(\overline{\mbfs{x}}, \overline{\mbfs{y}}) \in C_{\Gamma}$, then $C_{\Gamma}\cap H$ is full-dimensional.
\end{lemma}

\begin{proof}
As $C_{\Gamma}$ is full-dimensional, there exists some $(\widehat{\mbfs{x}}, \widehat{\mbfs{y}}) \in \inte(C_{\Gamma})$.
Given that $(\overline{\mbfs{x}}, \overline{\mbfs{y}}) \in C_{\Gamma}$, the convex combination $(\mbfs{x}_{\varepsilon}, \mbfs{y}_{\varepsilon} ):= (1-\varepsilon) \cdot (\overline{\mbfs{x}}, \overline{\mbfs{y}}) + \varepsilon\cdot(\widehat{\mbfs{x}}, \widehat{\mbfs{y}}) $ is in $\inte(C_{\Gamma})$ for all $\varepsilon \in (0,1]$.
The inequality $\mbfs{a}{}^\top\overline{\mbfs{x}} + \mbfs{d}{}^\top  \overline{\mbfs{y}} < 0$ is strict, so we can choose $\varepsilon \in (0,1)$ small enough such that $\mbfs{a}{}^\top \mbfs{x}_{\varepsilon} + \mbfs{d}{}^\top \mbfs{y}_{\varepsilon} < 0$. 
Hence, there exists an $(n+m)$-dimensional ball $B \subseteq C_{\Gamma}$ (centered at some $(\mbfs{x}_{\varepsilon}, \mbfs{y}_{\varepsilon} )$) such that $\mbfs{a}{}^\top\mbfs{x} + \mbfs{d}{}^\top\mbfs{y}< 0$ for all $(\mbfs{x}, \mbfs{y}) \in B$.
Each vector in $B$ can be scaled by a positive number to be in $H$.
Thus, $C_{\Gamma}\cap H$ is full-dimensional.
\end{proof}

\subsection{Valid inequalities and tilting}\label{sec:Tiliting}

Our main theorems `tilt' valid inequalities for $C_{\Gamma} \cap H$ so that they become valid for $C_{\Gamma}$.
This is important since, even though $C_{\Gamma} \cap H$ can be described using only inequalities that are valid for $C_{\Gamma}$ (plus the hyperplane $H$) it will be important to handle \emph{arbitrary} valid inequalities $C_{\Gamma} \cap H$.
Figure \ref{fig:tilting} illustrates this tilting.
This subsection contains results about tilting valid inequalities.  

\begin{figure}
\centering
\begin{tabular}{c@{\hskip 2 cm}c}
\begin{tikzpicture}[scale = .6,every node/.style={font=\normalsize}]

\begin{scope}
\clip(-2,-2) rectangle (2,2);

\draw[red!70](-3,-3) to (3,3) to (-3,3) to (3, -3) to cycle;
\draw[draw = none, fill = red!10, opacity = .5] (0,0) to (3,3) to (-3,3) to cycle;
\draw[draw = none, fill = red!10, opacity = .5] (0,0) to (3,-3) to (-3,-3) to cycle;

\draw[draw = blue!75, fill = blue!75, opacity = .15] (0,0) to (3,3) to (3,-3) to cycle;

\draw[draw =black, fill = black, opacity = .35, dashed, ultra thick] (-3,1) to (3,1);
\draw[draw =black, opacity = 1, line width = .5 mm] (1,3) to (1,-3);
\end{scope}
\draw[<->,black!30](0,-2) to node[above, pos = 1]{\color{black}$\mbfs{y}$}(0,2);
\draw[<->,black!30](-2,0)  to node[right, pos = 1]{\color{black}$\mbfs{x}$} (2,0);

\end{tikzpicture}
&
\begin{tikzpicture}[scale = .6,every node/.style={font=\normalsize}]

\begin{scope}
\clip(-2,-2) rectangle (2,2);

\draw[red!70](-3,-3) to (3,3) to (-3,3) to (3, -3) to cycle;
\draw[draw = none, fill = red!10, opacity = .5] (0,0) to (3,3) to (-3,3) to cycle;
\draw[draw = none, fill = red!10, opacity = .5] (0,0) to (3,-3) to (-3,-3) to cycle;

\draw[draw = blue!75, fill = blue!75, opacity = .15] (0,0) to (3,3) to (3,-3) to cycle;

\draw[draw =black, fill = black, opacity = .35, dashed, ultra thick] (-3,1) to (3,1);
\draw[draw =black, opacity = 1, line width = .5 mm] (3,3) to (-3,-3);
\end{scope}
\draw[<->,black!30](0,-2) to node[above, pos = 1]{\color{black}$\mbfs{y}$}(0,2);
\draw[<->,black!30](-2,0)  to node[right, pos = 1]{\color{black}$\mbfs{x}$} (2,0);

\end{tikzpicture}\\[.25cm]
(a) & (b)
\end{tabular}
\caption{In this figure, $n = m = 1$. 
The set $Q_h$ is in red, and $C_{\Gamma}$ is in blue for $\Gamma(-1) = \Gamma(1) = 1$.
The hyperplane $H$ for $\mbfs{a} = 1$ and $\mbfs{d} = 0$ is shown as a dashed line. 
In (a), the black inequality is valid for $C_{\Gamma} \cap H$ but not $C_{\Gamma}$.
In (b), the black inequality is tilted such that it is valid for both $C_{\Gamma} \cap H$ and $C_{\Gamma}$.
}\label{fig:tilting}
\end{figure}
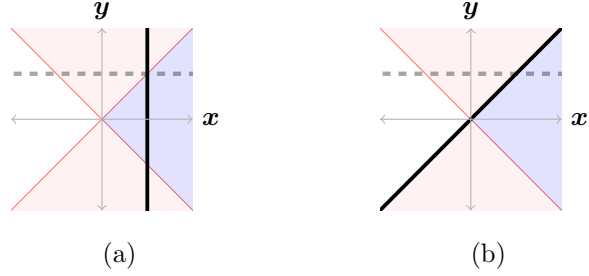

\begin{lemma}\label{lem:Wavy2}
Let $\mbfs{f} \in \mbb{R}^d\setminus \{\mbfs{0}\}$, and set $F := \{\mbfs{x} \in \mbb{R}^d:\ \mbfs{f}{}^\top\mbfs{x} = -1\}$.
Let $K \subseteq \mbb{R}^d$ be a closed convex cone such that $K \cap F \neq \emptyset$.
Let $\mbfs{e}{}^\top\mbfs{x} \le \tau$ be valid for $K \cap F$, where $(\mbfs{e}, \tau) \in \mbb{R}^d \times \mbb{R}$, and $\mbfs{e}$ and $\mbfs{f}$ are linearly independent. 
Then, there exists a number $\gamma \leq \tau $ and a vector $\mbfs{g} \in \mbb{R}^d \setminus \{\mbfs{0}\}$ such that $\mbfs{g}{}^\top\mbfs{x}\leq 0$ is valid for $K$ and 
\(
\mbfs{e} =  \mbfs{g}- \gamma \cdot \mbfs{f}.
\)
\end{lemma}

\begin{proof}
Define the number
\[
\gamma:= \inf \left\{\chi \in \mbb{R} :\ \text{$\mbfs{e}{}^\top\mbfs{x}  \le \chi$ is valid for $ K \cap F$}\right\}. 
\]
Note that  $\gamma\le \tau$ and $\gamma$ is finite because $ K \cap F\neq \emptyset$.

Define $\mbfs{g} := \mbfs{e} + \gamma \cdot \mbfs{f}$.
Note that $\mbfs{g} \neq \mbfs{0}$  because $\mbfs{e}$ and $\mbfs{f}$ are linearly independent.
It remains to show that $\mbfs{g}{}^\top\mbfs{x} \leq 0$ is valid for $K$.
Let $\overline{\mbfs{x}} \in K$.
We consider three cases based on the sign of $\mbfs{f}^\top\overline{\mbfs{x}}$.
If we multiply $\overline{\mbfs{x}}$ by a positive number, then the sign of $\mbfs{f}^\top\overline{\mbfs{x}}$ is preserved, as is containment in $K$ because it is a cone.
Thus, we may assume $\mbfs{f}^\top\overline{\mbfs{x}} \in \{\pm 1, 0\}$. 

In case 1, assume $\mbfs{f}{}^\top\overline{\mbfs{x}}  = -1$.
Here, we have $\overline{\mbfs{x}} \in K \cap F$.
Hence, $\mbfs{e}{}^\top\overline{\mbfs{x}} \le \gamma$. 
Thus,
\(
\mbfs{g}{}^\top\overline{\mbfs{x}} = \mbfs{e}{}^\top\overline{\mbfs{x}} + \gamma \cdot \mbfs{f}{}^\top\overline{\mbfs{x}}\le \gamma - \gamma = 0.
\)

In case 2, assume $\mbfs{f}^\top\overline{\mbfs{x}}  = 0$.
Let $\widehat{\mbfs{x}} \in K \cap F$. 
Given that $K$ is a cone, $\widehat{\mbfs{x}} + \lambda \cdot \overline{\mbfs{x}} \in K$ for all $\lambda \ge 0$.
Moreover, because $\mbfs{f}^\top\overline{\mbfs{x}}  = 0$, we have  $\widehat{\mbfs{x}} + \lambda \cdot \overline{\mbfs{x}} \in F$ for each $\lambda \ge 0$.
If $\mbfs{g}{}^\top\overline{\mbfs{x}}  > 0$, then $\mbfs{g}{}^\top (\widehat{\mbfs{x}} + \lambda \cdot \overline{\mbfs{x}}) > \tau - \gamma$ for large enough $\lambda$. However, this implies
\[
\mbfs{e}{}^\top(\widehat{\mbfs{x}} + \lambda \cdot \overline{\mbfs{x}})= (\mbfs{g} - \gamma \mbfs{f} )^\top(\widehat{\mbfs{x}} + \lambda \cdot \overline{\mbfs{x}}) > \tau,
\]
which contradicts that $\mbfs{e}{}^\top\mbfs{x} \le \tau$ is valid for $K \cap F$.
Thus, $\mbfs{g}{}^\top\overline{\mbfs{x}}  \leq 0$.

In case 3, assume $\mbfs{f}^\top\overline{\mbfs{x}}  = 1$.
Assume to the contrary that $\varepsilon := \mbfs{g}{}^\top\overline{\mbfs{x}} > 0$.
By the infimum definition of $\gamma$, we can choose $\widehat{\mbfs{x}}  \in K \cap F$ such that $\mbfs{e}{}^\top\widehat{\mbfs{x}} > \gamma - \varepsilon$.
The point $\overline{\mbfs{x}}+\widehat{\mbfs{x}}$ is in $K$ because $K$ is a convex cone. 
Also, $\mbfs{f}{}^\top(\overline{\mbfs{x}}+\widehat{\mbfs{x}}) = 0$ and 
\[
\mbfs{g}{}^\top(\overline{\mbfs{x}}+\widehat{\mbfs{x}}) = \varepsilon  + (\mbfs{e}{}^\top\widehat{\mbfs{x}} + \gamma \cdot \mbfs{f}{}^\top\widehat{\mbfs{x}}) > \varepsilon + (\gamma- \varepsilon) - \gamma = 0.
\]
However, this means that $\overline{\mbfs{x}}+\widehat{\mbfs{x}} $ contradicts case 2.
\end{proof}

The following result can be found in Rockafellar~\cite{R1970}, albeit slightly rephrased.

 \begin{theorem}[Theorem 17.3 from~\cite{R1970}]\label{Rockafellarpatched}
Let $I^* \subseteq \mathbb{R}^{d}\times \mbb{R}$ be a non-empty compact set of vectors, and set 
\[
B := \left\{\mbfs{x} \in \mbb{R}^d\,:\, \mbfs{x}^*{}^\top \mbfs{x} \leq \mu^* \quad \forall\ (\mbfs{x}^*,\mu^*)\in I^*\right\}.
\]
Suppose $B$ is full-dimensional and $(\mbfs{0}, 0)\not\in I^*$.
Then, for a vector $(\overline{\mbfs{x}},\mu_{\overline{\mbfs{\beta}}})$, where $\overline{\mbfs{x}}\neq \mbfs{0}$, the inequality $\overline{\mbfs{x}}{}^\top \mbfs{x} \leq \mu_{\overline{\mbfs{\beta}}}$ is valid for $B$ if and only if there exist vectors $(\mbfs{x}^*_1, \mu^*_1), \dotsc, (\mbfs{x}^*_d, \mu^*_d)\in I^*$ and numbers $\lambda_1, \dotsc, \lambda_d \geq 0$ such that
\[
\overline{\mbfs{x}} = \sum_{i=1}^d \lambda_i \mbfs{x}^*_i\quad\text{and}\quad \sum_{i=1}^d \lambda_i \mu^*_i \leq \mu_{\overline{\mbfs{\beta}}}.
\]
\end{theorem}

We will apply Theorem~\ref{Rockafellarpatched} to valid inequalities for $C_{\Gamma}$. 
The following lemma will be useful for this.

\begin{lemma}\label{lemma:lambdabounded}
Let $\Gamma:D^m \to D^n$ be such that $C_\Gamma$ is full-dimensional.
There exists a number $\tau = \tau(\Gamma) > 0$ such that the following holds:
For each closed set $J \subseteq D^m$ and each vector $(\mbfs{g}, \mbfs{b}) \in \mbb{R}^n\times \mbb{R}^m$ such that $\|(\mbfs{g}, \mbfs{b})\|  = \sqrt{2}$ and $-\mbfs{g}{}^\top \mbfs{x}+\mbfs{b}{}^\top\mbfs{y} \le 0$ is valid for $C^J_{\Gamma}$, we can write 
\[
(-\mbfs{g}, \mbfs{b}) = \sum_{i=1}^{n+m} \lambda_i \cdot (-\Gamma(\mbfs{\beta}_i), \mbfs{\beta}_i),
\]
where, for each $i \in \{1, \dotsc, n+m\}$, we have $\mbfs{\beta}_i \in D^m$ and $0 \le \lambda_i \le \tau$.
\end{lemma}
\begin{proof}
By Theorem~\ref{thmHQuad}, we have $(\mbfs{0}, \mbfs{0})\not\in \conv(\{(\Gamma(\mbfs{\beta}),\mbfs{\beta}) \,:\ \mbfs{\beta}\in D^m\})$.
The set $\conv(\{(\Gamma(\mbfs{\beta}),\mbfs{\beta}) \,:\ \mbfs{\beta}\in D^m\})$ is closed because $D^m$ is compact and $\Gamma$ is continuous (because it is non-expansive).
Thus, we can separate $(\mbfs{0}, \mbfs{0})$ strictly, i.e., there exists $(\mbfs{f},\mbfs{h})\in \mathbb{R}^{n}\times \mathbb{R}^m$ and $w>0$ such that
\begin{equation}\label{separation}
\mbfs{f}{}^\top \Gamma(\mbfs{\beta}) + \mbfs{h}{}^\top \mbfs{\beta} \geq w \qquad \forall~ \mbfs{\beta}\in D^m.
\end{equation}
Set $\tau := \sqrt{2} \cdot \|(\mbfs{f},\mbfs{h})\|/w$. Note that $\tau$ only depends on $\Gamma$.

Let $J \subseteq D^m$ be closed, and let $(\mbfs{g}, \mbfs{b}) \in \mbb{R}^n\times \mbb{R}^m$ satisfy $\|(\mbfs{g}, \mbfs{b})\|  = \sqrt{2}$ and $-\mbfs{g}{}^\top \mbfs{x}+\mbfs{b}{}^\top\mbfs{y} \le 0$ is valid for $C^J_{\Gamma}$.
The set $J$ is compact and does not contain $\mbfs{0}$.
Thus, by Theorem~\ref{Rockafellarpatched}, there exist numbers $\lambda_1, \dotsc, \lambda_{n+m} \ge 0$ and vectors $\mbfs{\beta}_1, \dotsc, \mbfs{\beta}_{n+m} \in J$ such that
\[
(-\mbfs{g}, \mbfs{b}) = \sum_{i=1}^{n+m} \lambda_i \cdot (-\Gamma(\mbfs{\beta}_i), \mbfs{\beta}_i).
\]

For each $i \in \{1, \dotsc, n+m\}$, multiply Inequality~\eqref{separation} associated with $\mbfs{\beta}_{i}$ by $\lambda_{i} $.
Adding the result yields
\[
\mbfs{f}{}^\top\mbfs{g}+\mbfs{h}{}^\top\mbfs{b} = \sum_{i=1}^{n+m} \lambda_{i} \mbfs{f}^\top \Gamma(\mbfs{\beta}_{i}) + \sum_{i=1}^{n+m} \lambda_{i} \mbfs{h}^\top\mbfs{\beta}_{i} \geq w \cdot \sum_{i=1}^{n+m} \lambda_{i}.
\]
Using $w >0$ and $\|(-\mbfs{g}, \mbfs{b})\|  = \sqrt{2}$, we see that $\sum_{i=1}^{n+m} \lambda_{i} \leq \tau$.
Thus, each $\lambda_i$ satisfies $0 \le \lambda_i \le \tau$.
\end{proof}

\section{A proof of Theorem~\ref{thm:Both}}\label{sec:Both}

We prove Theorem~\ref{thm:Both} using two cases: $\|\mbfs{a}\| \le \|\mbfs{d}\|\ = 1$ and $\|\mbfs{d}\| < \|\mbfs{a}\|\ = 1$.
These results are found in Lemmata \ref{lemma:necessary_easy} and \ref{lemma:hardcase_necessary}, respectively.
In the two cases, we actually prove a stronger version of Theorem~\ref{thm:Both} that allows us to prove the necessary condition for maximality in Theorems~\ref{thm:easycase} and~\ref{thm:hardcase}.
%

\subsection{A proof of Theorem~\ref{thm:Both} when \texorpdfstring{$\|\mbfs{a}\| \le \|\mbfs{d}\|\ = 1$}{}.}\label{sec:EasyNecessary}

We begin with the case when $\|\mbfs{a}\| \le \|\mbfs{d}\|\ = 1$.
Our proof in this case uses the following subset of $G$:
\begin{equation}\label{defnGcirc}
G^{\circ}: =\left\{\mbfs{\beta}  \in D^m\,:\, \mbfs{a}{}^\top \Gamma(\mbfs{\beta}) + \mbfs{d}{}^\top \mbfs{\beta} < 0\right\}.
\end{equation}
Perhaps counterintuitively, $\cl(G^\circ)$ might not equal $G$; see
Figure \ref{example:easycase}.

\begin{figure}\centering

\begin{tabular}{@{\hskip 0 cm}c@{\hskip .1cm}c}
\begin{tikzpicture}[scale = .55,every node/.style={font=\small}]

\draw[black!50](0,0) to (0,{pi});
\draw[black!50](0,0) to ({2*pi},0);
\draw[blue!50, line width = .5mm](0,0) to ({pi/2},{pi}) to ({pi},0) to ({3*pi/2},{pi}) to ({2*pi},0);

\draw[black!50](0, 0) to node[pos = 1, below]{$\mbfs{\beta}_{1}$} (0,-.1) ; 
\draw[black!50]({2*pi}, 0) to node[pos = 1, below]{$\mbfs{\beta}_{1}$} ({2*pi},-.1) ; 
\draw[black!50]({pi/2}, 0) to node[pos = 1, below]{$\mbfs{\beta}_{2}$} ({pi/2},-.1);
\draw[black!50]({pi}, 0) to node[pos = 1, below]{$\mbfs{\beta}_{3}$}({pi},-.1);
\draw[black!50]({3*pi/2}, 0) to node[pos = 1, below]{$\mbfs{\beta}_{4}$} ({3*pi/2},-.1);

\draw[black!50](0,0) to node[pos = 1, left]{$\Gamma(\mbfs{\beta}_{1})$}(-.1,0);
\draw[black!50](0, {pi}) to node[pos = 1, left]{$\Gamma(\mbfs{\beta}_{2})$} (-.1,{pi});
\end{tikzpicture}
&
\begin{tikzpicture}[scale = .55,every node/.style={font=\small}]

\draw[black!50](0,{-6/5}) to (0,{8/5});
\draw[black!50](0,0) to ({2*pi},0);

\begin{scope}[domain = 0:{pi/2}] 
\draw[blue!50, line width = .5mm] plot(\x, 0);
\end{scope}

\begin{scope}[domain = {pi/2}:{pi}] 
\draw[blue!50, line width = .5mm] plot(\x, {((2/pi)*\x-1)/sqrt(((2/pi)*\x-1)^2+(2 - (2/pi)*\x)^2)*(-6/5)});
\end{scope}

\begin{scope}[domain = {pi}:{3*pi/2}] 
\draw[blue!50, line width = .5mm] plot(\x, {(3 - (2/pi) * \x )/sqrt( (3 - (2/pi) * \x )^2+ ((2/pi) * \x - 2)^2 ) * (-6/5)+((2/pi) * \x - 2)/sqrt( (3 - (2/pi) * \x )^2+ ((2/pi) * \x - 2)^2 ) * (8/5)});
\end{scope}

\begin{scope}[domain = {3*pi/2}:{2*pi}] 
\draw[blue!50, line width = .5mm] plot(\x, {(4 - \x*2/pi)/sqrt((4 - \x*2/pi)^2+ (\x*(2/pi)-3)^2 ) * (8/5)});
\end{scope}

\draw[black!50](0, 0) to node[pos = 1, below right]{$\mbfs{\beta}_{1}$} (0,-.1) ; 
\draw[black!50]({2*pi}, 0) to node[pos = 1, below]{$\mbfs{\beta}_{1}$} ({2*pi},-.1) ; 
\draw[black!50]({pi/2}, 0) to node[pos = 1, below]{$\mbfs{\beta}_{2}$} ({pi/2},-.1);
\draw[black!50]({pi}, 0) to node[pos = 1, below]{$\mbfs{\beta}_{3}$}({pi},-.1);
\draw[black!50]({3.78509}, 0) to node[pos = 1, below]{$\mbfs{\beta}^{*}$}({3.78509},-.75);
\draw[black!50]({3*pi/2}, 0) to node[pos = 1, below]{$\mbfs{\beta}_{4}$} ({3*pi/2},-.1);

\draw[black!50](0,8/5) to node[pos = 1, left]{$\phantom{-}\sfrac{8}{5}$}(-.1,8/5);
\draw[black!50](0, -6/5) to node[pos = 1, left]{$-\sfrac{6}{5}$} (-.1,-6/5);

\end{tikzpicture}
\\[.1cm]
\begin{tabular}{r@{\hskip .5 cm}l}
(a) & A polar plot of $\Gamma(\mbfs{\beta}) := |\mbfs{\beta}|$.
\end{tabular}
&
\begin{tabular}{r@{\hskip .1 cm}l}
 (b) & A polar plot of $\mbfs{a}{}^\top\Gamma(\mbfs{\beta})+\mbfs{a}{}^\top\mbfs{\beta}$.
\end{tabular}\\[.45cm]
\multicolumn{2}{c}{
\def\startx{2} 
\def\endx{.8}
\def\starty{-1}
\def\endy{4}
\def\startz{-2.5}
\def\endz{1}
\def\dip{-1.1}
\pgfmathsetmacro{\ronebound}{\endy}
\pgfmathsetmacro{\rtwobd}{(-5/6 + \startx)/4}
\pgfmathsetmacro{\rthreebd}{(-5/6 + \startx)/3}

\begin{tikzpicture}[scale = 4,every node/.style={font=\small},
declare function={
PX(\x,\y,\z) = -.16*\x + .1*\y+ 0*\z;
PY(\x,\y,\z) = -.09*\x -.01*\y+ .15*\z;
f(\x,\y) = (1/25)*(9*\x-4*sqrt(16*\x*\x+6*\x*(4*\y+5)+9*\y*\y-40*\y-25)-12*\y - 15);
g(\x,\y) = (1/25)*(9*\x+4*sqrt(16*\x*\x+6*\x*(4*\y+5)+9*\y*\y-40*\y-25)-12*\y - 15);
h(\x,\z) = (-7*\x*\x-18*\x*\z-30*\x+25*\z*\z+30*\z+25)/(8*(3*\x-3*\z-5));
c(\y) = (1/16)*(-12*\y +5*sqrt(5)*sqrt(8*\y+5)-15);
}
]

\draw[draw = none, thick] ({PX(\startx,\starty,\startz)}, {PY(\startx,\starty,\startz)}) to ({PX(\startx,\starty,\endz)}, {PY(\startx,\starty,\endz)}) to ({PX(\endx,\starty,\endz)}, {PY(\endx,\starty,\endz)}) to ({PX(\endx,\endy,\endz)}, {PY(\endx,\endy,\endz)}) to ({PX(\endx,\endy,\startz)}, {PY(\endx,\endy,\startz))}) to ({PX(\startx,\endy,\startz)}, {PY(\startx,\endy,\startz)}) to cycle;

\draw[draw = none, thick] ({PX(\startx,\starty,\startz)}, {PY(\startx,\starty,\startz)}) to ({PX(\startx,\starty,\endz)}, {PY(\startx,\starty,\endz)}) to ({PX(\endx,\starty,\endz)}, {PY(\endx,\starty,\endz)}) to ({PX(\endx,\endy,\endz)}, {PY(\endx,\endy,\endz)}) to ({PX(\startx,\endy,\endz)}, {PY(\startx,\endy,\endz))}) to ({PX(\startx,\endy,\startz)}, {PY(\startx,\endy,\startz)}) to cycle;

\draw[black!15] ({PX(\startx,\starty,\startz)}, {PY(\startx,\starty,\startz)}) to ({PX(\startx,\starty,\endz)}, {PY(\startx,\starty,\endz)}) to ({PX(\startx,\endy,\endz)}, {PY(\startx,\endy,\endz)}) to ({PX(\startx,\endy,\startz)}, {PY(\startx,\endy,\startz)}) to node[pos = .5, below]{\color{black!50}$x_2$}cycle;
\draw[black!15] ({PX(\endx,\starty,\startz)}, {PY(\endx,\starty,\startz)}) to ({PX(\endx,\starty,\endz)}, {PY(\endx,\starty,\endz)}) to ({PX(\endx,\endy,\endz)}, {PY(\endx,\endy,\endz)}) to ({PX(\endx,\endy,\startz)}, {PY(\endx,\endy,\startz)}) to cycle;
\draw[black!15] ({PX(\startx,\starty,\startz)}, {PY(\startx,\starty,\startz)}) to node[pos = .5, left]{\color{black!50}$y_1$} ({PX(\startx,\starty,\endz)}, {PY(\startx,\starty,\endz)}) to ({PX(\endx,\starty,\endz)}, {PY(\endx,\starty,\endz)}) to ({PX(\endx,\starty,\startz)}, {PY(\endx,\starty,\startz)}) to cycle;
\draw[black!15] ({PX(\startx,\endy,\startz)}, {PY(\startx,\endy,\startz)}) to ({PX(\startx,\endy,\endz)}, {PY(\startx,\endy,\endz)}) to ({PX(\endx,\endy,\endz)}, {PY(\endx,\endy,\endz)}) to ({PX(\endx,\endy,\startz)}, {PY(\endx,\endy,\startz)}) to  node[pos = .5, below]{\color{black!50}$x_1$}  cycle;

\begin{scope}
\clip  ({PX(\startx,\starty,\startz)}, {PY(\startx,\starty,\startz)}) to ({PX(\startx,\starty,\endz)}, {PY(\startx,\starty,\endz)}) to ({PX(\endx,\starty,\endz)}, {PY(\endx,\starty,\endz)}) to ({PX(\endx,\endy,\endz)}, {PY(\endx,\endy,\endz)}) to ({PX(\endx,\endy,\startz)}, {PY(\endx,\endy,\startz))}) to ({PX(\startx,\endy,\startz)}, {PY(\startx,\endy,\startz)}) to cycle;

\draw[green!50!black!50, line width = .2mm, left color = green!50!black!30, right color = green!50!black!50, opacity = .35]
	plot[domain = \startx:\endx, variable = \x] ({PX(\x, \endy, g(\x, \endy))},{PY(\x, \endy, g(\x, \endy))})
	to ({PX(\endx, \endy, \endz)},{PY(\endx, \endy, \endz)})
	to ({PX(\startx, \endy, \endz)},{PY(\startx, \endy, \endz)})
	to cycle;

\draw[green!50!black!50, line width = .2mm,  left color = green!50!black!20, right color = green!50!black!30, opacity = .35]
	({PX(\startx, \endy, \endz)},{PY(\startx, \endy, \endz)})
	to ({PX(\startx, \endy, g(\startx, \endy))},{PY(\startx, \endy, g(\startx, \endy))})
	to ({PX(\startx, 2.3, \endz)},{PY(\startx, 2.3, \endz)})
	to cycle;

\draw[green!50!black!50, line width = .2mm]
	plot[domain = \startx:\endx, variable = \x] ({PX(\x, \endy, g(\x, \endy))},{PY(\x, \endy, g(\x, \endy))})
	to ({PX(\endx, \endy, \endz)},{PY(\endx, \endy, \endz)})
	to ({PX(\startx, \endy, \endz)},{PY(\startx, \endy, \endz)})
	to cycle;
\draw[green!50!black!50, line width = .2mm]
	({PX(\startx, \endy, \endz)},{PY(\startx, \endy, \endz)})
	to ({PX(\startx, \endy, g(\startx, \endy))},{PY(\startx, \endy, g(\startx, \endy))})
	to ({PX(\startx, 2.3, \endz)},{PY(\startx, 2.3, \endz)})
	to cycle;

\draw[draw = none, line width = .2mm, bottom color = green!50!black!20, top color = green!50!black!60, opacity = .35]
	({PX(\endx, h(\endx,\startz), \startz)},{PY(\endx, h(\endx,\startz), \startz)})
	plot[domain = \endx:\startx, variable = \x] ({PX(\x, h(\x,\startz), \startz)},{PY(\x, h(\x,\startz), \startz)}) --
	plot[domain = 1.59191:\starty, variable = \y] ({PX(\startx, \y, f(\startx,\y))},{PY(\startx, \y, f(\startx,\y))})
	to ({PX(.6, \starty, f(.6,\starty))},{PY(.6, \starty, f(.6,\starty))}) --
	plot[domain = -.828:\endz, variable = \z] ({PX(.6, h(.6,\z), \z)},{PY(.6, h(.6,\z), \z)}) to
	({PX(\endx, \endy, \endz)},{PY(\endx, \endy, \endz)})
	to ({PX(\endx, \endy, -.892)},{PY(\endx, \endy, -.892)}) --
	plot[domain = -.892:\startz, variable = \z] ({PX(\endx, h(\endx,\z), \z)},{PY(\endx, h(\endx,\z), \z)});

\draw[draw = green!50!black!80, line width = .2mm]
	({PX(\endx, h(\endx,\startz), \startz)},{PY(\endx, h(\endx,\startz), \startz)})
	plot[domain = \endx:\startx, variable = \x] ({PX(\x, h(\x,\startz), \startz)},{PY(\x, h(\x,\startz), \startz)}) --
	plot[domain = 1.59191:\starty, variable = \y] ({PX(\startx, \y, f(\startx,\y))},{PY(\startx, \y, f(\startx,\y))})
	to ({PX(.6, \starty, f(.6,\starty))},{PY(.6, \starty, f(.6,\starty))})--
	plot[domain = -.828:0, variable = \z] ({PX(.6, h(.6,\z), \z)},{PY(.6, h(.6,\z), \z)})
	;
\draw[draw = green!50!black!80, line width = .2mm, dashed]
	plot[domain = 0:\endz, variable = \z] ({PX(.6, h(.6,\z), \z)},{PY(.6, h(.6,\z), \z)});
	
\draw[draw = green!50!black!80, line width = .2mm]	
	plot[domain = -.892:\startz, variable = \z] ({PX(\endx, h(\endx,\z), \z)},{PY(\endx, h(\endx,\z), \z)});

\begin{scope}
\clip ({PX(\startx,\starty,\endz)}, {PY(\startx,\starty,\endz)}) to ({PX(\endx,\starty,\endz)}, {PY(\endx,\starty,\endz)}) to ({PX(\endx,\endy,\endz)}, {PY(\endx,\endy,\endz)}) to ({PX(\startx,\endy,\endz)}, {PY(\startx,\endy,\endz))}) to cycle;

\draw[green!50!black!50, line width = .2mm, left color = green!50!black!20, right color = green!50!black!30, opacity = .75]
	plot[domain = .6:\startx, variable = \x] ({PX(\x, h(\x,.85), .85)},{PY(\x, h(\x,.85), .85)})
	to ({PX(.6, h(2.5,.85), .85)},{PY(.6, h(2.5,.85), .85)})
	to ({PX(.6, h(2.5,.85), .85)},{PY(.6, h(2.5,.85), .85)})
	to cycle;
\draw[green!50!black!50, line width = .2mm]
	plot[domain = .6:\startx, variable = \x] ({PX(\x, h(\x,.85), .85)},{PY(\x, h(\x,.85), .85)})
	to ({PX(.6, h(2.5,.85), .85)},{PY(.6, h(2.5,.85), .85)})
	to ({PX(.6, h(2.5,.85), .85)},{PY(.6, h(2.5,.85), .85)})
	to cycle;
\end{scope}

\end{scope}

%
\draw[blue!75, line width = .35mm]({PX(5/6,0,-5/6)}, {PY(5/6,0,-5/6)}) to ({PX(5/6,\endy,-5/6)}, {PY(5/6,\endy,-5/6)});
\draw[blue!75, line width = .35mm]({PX(5/6,0,-5/6)}, {PY(5/6,0,-5/6)}) to ({PX(5/6+1*\rtwobd*4,0+1*\rtwobd*3,-5/6-1*\rtwobd*4)}, {PY(5/6+1*\rtwobd*4,0+1*\rtwobd*3,-5/6-1*\rtwobd*4)});

\begin{scope}
\clip({PX(\startx,\starty,\startz)}, {PY(\startx,\starty,\startz)}) to ({PX(\startx,\starty,\endz)}, {PY(\startx,\starty,\endz)}) to ({PX(\endx,\starty,\endz)}, {PY(\endx,\starty,\endz)}) to ({PX(\endx,\endy,\endz)}, {PY(\endx,\endy,\endz)}) to ({PX(\endx,\endy,\startz)}, {PY(\endx,\endy,\startz))}) to ({PX(\startx,\endy,\startz)}, {PY(\startx,\endy,\startz)}) to cycle;
\draw[blue!75, line width = .35mm]({PX(5/6,0,-5/6)}, {PY(5/6,0,-5/6)}) to ({PX(5/6+1*\rthreebd*3,0-1*\rthreebd*4,-5/6+1*\rthreebd*3)}, {PY(5/6+1*\rthreebd*3,0-1*\rthreebd*4,-5/6+1*\rthreebd*3)});
\end{scope}

\begin{scope}
\clip ({PX(\startx,\starty,\startz)}, {PY(\startx,\starty,\startz)}) to ({PX(\startx,\starty,\endz)}, {PY(\startx,\starty,\endz)}) to ({PX(\endx,\starty,\endz)}, {PY(\endx,\starty,\endz)}) to ({PX(\endx,\endy,\endz)}, {PY(\endx,\endy,\endz)}) to ({PX(\endx,\endy,\startz)}, {PY(\endx,\endy,\startz))}) to ({PX(\startx,\endy,\startz)}, {PY(\startx,\endy,\startz)}) to cycle;

\draw[draw = none, opacity = .7, fill = blue!90] ({PX(5/6,0,-5/6)}, {PY(5/6,0,-5/6)}) to ({PX(5/6,\endy,-5/6)}, {PY(5/6,\endy,-5/6)}) to ({PX(5/6+\rthreebd*3,\endy,-5/6+\rthreebd*3)}, {PY(5/6+\rthreebd*3,\endy,-5/6+\rthreebd*3)}) to ({PX(5/6+\rthreebd*3,0-\rthreebd*4,-5/6+\rthreebd*3)}, {PY(5/6+\rthreebd*3,0-\rthreebd*4,-5/6+\rthreebd*3)}) to cycle;

\draw[draw = none,  opacity = .4, fill = blue!80] ({PX(5/6,0,-5/6)}, {PY(5/6,0,-5/6)}) to ({PX(5/6,\endy,-5/6)}, {PY(5/6,\endy,-5/6)}) to ({PX(5/6+\rtwobd*4,\endy,-5/6-\rtwobd*4)}, {PY(5/6+\rtwobd*4,\endy,-5/6-\rtwobd*4)}) to ({PX(5/6+\rtwobd*4,0+\rtwobd*3,-5/6-\rtwobd*4)}, {PY(5/6+\rtwobd*4,0+\rtwobd*3,-5/6-\rtwobd*4)}) to cycle;

\draw[draw = none, opacity = .5, fill = blue!30] ({PX(5/6,0,-5/6)}, {PY(5/6,0,-5/6)}) to ({PX(5/6+\rtwobd*4,0+\rtwobd*3,-5/6-\rtwobd*4)}, {PY(5/6+\rtwobd*4,0+\rtwobd*3,-5/6-\rtwobd*4)}) to ({PX(5/6+\rthreebd*3,0-\rthreebd*4,-5/6+\rthreebd*3)}, {PY(5/6+\rthreebd*3,0-\rthreebd*4,-5/6+\rthreebd*3)}) to cycle;

\draw[draw = blue!30, opacity = .77, left color=blue!15, right color=blue!20] ({PX(5/6+\rtwobd*4,0+\rtwobd*3,-5/6-\rtwobd*4)}, {PY(5/6+\rtwobd*4,0+\rtwobd*3,-5/6-\rtwobd*4)}) to ({PX(5/6+\rthreebd*3,0-\rthreebd*4,-5/6+\rthreebd*3)}, {PY(5/6+\rthreebd*3,0-\rthreebd*4,-5/6+\rthreebd*3)}) to ({PX(5/6+\rthreebd*3,\endy,-5/6+\rthreebd*3)}, {PY(5/6+\rthreebd*3,\endy,-5/6+\rthreebd*3)}) to ({PX(5/6+\rtwobd*4,\endy,-5/6-\rtwobd*4)}, {PY(5/6+\rtwobd*4,\endy,-5/6-\rtwobd*4)}) to cycle;

\shade[left color=blue!20, right color=blue!40, opacity = .75]({PX(5/6,\endy,-5/6)}, {PY(5/6,\endy,-5/6)}) to ({PX(5/6+\rthreebd*3,\endy,-5/6+\rthreebd*3)}, {PY(5/6+\rthreebd*3,\endy,-5/6+\rthreebd*3)}) to ({PX(5/6+\rtwobd*4,\endy,-5/6-\rtwobd*4)}, {PY(5/6+\rtwobd*4,\endy,-5/6-\rtwobd*4)}) to cycle;

\draw[draw = blue!80, dashed, line width = .2mm] ({PX(5/6+\rtwobd*4,0+\rtwobd*3,-5/6-\rtwobd*4)}, {PY(5/6+\rtwobd*4,0+\rtwobd*3,-5/6-\rtwobd*4)}) to  ({PX(5/6+\rtwobd*4,\endy,-5/6-\rtwobd*4)}, {PY(5/6+\rtwobd*4,\endy,-5/6-\rtwobd*4)}) to ({PX(5/6+\rthreebd*3,\endy,-5/6+\rthreebd*3)}, {PY(5/6+\rthreebd*3,\endy,-5/6+\rthreebd*3)}) to ({PX(5/6+\rthreebd*3,0-\rthreebd*4,-5/6+\rthreebd*3)}, {PY(5/6+\rthreebd*3,0-\rthreebd*4,-5/6+\rthreebd*3)}) to cycle;

\begin{scope}
\clip({PX(5/6+\rtwobd*4,0+\rtwobd*3,-5/6-\rtwobd*4)}, {PY(5/6+\rtwobd*4,0+\rtwobd*3,-5/6-\rtwobd*4)}) to  ({PX(5/6+\rtwobd*4,\endy,-5/6-\rtwobd*4)}, {PY(5/6+\rtwobd*4,\endy,-5/6-\rtwobd*4)}) to ({PX(5/6+\rthreebd*3,\endy,-5/6+\rthreebd*3)}, {PY(5/6+\rthreebd*3,\endy,-5/6+\rthreebd*3)}) to ({PX(5/6+\rthreebd*3,0-\rthreebd*4,-5/6+\rthreebd*3)}, {PY(5/6+\rthreebd*3,0-\rthreebd*4,-5/6+\rthreebd*3)}) to cycle;

\draw[draw = blue!80, dashed, line width = .25mm] ({PX(\startx,\starty,\startz)}, {PY(\startx,\starty,\startz)}) to ({PX(\startx,\starty,\endz)}, {PY(\startx,\starty,\endz)});
\end{scope}

\draw[draw = blue!80, dashed, line width = .2mm] ({PX(5/6,\endy,-5/6)}, {PY(5/6,\endy,-5/6)}) to ({PX(5/6+\rthreebd*3,\endy,-5/6+\rthreebd*3)}, {PY(5/6+\rthreebd*3,\endy,-5/6+\rthreebd*3)}) to ({PX(5/6+\rtwobd*4,\endy,-5/6-\rtwobd*4)}, {PY(5/6+\rtwobd*4,\endy,-5/6-\rtwobd*4)}) to cycle;

\end{scope}
\end{tikzpicture}}%
\\
\multicolumn{2}{c}{
\begin{tabular}{r@{\hskip .1 cm}p{.9\textwidth}}
(c) & A plot of the boundary of $Q_g$ is in green, and $C^G_{\Gamma}\cap H$ is in blue.
The ambient space is $H$ projected onto the coordinates $x_1, x_2,$ and $y_1$.
Using Lemma 5.4 in \cite{munoz2025characterization} and Lemma~\ref{lem:Implied0}, one can argue that $C^{G}_{\Gamma} \cap H$ can be described by intersecting $H$ with the three inequalities defining $C_{\Gamma}$ corresponding to $\mbfs{\beta}_2, \mbfs{\beta}_3$, and $\mbfs{\beta}^*$.
\end{tabular}}
\end{tabular}
\caption{This figure contains three plots illustrating $\Gamma(\mbfs{\beta}) := |\mbfs{\beta}|$, where the absolute value is taken component-wise, $\mbfs{a}{}^\top\Gamma(\mbfs{\beta})+\mbfs{a}{}^\top\mbfs{\beta}$, and $C_{\Gamma} \cap H$, respectively, for $\mbfs{d} := (\sfrac{3}{5},-\sfrac{4}{5})$ and $\mbfs{a} := (-\sfrac{3}{5},\sfrac{4}{5})$. 
We set $\mbfs{\beta}_1 :=  -\mbfs{\beta}_3 := (1,0)$, $\mbfs{\beta}_2 := -\mbfs{\beta}_4 := (0,1)$, and $\mbfs{\beta}^* := (-\sfrac{4}{5}, -\sfrac{3}{5})$.
Fig.~\ref{example:easycase} (b) shows that $G^{\circ}$ is the open interval from $\mbfs{\beta}_2$ to $\mbfs{\beta}^*$, while $\mbfs{\beta}_1 \in G \setminus \cl(G^\circ)$.
Thus, $\cl(G^\circ) \neq G$.
}
\label{example:easycase}
\end{figure}

The set $G^{\circ}$ is helpful in most proofs of our main results. For example, in this section, we critically use $G^\circ$ to prove that $C^G_{\Gamma} \cap H$ is $Q_g$-free when $\Gamma(-\mbfs{d}) \neq \mbfs{a}$; see Lemma \ref{lemma:patching}.
In later sections, when proving the sufficient condition for maximality, the strict inequality in the definition of $G^\circ$ will be leveraged in rescaling arguments.

For our proof that $C^G_{\Gamma} \cap H$ is $Q_g$-free when $\Gamma(-\mbfs{d}) \neq \mbfs{a}$, we use the following two lemmata.
Lemma~\ref{lem:Implied1} follows from the Cauchy-Schwarz Inequality and the inequality $\|\mbfs{a}\| \le \|\mbfs{d}\| = 1$.

\begin{lemma}\label{lem:Implied1}
Assume $\|\mbfs{a}\| \le \|\mbfs{d}\|\ = 1$.
If $\Gamma:D^m \to D^n$ is non-expansive, then $\mbfs{d} \not \in G^{\circ}$.
Furthermore, if $\Gamma(-\mbfs{d}) \neq \mbfs{a}$, then $- \mbfs{d} \in G^{\circ}$.
\end{lemma}

\begin{lemma}\label{lem:Implied0}
Assume $\|\mbfs{a}\| \le \|\mbfs{d}\|\ = 1$.
If $\Gamma:D^m \to D^n$ is non-expansive and $\Gamma(-\mbfs{d}) \neq \mbfs{a}$, then $C^G_{\Gamma} \cap H = C^{G^{\circ}}_{\Gamma} \cap H$.
\end{lemma}

\begin{proof}
The inclusion $G^{\circ} \subseteq G$ implies that $C^G_{\Gamma} \cap H \subseteq C^{G^{\circ}}_{\Gamma} \cap H$.
Hence, it suffices to show $C^G_{\Gamma} \cap H \supseteq C^{G^{\circ}}_{\Gamma} \cap H$.
To this end, we show every inequality defining $C^G_{\Gamma} \cap H$ is valid for $C^{G^{\circ}}_{\Gamma} \cap H$, i.e., we prove the following: for each $\overline{\mbfs{\beta}} \in G \setminus G^{\circ}$,
\begin{equation}\label{eqImpliedWeakly} 
\begin{array}{@{\hskip 0 cm}l}
\text{the inequality $-\Gamma(\overline{\mbfs{\beta}}){}^\top\mbfs{x} + \overline{\mbfs{\beta}}{}^\top\mbfs{y} \le 0$ is implied by $\mbfs{a}{}^\top\mbfs{x}+\mbfs{d}{}^\top\mbfs{y} = -1$}\\
\text{and inequalities of the form $-\Gamma({\mbfs{\beta}}){}^\top\mbfs{x} + {\mbfs{\beta}}{}^\top\mbfs{y} \le 0$, where $\mbfs{\beta} \in G^{\circ}$.}
 \end{array}
 \end{equation}

As a first case, suppose $\mbfs{d} \in G \setminus G^{\circ}$.
We have $0 \ge \mbfs{a}{}^\top\Gamma(\mbfs{d}) + \mbfs{d}{}^\top\mbfs{d} = \mbfs{a}{}^\top\Gamma(\mbfs{d})  + 1$ because $\mbfs{d} \in G$.
Thus, $\Gamma(\mbfs{d}) = - \mbfs{a}$ by the Cauchy-Schwarz Inequality. 
Hence, the inequality $-\Gamma(\mbfs{d}){}^\top\mbfs{x} + \mbfs{d}{}^\top\mbfs{y} \le 0$ is equivalent to $\mbfs{a}{}^\top\mbfs{x} + \mbfs{d}{}^\top \mbfs{y} \le 0$, which is implied by $\mbfs{a}{}^\top\mbfs{x}+\mbfs{d}{}^\top\mbfs{y} = -1$.

As a second case, suppose $m = 1$ and let $\overline{\mbfs{\beta}} \in G \setminus G^{\circ}$.
Here, we have $D^m = D^1 = \{\pm \mbfs{d}\}$.
By Lemma~\ref{lem:Implied1}, the only case to consider is if $\overline{\mbfs{\beta}} = \mbfs{d} \in G \setminus G^{\circ}$, which is the first case.

In light of the previous two cases, it suffices to show that \eqref{eqImpliedWeakly} is true when $m \ge 2$ and $\overline{\mbfs{\beta}} \in G \setminus (G^{\circ} \cup \{\mbfs{d}\})$; we assume this for the rest of the proof.

There exists an orthonormal basis $\mbfs{d},\overline{\mbfs{d}}$ of $\text{span}\{\mbfs{d}, \overline{\mbfs{\beta}}\} $.
After possibly swapping $\overline{\mbfs{d}}$ with $-\overline{\mbfs{d}}$, we assume $\overline{\mbfs{d}}{}^\top \overline{\mbfs{\beta}} \ge 0$.
As $\overline{\mbfs{\beta}} \not\in \{\pm \mbfs{d}\}$, there exists a number $\overline{\alpha}$ such that $-1 < \overline{\alpha} < 1$ and
\[
\textstyle\overline{\mbfs{\beta}} = \overline{\alpha} \cdot \mbfs{d} + \sqrt{1- \overline{\alpha}{}^2} \cdot \overline{\mbfs{d}}.
\]

For $\alpha \in \mbb{R}$ satisfying $-1 < \alpha < 1$, define the vector 
\[
\textstyle\mbfs{\beta}_{\alpha} :=  {\alpha} \cdot \mbfs{d} + \sqrt{1- {\alpha}{}^2} \cdot \overline{\mbfs{d}}.
\]
Thus, $\overline{\mbfs{\beta}} = \mbfs{\beta}_{\overline{\alpha}}$.
Note that this parametrization only considers a half circle: $\alpha$ can be negative, but $\sqrt{1- {\alpha}{}^2}$ cannot.

Define 
\[
\alpha^* := \sup\{\alpha \in (-1,\overline{\alpha}): \mbfs{a}{}^\top \Gamma(\mbfs{\beta}_{\alpha}) + \mbfs{d}{}^\top \mbfs{\beta}_{\alpha} < 0 \}.
\]
The set defining $\alpha^*$ is non-empty because $-\mbfs{d} \in G^{\circ}$ and $\alpha \mapsto \mbfs{a}{}^\top \Gamma(\mbfs{\beta}_{\alpha}) + \mbfs{d}{}^\top \mbfs{\beta}_{\alpha}$ is continuous.
Continuity also implies $-1 < \alpha^* \le \overline{\alpha} < 1$ and $\mbfs{a}{}^\top \Gamma(\mbfs{\beta}_{\alpha^*}) + \mbfs{d}{}^\top \mbfs{\beta}_{\alpha^*} =0 $.

We will prove that there are conic coefficients $\lambda_1, \lambda_2 \ge 0$ such that
 \[
 (-\Gamma(\overline{\mbfs{\beta}}),  \overline{\mbfs{\beta}}) = \lambda_1 \cdot (\mbfs{a}, \mbfs{d}) + \lambda_2 \cdot (-\Gamma(\mbfs{\beta}_{\alpha^*}), \mbfs{\beta}_{\alpha^*}).
 \]
After proving this, we will have that the inequality $-\Gamma(\overline{\mbfs{\beta}}){}^\top\mbfs{x} + \overline{\mbfs{\beta}}{}^\top\mbfs{y} \le 0$ is implied by $-\Gamma(\mbfs{\beta}_{\alpha^*}){}^\top\mbfs{x} + \mbfs{\beta}_{\alpha^*} {}^\top\mbfs{y} \le 0$ and $\mbfs{a}{}^\top\mbfs{x}+\mbfs{d}{}^\top\mbfs{y} = -1$.
The definition of $\alpha^* $ implies that $-\Gamma(\mbfs{\beta}_{\alpha^*}){}^\top\mbfs{x} + \mbfs{\beta}_{\alpha^*}{}^\top\mbfs{y} \le 0$ is a limit of inequalities of the form $-\Gamma(\mbfs{\beta}_{\alpha_t}){}^\top\mbfs{x} + \mbfs{\beta}_{\alpha_t}^\top\mbfs{y} \le 0$ for $\mbfs{\beta}_{\alpha_t} \in G^{\circ}$. 
Hence, \eqref{eqImpliedWeakly} holds for $\overline{\mbfs{\beta}}$.
Thus, it remains to find $\lambda_1$ and $\lambda_2$.

Set
\[
\lambda_1 := \overline{\alpha} - \lambda_2  \alpha^* \quad \text{and} \quad \lambda_2 := \frac{\sqrt{1-\overline{\alpha}^2}}{\sqrt{1-(\alpha^*)^2}}.
\]
Note that $1-(\alpha^*)^2 > 0$ because $-1 < \alpha^* < 1$.
Clearly, $\lambda_2 \ge 0$.
Next, we prove $\lambda_1 \ge 0$.
Recall $\alpha^* \le \overline{\alpha}$.
If $0 \le \alpha^*\le \overline{\alpha}$ or $\alpha^* \le \overline{\alpha} \le 0$, then 
\[
\lambda_1 = \frac{\overline{\alpha}\sqrt{1-(\alpha^*)^2\vphantom{\overline{\alpha}^2}}-\alpha^*\sqrt{1-\overline{\alpha}^2}}{\sqrt{1-(\alpha^*)^2}} \ge \frac{\alpha^*\sqrt{1-(\alpha^*)^2\vphantom{\overline{\alpha}^2}}-\alpha^*\sqrt{1-\overline{\alpha}^2}}{\sqrt{1-(\alpha^*)^2}}  \ge 0.
\]
If $\alpha^* < 0 < \overline{\alpha}$, then 
\[
\lambda_1 = \frac{\overline{\alpha}\sqrt{1-(\alpha^*)^2}-\alpha^*\sqrt{1-\overline{\alpha}^2}}{\sqrt{1-(\alpha^*)^2}} \ge  0.
\]
Hence, $\lambda_1 \ge 0$.

It can be checked that $\lambda_1 \cdot \mbfs{d} + \lambda_2 \cdot \mbfs{\beta}_{\alpha^*} = \overline{\mbfs{\beta}}$.
We will complete the proof by showing $\Gamma(\overline{\mbfs{\beta}})= -\lambda_1 \cdot \mbfs{a} + \lambda_2 \cdot \Gamma(\mbfs{\beta}_{\alpha^*})$.
Set $\widetilde{\gamma} := -\lambda_1 \cdot \mbfs{a} + \lambda_2 \cdot \Gamma(\mbfs{\beta}_{\alpha^*})$.

Using the fact that $\mbfs{\beta}_{\alpha^*} \in G \setminus G^{\circ}$ and $\|\mbfs{a}\| \le 1$, we have
\[
\|\widetilde{\gamma} \|^2 = \|\mbfs{a}\|^2 \lambda_1^2+\lambda_2^2 - 2\lambda_1\lambda_2 \mbfs{a}{}^\top\Gamma(\mbfs{\beta}_{\alpha^*}) \le  \lambda_1^2+\lambda_2^2 + 2\lambda_1\lambda_2 \mbfs{d}{}^\top\mbfs{\beta}_{\alpha^*} = \|\overline{\mbfs{\beta}}\|^2 = 1.
\]
Hence, $\|\widetilde{\gamma} \| \le 1$. 
Using this, together with $\overline{\mbfs{\beta}} \in G \setminus G^{\circ}$ and the fact that $\Gamma$ is non-expansive, we have
\begin{align*}
1 \ge  \|\Gamma(\overline{\mbfs{\beta}})\| \cdot \|\widetilde{\gamma}\|  \ge \Gamma(\overline{\mbfs{\beta}}){}^\top \widetilde{\gamma} =&~ -\lambda_1 \Gamma(\overline{\mbfs{\beta}}){}^\top \mbfs{a}+\lambda_2 \Gamma(\overline{\mbfs{\beta}}){}^\top \Gamma(\mbfs{\beta}_{\alpha^*})\\
=&~ \lambda_1 \mbfs{d}{}^\top\overline{\mbfs{\beta}} + \lambda_2 \Gamma(\overline{\mbfs{\beta}}){}^\top \Gamma(\mbfs{\beta}_{\alpha^*})\\
\ge &~ \lambda_1 \mbfs{d}{}^\top\overline{\mbfs{\beta}} + \lambda_2 \overline{\mbfs{\beta}}{}^\top \mbfs{\beta}_{\alpha^*}\\
= &~\overline{\mbfs{\beta}}{}^\top \overline{\mbfs{\beta}}\\
= &~1.
\end{align*}
Hence, $\|\widetilde{\gamma} \| = 1$ and $\widetilde{\gamma} = \Gamma(\overline{\mbfs{\beta}})$.
\end{proof}

We now prove that $C^{G}_{\Gamma} \cap H = C^{G^{\circ}}_{\Gamma} \cap H$ is $Q_g$-free, provided $\mbfs{a} \neq \Gamma(-\mbfs{d})$.

\begin{lemma}\label{lemma:patching}
Assume $\|\mbfs{a}\| \le \|\mbfs{d}\|\ = 1$.
If $\Gamma : D^m \to D^n$ is such that $C_{\Gamma}$ is a full-dimensional maximal $Q_h$-free set and $\mbfs{a} \neq \Gamma(-\mbfs{d})$, then $C^{G}_{\Gamma} \cap H$ is $Q_g$-free.
\end{lemma}
\begin{proof}
In this proof, we design a function $\overline{\Gamma}:D^m \to \mbb{R}^n$ such that $\|\overline{\Gamma}(\mbfs{\beta}) \| \le 1$ for each $\mbfs{\beta} \in D^m$, and
\[
K   = C^G_{\Gamma} \cap H = C^{G^{\circ}}_{\Gamma} \cap H  =  C^{D^m}_{\overline{\Gamma}}  \cap H= C_{\overline{\Gamma}}  \cap H.\nonumber
\]
The second equation follows from Lemma~\ref{lem:Implied0}.
The $\overline{\Gamma}$ we define will be constructed to satisfy the third equation.
The idea behind the construction of $\overline{\Gamma}$ is that it will coincide with $\Gamma$ when $\mbfs{\beta}\in G^\circ$, and for $\mbfs{\beta}\in D^m \setminus G^\circ$ the inequality $-\overline{\Gamma}(\mbfs{\beta}){}^\top\mbfs{x}+\mbfs{\beta}{}^\top\mbfs{y} \le 0$ will be redundant.

Once $\overline{\Gamma}$ is constructed, we can argue the set $C_{\overline{\Gamma}}$ is $Q_h$-free.
Indeed, let $(\mbfs{x}, \mbfs{y}) \in Q_h$.
If $\mbfs{y} \neq \mbfs{0}$, then $-\overline{\Gamma}(\mbfs{y}/\|\mbfs{y}\|)^\top\mbfs{x} + (\mbfs{y}/\|\mbfs{y}\|){}^\top \mbfs{y} \ge 0 $ by the Cauchy-Schwarz Inequality and the fact that $\|\mbfs{x}\| \le \|\mbfs{y}\|$.
If $\mbfs{y} = \mbfs{0}$, then $(\mbfs{x}, \mbfs{y}) = (\mbfs{0}, \mbfs{0})$ satisfies all inequalities defining $C_{\overline{\Gamma}}$ at equality. 
Thus, $C_{\overline{\Gamma}}\cap H$ is $Q_g$-free.
In what remains, we focus on defining $\overline{\Gamma}$ such that $C^{G^{\circ}}_{\Gamma} \cap H  =  C^{D^m}_{\overline{\Gamma}}  \cap H$.

For $\mu \ge 0$ and $\overline{\mbfs{\beta}} \in D^m \setminus (G^{\circ} \cup \{\mbfs{d}\})$, define the vectors
\begin{align*}
\overline{\mbfs{\beta}}_{\mu} & := \frac{ \overline{\mbfs{\beta}} - \mu \cdot \mbfs{d}}{ \left\|\overline{\mbfs{\beta}} - \mu \cdot \mbfs{d}\right\|} \quad\text{and}\\[.15cm]
\Psi_{\overline{\mbfs{\beta}}}(\mu) &:= 
\left\|\overline{\mbfs{\beta}} - \mu \cdot \mbfs{d}\right\| \cdot \Gamma\left(\overline{\mbfs{\beta}}_{\mu} \right) - \mu \cdot \mbfs{a}.
\end{align*}
For $\overline{\mbfs{\beta}} \in D^m \setminus (G^{\circ} \cup \{\mbfs{d}\})$, define the number
\[
\mu_{\overline{\mbfs{\beta}}}  := \sup \{\mu \ge 0:\ \overline{\mbfs{\beta}}_{\mu} \not\in G^{\circ}\}.
\]

\begin{cm}\label{claim2}
For each $\overline{\mbfs{\beta}} \in  D^m \setminus (G^{\circ} \setminus \{\mbfs{d}\})$, we have $\mu_{\overline{\mbfs{\beta}}} < \infty$.
\end{cm}
\begin{cpf}
We have $\lim_{\mu\to\infty} \overline{\mbfs{\beta}}_{\mu}  = -\mbfs{d}$. 
By Lemma~\ref{lem:Implied1}, we know $-\mbfs{d}\in G^{\circ}$. 
The continuity of $\mu \mapsto \mbfs{a}{}^\top \Gamma(\overline{\mbfs{\beta}}_{\mu}) + \mbfs{d}{}^\top \overline{\mbfs{\beta}}_{\mu} $ then implies that $\mu_{\overline{\mbfs{\beta}}} < \infty$. 
\end{cpf}

\begin{cm}\label{claim3}
For each $\overline{\mbfs{\beta}} \in  D^m \setminus (G^{\circ} \cup \{\mbfs{d}\})$, we have $\mbfs{a}{}^\top\Gamma(\overline{\mbfs{\beta}}_{\mu_{\overline{\mbfs{\beta}}}})+\mbfs{d}{}^\top \overline{\mbfs{\beta}}_{\mu_{\overline{\mbfs{\beta}}}} = 0$.
\end{cm}
\begin{cpf}
By the definition of $\mu_{\overline{\mbfs{\beta}}}$ and the continuity of $\Gamma$, we have $\mbfs{a}{}^\top\Gamma(\overline{\mbfs{\beta}}_{\mu_{\overline{\mbfs{\beta}}}})+\mbfs{d}{}^\top \overline{\mbfs{\beta}}_{\mu_{\overline{\mbfs{\beta}}}} \geq 0$.
Assume that $\mbfs{a}{}^\top\Gamma(\overline{\mbfs{\beta}}_{\mu_{\overline{\mbfs{\beta}}}})+\mbfs{d}{}^\top \overline{\mbfs{\beta}}_{\mu_{\overline{\mbfs{\beta}}}} > 0$.
By the continuity of $\Gamma$, one can place an open ball around $\overline{\mbfs{\beta}}_{\mu_{\overline{\mbfs{\beta}}}}$ that is disjoint from $G^{\circ}$.
Such an open ball would let us increase $\mu_{\overline{\mbfs{\beta}}}$, which is a contradiction.
\end{cpf}

Define the function $\overline{\Gamma}:D^m \to \mbb{R}^n$ as follows: 
\[
\overline{\Gamma}({\mbfs{\beta}}) := 
\left\{
\begin{array}{l@{\hskip .5cm}l}
\Gamma({\mbfs{\beta}}) &\text{if}~{{\mbfs{\beta}}} \in G^{\circ}\\[.1cm]
\Psi_{{\mbfs{\beta}}}(\mu_{{\mbfs{\beta}}})&\text{if}~{{\mbfs{\beta}}} \in D^m \setminus ( G^{\circ} \cup \{\mbfs{d}\})\\[.1cm]
-\mbfs{a} &\text{if}~\mbfs{\beta} = \mbfs{d}
\end{array}
\right.
\]

\begin{cm}\label{claim4}
For each $\mbfs{\beta} \in  D^m $, we have $\|\overline{\Gamma}(\mbfs{\beta})\| \leq 1$.
\end{cm}
\begin{cpf}
If $\mbfs{\beta} \in G^{\circ}$, then the result holds because $\Gamma$ maps into $D^n$.
If $\mbfs{\beta} \in D^m \setminus ( G^{\circ} \cup \{\mbfs{d}\})$, then set $\varepsilon := \|\mbfs{\beta} - \mu_{\mbfs{\beta}} \cdot \mbfs{d}\|$.
Observe that
\begin{align*}
\|\Psi_{\mbfs{\beta}}(\mu_{\mbfs{\beta}})\|^2
= &  \varepsilon^2 - 2\varepsilon \mu_{\mbfs{\beta}} \cdot\mbfs{a}{}^\top\Gamma(\mbfs{\beta}_{\mu_{\mbfs{\beta}}}) + \mu_{\mbfs{\beta}}{}^2 \cdot\mbfs{a}{}^\top\mbfs{a}\\
 = &   \varepsilon^2 + 2\varepsilon\mu_{\mbfs{\beta}} \cdot\mbfs{d}{}^\top \mbfs{\beta}_{\mu_{\mbfs{\beta}}} + \mu_{\mbfs{\beta}}{}^2 \cdot\mbfs{a}{}^\top\mbfs{a} && \text{by Claim~\ref{claim3}}\\
\leq&   \varepsilon^2 + 2\varepsilon\mu_{\mbfs{\beta}} \cdot\mbfs{d}{}^\top \mbfs{\beta}_{\mu_{\mbfs{\beta}}} + \mu_{\mbfs{\beta}}{}^2\cdot \mbfs{d}{}^\top\mbfs{d} && \text{because $\|\mbfs{a}\| \leq \|\mbfs{d}\| $ }\\
=&  \varepsilon^2 \cdot\mbfs{\beta}_{\mu_{\mbfs{\beta}}}{}^\top \mbfs{\beta}_{\mu_{\mbfs{\beta}}} + 2\varepsilon\mu_{\mbfs{\beta}}\cdot \mbfs{d}{}^\top\mbfs{\beta}_{\mu_{\mbfs{\beta}}} + \mu_{\mbfs{\beta}}{}^2 \cdot\mbfs{d}{}^\top\mbfs{d} && \text{because $\|\mbfs{\beta}_{\mu_{\mbfs{\beta}}}\| =1 $}\\
=&  \|\varepsilon \cdot \mbfs{\beta}_{\mu_{\mbfs{\beta}}} + \mu_{\mbfs{\beta}}\cdot \mbfs{d}\|^2\\
=&  \|\mbfs{\beta}\|^2\\
=&  1.
\end{align*}
If $\mbfs{\beta} = \mbfs{d}$, then the result holds because $\|\mbfs{a}\| \le 1$.
\end{cpf}

\begin{cm}\label{claim5}
For each $\overline{\mbfs{\beta}} \in  D^m \setminus G^{\circ}$, the inequality 
\(
-\overline{\Gamma}(\overline{\mbfs{\beta}}){}^\top \mbfs{x} + \overline{\mbfs{\beta}}{}^\top \mbfs{y} \le 0
\)
is implied by the equation $\mbfs{a}{}^\top \mbfs{x}+\mbfs{d}{}^\top\mbfs{y} = -1$ and inequalities of the form $-\Gamma(\mbfs{\beta}){}^\top \mbfs{x} + \mbfs{\beta}{}^\top \mbfs{y} \le 0$, where $\mbfs{\beta} \in G^{\circ}$.
\end{cm}

\begin{cpf}
First, assume $\overline{\mbfs{\beta}} = \mbfs{d}$.
The inequality $-\overline{\Gamma}(\mbfs{d}){}^\top \mbfs{x} + \mbfs{d}{}^\top \mbfs{y} \le 0$ is equivalent to $\mbfs{a}{}^\top \mbfs{x} + \mbfs{d}{}^\top \mbfs{y} \le 0$, which is implied by $\mbfs{a}{}^\top \mbfs{x}+\mbfs{d}{}^\top\mbfs{y} = -1$.

Now, assume $\overline{\mbfs{\beta}}  \neq \mbfs{d}$.
Observe that
\begin{equation}\label{eqImplied}
(-\Psi_{\overline{\mbfs{\beta}}}(\mu_{\overline{\mbfs{\beta}}}),~ \overline{\mbfs{\beta}}) =  \|\overline{\mbfs{\beta}} - \mu_{\overline{\mbfs{\beta}}} \cdot \mbfs{d}\| \cdot (-\Gamma(\overline{\mbfs{\beta}}_{\mu_{\overline{\mbfs{\beta}}}}),~ \overline{\mbfs{\beta}}_{\mu_{\overline{\mbfs{\beta}}}}) + \mu_{\overline{\mbfs{\beta}}}\cdot \left(\mbfs{a}, \mbfs{d}\right).
\end{equation}
Thus, $-\Psi_{\overline{\mbfs{\beta}}}(\mu_{\overline{\mbfs{\beta}}})^\top \mbfs{x} + \overline{\mbfs{\beta}}{}^\top \mbfs{y} \le 0$ is implied by $\mbfs{a}{}^\top \mbfs{x}+\mbfs{d}{}^\top\mbfs{y} = -1$ and $-\Gamma(\overline{\mbfs{\beta}}_{\mu_{\overline{\mbfs{\beta}}}}){}^\top \mbfs{x} + \overline{\mbfs{\beta}}_{\mu_{\overline{\mbfs{\beta}}}}{}^\top \mbfs{y} \le 0$. 
For $\delta > 0$, we have $\mbfs{\beta}_{\mu_{\overline{\mbfs{\beta}}} + \delta} \in G^{\circ}$ by the supremum definition of $\mu_{\overline{\mbfs{\beta}}}$.
Thus, we have the following, where $\delta \to 0$ from the right:
\[ 
\lim_{\delta \to 0} (-\Gamma(\mbfs{\beta}_{\mu_{\overline{\mbfs{\beta}}} + \delta}),~\mbfs{\beta}_{\mu_{\overline{\mbfs{\beta}}} + \delta}) = (-\Gamma(\overline{\mbfs{\beta}}_{\mu_{\overline{\mbfs{\beta}}}}),~\overline{\mbfs{\beta}}_{\mu_{\overline{\mbfs{\beta}}}}).
\]
Hence, $-\overline{\Gamma}(\overline{\mbfs{\beta}}){}^\top \mbfs{x} + \overline{\mbfs{\beta}}{}^\top \mbfs{y} \le 0$ is implied by the equation $\mbfs{a}{}^\top \mbfs{x}+\mbfs{d}{}^\top\mbfs{y} = -1$ and inequalities of the form $-\Gamma(\mbfs{\beta}){}^\top \mbfs{x} + \mbfs{\beta}{}^\top \mbfs{y} \le 0$, where $\mbfs{\beta} \in G^{\circ}$.
\end{cpf}
Claims~\ref{claim4} and~\ref{claim5} complete the proof of Lemma~\ref{lemma:patching}.
\end{proof}

We are now prepared to prove Theorem~\ref{thm:Both} when $\|\mbfs{a}\| \le \|\mbfs{d}\|\ = 1$.

\begin{lemma}[Theorem~\ref{thm:Both} for $\|\mbfs{a}\| \le \|\mbfs{d}\|\ = 1$]\label{lemma:necessary_easy}
Assume $\|\mbfs{a}\| \le \|\mbfs{d}\|\ = 1$.
If $K\subseteq H$ is a maximal $Q_g$-free set, then $K = C^G_{\Gamma} \cap H$ for some $\Gamma : D^m \to D^n$ such that $C_{\Gamma}$ is a full-dimensional maximal $Q_h$-free set.
Furthermore, if $\Gamma$ satisfies this conclusion, then $\mbfs{a} \neq \Gamma(-\mbfs{d})$.
\end{lemma}

\begin{proof}
By Lemma \ref{sliceofCgamma}, there exists a function $\Gamma:D^m\to D^n$ such that $C_{\Gamma}$ is a full-dimensional maximal $Q_h$-free set, and
$K = C_{\Gamma} \cap H $.

Note that $\mbfs{a} \neq \Gamma(-\mbfs{d})$.
Otherwise, $-\mbfs{a}{}^\top\mbfs{x}-\mbfs{d}{}^\top \mbfs{x} \le 0$ is valid for $C^G_{\Gamma}$ because $-\mbfs{d} \in G$, yet invalid for $H$.
Thus, we would have $K = C^G_{\Gamma} \cap H = \emptyset$, which contradicts that $K$ is full-dimensional.

The inclusion $K = C_{\Gamma} \cap H \subseteq C^G_{\Gamma} \cap H$ follows because $G \subseteq D^m$.
By Lemma \ref{lemma:patching}, $C^G_{\Gamma}\cap H$ is $Q_g$-free. 
As $K$ is maximal, we conclude that $K = C^G_{\Gamma}\cap H$.
\end{proof}
%

\subsection{A proof of Theorem~\ref{thm:Both} when \texorpdfstring{$\|\mbfs{d}\| < \|\mbfs{a}\|\ = 1$}{}.}\label{sec:HardNecessary}

In our proof of Theorem~\ref{thm:Both} when $\|\mbfs{a}\| \le \|\mbfs{d}\|\ = 1$, we relied on the set $G^{\circ}$ defined in \eqref{defnGcirc}.
If $\|\mbfs{a}\| \le \|\mbfs{d}\|\ = 1$ and $\mbfs{a} \neq \Gamma(-\mbfs{d})$, then $G^{\circ} \neq \emptyset$ as it contains $-\mbfs{d}$.
However, if $\|\mbfs{d}\| < \|\mbfs{a}\|\ = 1$, then $G^{\circ}$ may be empty even when $C^G_{\Gamma} \cap H$ is maximal $Q_g$-free; see, e.g., Figure~\ref{example:hardcase}.
Due to this distinction, our proofs have a different flavor when $\|\mbfs{d}\| < \|\mbfs{a}\|\ = 1$.
For instance, in our proof of Theorem~\ref{thm:Both} for $\|\mbfs{d}\| < \|\mbfs{a}\|\ = 1$, we adapt a technique from \cite{MS2022}: If we assume to the contrary that there exists some $\mbfs{\beta} \in D^m \setminus G$, then we can `push out' the $-\Gamma(\mbfs{\beta}){}^\top\mbfs{x}+\mbfs{\beta}{}^\top\mbfs{y} \le 0$, which is valid for $C_{\Gamma} \cap H$, by some offset $\Lambda (\mbfs{\beta}) > 0$ and maintain $Q_g$-freeness.

\begin{figure}
\begin{center}
\begin{tikzpicture}[scale = .4,every node/.style={font=\normalsize},
declare function={
PX(\x,\y,\z) = \x - .75*\y ;
PY(\x,\y,\z) = -.15*\x-.25*\y + \z;
} 
]

\draw[opacity = .15]({PX(0,0,0)}, {PY(0,0,0)}) to ({PX(0,0,1.55)}, {PY(0,0,1.55)});
\draw[opacity = .15]({PX(0,0,0)}, {PY(0,0,0)}) to ({PX(1.25,0,0)}, {PY(1.25,0,0)});
\draw[opacity = .15]({PX(0,0,0)}, {PY(0,0,0)}) to ({PX(0,1,0)}, {PY(0,1,0)});

\draw[draw = none] ({PX(-2,2.15,-2)}, {PY(-2,2.15,-2)}) to ({PX(-2,2.15,2)}, {PY(-2,2.15,2)}) to ({PX(-2,-2.15,2)}, {PY(-2,-2.15,2)}) to ({PX(2.15,-2.15,2)}, {PY(2.15,-2.15,2)}) to ({PX(2.15,-2.15,-2)}, {PY(2.15,-2.15,-2)}) to ({PX(2.15,2.15,-2)}, {PY(2.15,2.15,-2)}) to cycle;

\begin{scope}
\color{blue!75}
\pgfsetstrokeopacity{.65}
\foreach \z in {0, 0.05, ..., 3}{
\pgfpathellipse{\pgfpointxy{{PX(0,0,\z)}}{PY(0,0,\z)}}{\pgfpointxy{1*\z}{-.15*\z}}{\pgfpointxy{.75*\z}{.25*\z}} 
}
\pgfusepath{stroke}
\end{scope}%

\begin{scope}
\color{green!70!black!30}
\pgfsetstrokeopacity{.85}
\foreach \z in {1.75, 1.74, ..., 0.05}{
\pgfpathellipse{\pgfpointxy{{PX(0,0,-\z)}}{PY(0,0,-\z)}}{\pgfpointxy{sqrt(1+\z*\z)}{-.15*sqrt(1+\z*\z)}}{\pgfpointxy{.75*sqrt(1+\z*\z)}{.25*sqrt(1+\z*\z)}} 
}

\foreach \z in {0, 0.05, ..., 2.65}{
\pgfpathellipse{\pgfpointxy{{PX(0,0,\z)}}{PY(0,0,\z)}}{\pgfpointxy{sqrt(1+\z*\z)}{-.15*sqrt(1+\z*\z)}}{\pgfpointxy{.75*sqrt(1+\z*\z)}{.25*sqrt(1+\z*\z)}} 
}
\pgfusepath{stroke}
\end{scope}%

\begin{scope}
\color{green!70!black!80}
\pgfsetstrokeopacity{1}
\foreach \z in {2.65}{
\pgfpathellipse{\pgfpointxy{{PX(0,0,\z)}}{PY(0,0,\z)}}{\pgfpointxy{sqrt(1+\z*\z)}{-.15*sqrt(1+\z*\z)}}{\pgfpointxy{.75*sqrt(1+\z*\z)}{.25*sqrt(1+\z*\z)}} 
}
\pgfusepath{stroke}
\end{scope}%

\begin{scope}
\clip[rotate = 1] ({PX(0,0,3.4)}, {PY(0,0,3.4)}) ellipse (4.6 cm and 1.57 cm);
\color{blue!75}
\pgfsetstrokeopacity{.65}
\foreach \z in {0, 0.05, ..., 3}{
\pgfpathellipse{\pgfpointxy{{PX(0,0,\z)}}{PY(0,0,\z)}}{\pgfpointxy{1*\z}{-.15*\z}}{\pgfpointxy{.75*\z}{.25*\z}} 
}
\pgfusepath{stroke}
\end{scope}%

\begin{scope}

\color{blue!15}
\pgfsetfillopacity{.65}
\foreach \z in {2.95, 3}{
\pgfpathellipse{\pgfpointxy{{PX(0,0,\z)}}{PY(0,0,\z)}}{\pgfpointxy{1*\z}{-.15*\z}}{\pgfpointxy{.75*\z}{.25*\z}} 
}
\pgfusepath{fill}
\end{scope}

\begin{scope}
\color{blue}
\pgfsetstrokeopacity{1}
\foreach \z in {3}{
\pgfpathellipse{\pgfpointxy{{PX(0,0,\z)}}{PY(0,0,\z)}}{\pgfpointxy{1*\z}{-.15*\z}}{\pgfpointxy{.75*\z}{.25*\z}} 
}
\pgfusepath{stroke}
\end{scope}%

\draw[->]({PX(0,0,2)}, {PY(0,0,2)}) to node[pos = 1, left]{$x_2$} ({PX(0,0,4.5)}, {PY(0,0,4.5)});
\draw[->]({PX(1,0,0)}, {PY(1,0,0)}) to node[pos = 1, right]{$y_1$}({PX(3,0,0)}, {PY(3,0,0)});
\draw[->]({PX(0,1,0)}, {PY(0,1,0)}) to node[pos = 1, left]{$y_2$} ({PX(0,3,0)}, {PY(0,3,0)});

\end{tikzpicture}
\end{center}
\caption{%
This figure illustrates $C_{\Gamma} \cap H$ for $\mbfs{d} := (0,0)$, $\mbfs{a} := (1,0)$, and $\Gamma(\mbfs{\beta}) := (0,1)$ for all $\mbfs{\beta} \in D^2$.
The boundary of $Q_g$ is in green, and the boundary of $C_{\Gamma} \cap H$ is in blue.
The ambient space of the figure is $H$.
The set $C_{\Gamma} \cap H$ is a maximal $Q_g$-free set, and $G^{\circ} = \emptyset$.}\label{example:hardcase}
\end{figure}
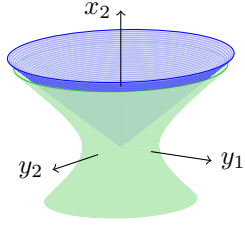

\begin{lemma}[Theorem~\ref{thm:Both} for $\|\mbfs{d}\| < \|\mbfs{a}\|\ = 1$]\label{lemma:hardcase_necessary}
Assume $\|\mbfs{d}\| < \|\mbfs{a}\|\ = 1$.
If $K\subseteq H$ is a maximal $Q_g$-free set, then $K = C^G_{\Gamma} \cap H$ for some $\Gamma : D^m \to D^n$ such that $C_{\Gamma}$ is a full-dimensional maximal $Q_h$-free set.
Furthermore, if $\Gamma$ satisfies this conclusion, then $D^m = G$.
\end{lemma}

\begin{proof}
By Lemma~\ref{sliceofCgamma}, there exists a function $\Gamma : D^m \to D^n$ such that $C_{\Gamma}$ is a full-dimensional maximal $Q_h$-free set and $K =  C_{\Gamma} \cap H = C^{D^m}_{\Gamma} \cap H$.
We claim that any such $\Gamma$ satisfies $G = D^m$; proving this will prove the lemma.

Assume to the contrary that $G \subsetneq D^m $.
We construct a $Q_g$-free set $C^+_\Gamma \cap H$ that is a strict superset of $K$, which contradicts that $K$ is maximal.

For each $\mbfs{\beta} \in D^m$, define
\[
\Lambda(\mbfs{\beta}) := \max \left\{0, ~ 2 \cdot\frac{\mbfs{a}{}^\top \Gamma(\mbfs{\beta}) + \mbfs{d}{}^\top \mbfs{\beta}}{1-(\mbfs{d}{}^\top \mbfs{\beta})^2}\right\}.
\]
The number $\Lambda(\mbfs{\beta})$ is well-defined because $\|\mbfs{d}\| < 1$. 
Define
\[
C^+_\Gamma := \{(\mbfs{x},\mbfs{y}) : - \Gamma(\mbfs{\beta}){}^\top \mbfs{x}  + \mbfs{\beta}{}^\top \mbfs{y} \le \Lambda(\mbfs{\beta}) ~~\forall ~\mbfs{\beta} \in D^m \}.
\]

\begin{cm}\label{cm:Qgfree}
The set $C^+_\Gamma \cap H$ is $Q_g$-free.
\end{cm}
Note that $C^+_\Gamma$ may not be $Q_h$-free, but in this case we are only concerned with the behavior on the slice.

\begin{cpf} 
Let $(\mbfs{x}_0, \mbfs{y}_0)\in Q_g$.
We prove
\[
-\Gamma(\mbfs{\beta_0}){}^\top \mbfs{x}_0 + \mbfs{\beta}_0{}^\top \mbfs{y}_0 \geq  \Lambda(\mbfs{\beta}_0), 
\]
where $\mbfs{\beta}_0 := \mbfs{y}_0 / \| \mbfs{y}_0\|$. 
Note that $\mbfs{y}_0 \neq \mbfs{0}$, as otherwise $\mbfs{x}_0 = \mbfs{0}$ because $(\mbfs{x}_0, \mbfs{y}_0) \in Q_g \subseteq Q_h$; however, $(\mbfs{0}, \mbfs{0}) \not\in H \supseteq Q_g$.

Assume to the contrary that 
\[
-\Gamma(\mbfs{\beta_0}){}^\top \mbfs{x}_0 + \mbfs{\beta}_0{}^\top \mbfs{y}_0 <  \Lambda(\mbfs{\beta}_0),
\]
or equivalently, because $\mbfs{\beta}_0^\top \mbfs{y}_0 = \|\mbfs{y}_0\|$, that
\[
\Gamma(\mbfs{\beta}_0)^\top \mbfs{x}_0 > \|\mbfs{y}_0\| - \Lambda(\mbfs{\beta}_0).
\]
If $\mbfs{\beta}_0 \in G$, then $\Lambda(\mbfs{\beta}_0) =0 $.
By the Cauchy-Schwarz Inequality, we have $ \Gamma(\mbfs{\beta}_0)^\top \mbfs{x}_0  \le \|\mbfs{x}_0\|$.
Thus, $\|\mbfs{y}_0\| < \|\mbfs{x}_0\|$, which contradicts that $(\mbfs{x}_0, \mbfs{y}_0) \in Q_g \subseteq Q_h$.
Hence, $\mbfs{\beta}_0 \not\in G$.
Consequently,  $\Lambda(\mbfs{\beta}_0) > 0$.

As $(\mbfs{x}_0, \mbfs{y}_0)\in H$, we can write
\[
\mbfs{a}{}^\top \mbfs{x}_0 = -1 - \mbfs{d}{}^\top \mbfs{\beta}_0 \cdot \|\mbfs{y}_0\|.
\]
Thus,
\begin{align}
(\Gamma(\mbfs{\beta}_0)-\Lambda(\mbfs{\beta}_0)\cdot\mbfs{a})^\top \mbfs{x}_0 &= \Gamma(\mbfs{\beta}_0)^\top \mbfs{x}_0  - \Lambda(\mbfs{\beta}_0) \cdot\mbfs{a}{}^\top \mbfs{x}_0 \nonumber \\
& > \|\mbfs{y}_0\|-\Lambda(\mbfs{\beta}_0) - \Lambda(\mbfs{\beta}_0) \cdot \left(-1 - \mbfs{d}{}^\top \mbfs{\beta}_0\cdot\|\mbfs{y}_0\|\right)\nonumber \\
 & = \|\mbfs{y}_0\| (1 + \Lambda(\mbfs{\beta}_0)\cdot\mbfs{d}{}^\top \mbfs{\beta}_0 ). \label{rbeta0_last}
\end{align}

Additionally, by the Cauchy-Schwarz Inequality, we have
\[
(\Gamma(\mbfs{\beta}_0)-\Lambda(\mbfs{\beta}_0)\cdot \mbfs{a})^\top \mbfs{x}_0 \leq  \|\Gamma(\mbfs{\beta}_0)-\Lambda(\mbfs{\beta}_0)\cdot\mbfs{a}\| \cdot\|\mbfs{x}_0\| \leq \|\Gamma(\mbfs{\beta}_0)-\Lambda(\mbfs{\beta}_0)\cdot\mbfs{a}\|\cdot \|\mbfs{y}_0\|.
\]
Combining this with \eqref{rbeta0_last}, we get
\begin{equation}\label{contradictionineq}
\|\Gamma(\mbfs{\beta}_0)-\Lambda(\mbfs{\beta}_0)\cdot\mbfs{a}\| > 1 + \Lambda(\mbfs{\beta}_0 ) \cdot \mbfs{d}{}^\top \mbfs{\beta}_0.
\end{equation}

On the other hand, we have
\begin{align*}
  \|\Gamma(\mbfs{\beta}_0) - \Lambda(\mbfs{\beta}_0)\cdot \mbfs{a} \|^2
  &= 1 - 2\Lambda(\mbfs{\beta}_0)\cdot \mbfs{a}{}^\top \Gamma(\mbfs{\beta}_0)  + \Lambda(\mbfs{\beta}_0)^2 \\
  &= \left(\frac{1 + 2 \mbfs{a}{}^\top \Gamma(\mbfs{\beta}_0) \cdot \mbfs{d}{}^\top \mbfs{\beta}_0 + (\mbfs{d}{}^\top \mbfs{\beta}_0)^2}{1-(\mbfs{d}{}^\top \mbfs{\beta}_0)^2}\right)^2.
\end{align*}
As $1 - (\mbfs{a}{}^\top \Gamma(\mbfs{\beta}_0))^2 \ge 0$, the numerator of the latter fraction satisfies
\[ 
\begin{array}{rl}
& 1 + 2 \mbfs{a}{}^\top \Gamma(\mbfs{\beta}_0)\cdot\mbfs{d}{}^\top \mbfs{\beta}_0 + (\mbfs{d}{}^\top \mbfs{\beta}_0)^2\\[.15cm]
= & 1 - (\mbfs{a}{}^\top \Gamma(\mbfs{\beta}_0))^2 + (\mbfs{a}{}^\top \Gamma(\mbfs{\beta}_0) + \mbfs{d}{}^\top \mbfs{\beta}_0)^2 \geq 0. 
\end{array}
\]
Therefore,
\[ 
\|\Gamma(\mbfs{\beta}_0) - \Lambda(\mbfs{\beta}_0)\cdot \mbfs{a} \| = \frac{1 + 2 \mbfs{a}{}^\top \Gamma(\mbfs{\beta}_0) \cdot \mbfs{d}{}^\top \mbfs{\beta}_0 + (\mbfs{d}{}^\top \mbfs{\beta}_0)^2}{1-(\mbfs{d}{}^\top \mbfs{\beta}_0)^2}. 
\]
Finally, we see that
\begin{align*}
& \|\Gamma(\mbfs{\beta}_0) - \Lambda(\mbfs{\beta}_0)\cdot \mbfs{a} \| - \Lambda(\mbfs{\beta}_0 ) \cdot \mbfs{d}{}^\top \mbfs{\beta}_0 \\
=\, &  \frac{1 + 2 \mbfs{a}{}^\top \Gamma(\mbfs{\beta}_0) \cdot \mbfs{d}{}^\top \mbfs{\beta}_0 + (\mbfs{d}{}^\top \mbfs{\beta}_0)^2 - 2 \mbfs{d}{}^\top \mbfs{\beta}_0\cdot(\mbfs{a}{}^\top \Gamma(\mbfs{\beta}_0) + \mbfs{d}{}^\top \mbfs{\beta}_0) }{1-(\mbfs{d}{}^\top \mbfs{\beta}_0)^2}\\
=\, & 1.
\end{align*}
However, this contradicts \eqref{contradictionineq}. 
Thus, $C^+_\Gamma \cap H$ is $Q_g$-free.
\end{cpf}

\begin{cm}\label{cm:Qgstrict}
We have $C_\Gamma \cap H \subsetneq C^+_\Gamma \cap H$.
\end{cm}
\begin{cpf}
Take $(\mbfs{x}_0, \mbfs{y}_0) \in C_\Gamma \cap H$ and consider $(\mbfs{x}_0 - \rho \cdot\mbfs{a}, \mbfs{y}_0 + \rho \cdot \mbfs{d})$
for some $\rho \geq 0$.
Consider a valid inequality for $C_\Gamma$, $-\Gamma(\mbfs{\beta})^\top \mbfs{x}+ \mbfs{\beta}{}^\top \mbfs{y} \le 0$.
Evaluating this inequality at $(\mbfs{x}_0 - \rho\cdot \mbfs{a}, \mbfs{y}_0 + \rho \cdot\mbfs{d})$, we obtain
\[-\Gamma(\mbfs{\beta})^\top \mbfs{x}_0 + \rho\cdot\mbfs{a}{}^\top \Gamma(\mbfs{\beta}) + \mbfs{\beta}{}^\top \mbfs{y}_0  + \rho\cdot \mbfs{d}{}^\top \mbfs{\beta} \le 0 \]
which is
\begin{equation}\label{eq:modifiedineq}
\mbfs{\beta}{}^\top \mbfs{y}_0  + \rho \cdot (\mbfs{a}{}^\top \Gamma(\mbfs{\beta}) + \mbfs{d}{}^\top \mbfs{\beta}) \le \Gamma(\mbfs{\beta})^\top \mbfs{x}_0. 
\end{equation}
As $G\subsetneq D^m$, there exists $\mbfs{\beta}$ such that $\mbfs{a}{}^\top \Gamma(\mbfs{\beta}) + \mbfs{d}{}^\top \mbfs{\beta}> 0$. 
Thus, from \eqref{eq:modifiedineq}, there is $\rho \geq 0$ such that $(\mbfs{x}_0 - \rho\cdot \mbfs{a}, \mbfs{y}_0 + \rho \cdot\mbfs{d}) \not\in C_\Gamma$.
Let
\[
\rho_0 = \inf\{\rho\, :\, (\mbfs{x}_0 - \rho\cdot \mbfs{a}, \mbfs{y}_0 + \rho \cdot\mbfs{d}) \not\in C_\Gamma,\, \rho \geq 0\}.
\]
It must hold $(\mbfs{x}_0 - \rho_0 \cdot\mbfs{a}, \mbfs{y}_0 + \rho_0 \cdot\mbfs{d}) \in C_\Gamma$ as otherwise one could decrease $\rho_0$.

Define the numbers
\[
\begin{array}{rcl}
\delta &:=& \displaystyle\min\left\{ \frac{2}{1-(\mbfs{d}{}^\top \mbfs{\beta})^2}:\ \mbfs{\beta}\in D^m \right\}\\[.35cm]
\lambda &:=& \displaystyle\frac{1}{1 + (\rho_0 + \delta) (1 - \|\mbfs{d}\|^2)} .
\end{array}
\]
The number $\delta > 0$ is well-defined because $\|\mbfs{d}\| < 1$.
Also, $0 < \lambda < 1$.

We claim that
\[ 
\lambda \cdot (\mbfs{x}_0 - (\rho_0 + \delta)\cdot \mbfs{a}, \mbfs{y}_0 + (\rho_0 + \delta) \cdot\mbfs{ d}) \in (C^+_\Gamma \cap H) \setminus  (C_\Gamma \cap H).
\]
By definition of $\rho_0$, we have $(\mbfs{x}_0 - (\rho_0 + \delta)\cdot \mbfs{a}, \mbfs{y}_0 + (\rho_0 + \delta)\cdot \mbfs{ d})\not\in C_\Gamma$.
Thus, $\lambda \cdot (\mbfs{x}_0 - (\rho_0 + \delta) \cdot\mbfs{a}, \mbfs{y}_0 + (\rho_0 + \delta)\cdot \mbfs{ d})\not\in C_\Gamma$ because $C_\Gamma$ is a cone and $\lambda > 0$.

Using $(\mbfs{x}_0,\mbfs{y}_0) \in H$, we see that
\begin{align*}
& \lambda \cdot(\mbfs{a}{}^\top (\mbfs{x}_0 - (\rho_0 + \delta)\cdot \mbfs{a}) + \mbfs{d}{}^\top (\mbfs{y}_0 + (\rho_0 + \delta)\cdot \mbfs{ d})) \\
=\,& \lambda \cdot (-1 + (\rho_0 + \delta) \cdot(\|\mbfs{d}\|^2 - 1) )\\
=\, & -1.
\end{align*}
Thus, $\lambda \cdot (\mbfs{x}_0 - (\rho_0 + \delta)\cdot \mbfs{a}, \mbfs{y}_0 + (\rho_0 + \delta) \cdot\mbfs{ d})\in H$.

Finally, for membership in $C_\Gamma^+ $, note that
\begin{align*}
 & - \Gamma(\mbfs{\beta})^\top (\mbfs{x}_0 - (\rho_0 + \delta) \cdot\mbfs{a}) + \mbfs{\beta}{}^\top(\mbfs{y}_0 + (\rho_0 + \delta) \cdot\mbfs{ d}) \lambda - \Lambda(\mbfs{\beta}) \nonumber\\
\leq\, & \delta \lambda \cdot \left( \Gamma(\mbfs{\beta})^\top  \mbfs{a} + \mbfs{\beta}{}^\top \mbfs{ d}\right)  - \Lambda(\mbfs{\beta})\\
\leq\, & \delta \lambda \cdot \max\{0, \Gamma(\mbfs{\beta})^\top  \mbfs{a} + \mbfs{\beta}{}^\top \mbfs{ d}\} - 2 \cdot\frac{\max\{0,\mbfs{a}{}^\top \Gamma(\mbfs{\beta}) + \mbfs{d}{}^\top \mbfs{\beta}\}}{1-(\mbfs{d}{}^\top \mbfs{\beta})^2}\\
\leq\, &  \frac{2}{1-(\mbfs{d}{}^\top \mbfs{\beta})^2} \cdot \max\{0,\Gamma(\mbfs{\beta})^\top  \mbfs{a} + \mbfs{\beta}{}^\top \mbfs{ d}\} - 2 \cdot \frac{\max\{0,\mbfs{a}{}^\top \Gamma(\mbfs{\beta}) + \mbfs{d}{}^\top \mbfs{\beta}\}}{1-(\mbfs{d}{}^\top \mbfs{\beta})^2} \\
=\, & 0, \nonumber
\end{align*}
where the first inequality follows from $(\mbfs{x}_0 - \rho_0 \cdot\mbfs{a}, \mbfs{y}_0 + \rho_0 \cdot\mbfs{d}) \in C_\Gamma$, the second follows from the definition of $\Lambda(\mbfs{\beta})$, and the third follows from the definition of $\delta$ and $\lambda < 1$.
Thus, $\lambda \cdot (\mbfs{x}_0 - (\rho_0 + \delta)\cdot \mbfs{a}, \mbfs{y}_0 + (\rho_0 + \delta) \cdot\mbfs{ d})\in C^+_\Gamma \cap H$. 
\end{cpf}

By Claim \ref{cm:Qgfree}, $C^+_\Gamma \cap H$ is $Q_g$-free.
By Claim \ref{cm:Qgstrict}, $C_\Gamma \cap H \subsetneq C^+_\Gamma \cap H$.
Thus, $C_\Gamma \cap H$ is not a maximal $Q_g$-free set.
\end{proof}

\section{A sufficient condition for maximal $Q_g$-free sets}\label{secMaxH}

Mu\~{n}oz et al. developed a sufficient condition for full-dimensional maximal $S$-free sets for general $S$ in \cite[Theorem 1.4]{munoz2025characterization}.
We could project $C_{\Gamma} \cap H$ into an $(n+m-1)$-dimensional space and apply their result there, where $C_{\Gamma} \cap H$ is full-dimensional.
However, in light of Lemma \ref{sliceofCgamma}, it will be convenient to stay in $\mathbb{R}^n\times \mathbb{R}^m$ and maintain a perspective of $Q_g$-free sets as slices of $Q_h$-free sets; this preference prevents us from directly using their result. 
Instead, we modify their criterion to handle `full-dimensionality' in the $(n+m-1)$-dimensional hyperplane $H$.
In this section, we restate the necessary definitions and prove the necessary results needed in our modification.

We begin with the definition of an exposing sequence, which is the main ingredient in the characterization of \cite{munoz2025characterization}.
\begin{definition}[Exposing sequence]\label{def:exposingseq}
Let $C \subseteq \mbb{R}^d$ be a convex set, and let $\mbfs{\alpha}^\top \mbfs{x} \le \overline{\alpha}$, with $\mbfs{\alpha} \neq  \mbfs{0}$, be a valid inequality for $C$.
  A sequence $(\mbfs{x}_t)_{t=1}^\infty$ in $\mbb{R}^d$ is an {\bf exposing sequence for $ \mbfs{\alpha}{}^\top \mbfs{x} \le \overline{\alpha}$ with respect to $\bm C$} if 
  \[
  \lim_{t\to\infty}(\mbfs{\delta}_t, {\varepsilon_t}) = (\mbfs{\alpha}, \overline{\alpha})
  \]
  for every sequence $((\mbfs{\delta}_t, \varepsilon_t))_{t=1}^\infty$ in ${\mbb{R}}^d \times \mbb{R}$ such that $\|\mbfs{\delta}_t\| = \|\mbfs{\alpha}\|$, $\mbfs{\delta}_t{}^\top \mbfs{x} \le \varepsilon_t$ is a valid inequality for $C$, and $\mbfs{\delta}_t{}^\top \mbfs{x}_t \ge \varepsilon_t$ for each $t$.
\end{definition}

The following is our maximality criterion.

\begin{theorem}\label{thm:maximalitycriterion}
Let $\Gamma:D^m \to D^n$ be such that $C_\Gamma$ is a full-dimensional maximal $Q_h$-free set and $C_\Gamma \cap H$ is full-dimensional.
Assume there exists a set $I\subseteq D^m$ such that the following hold:
\begin{enumerate}[leftmargin = *]
\item The set $C^I_\Gamma \cap H $ is $Q_g$-free.
\item For each $\mbfs{\beta} \in I $, the inequality $-\Gamma(\mbfs{\beta}){}^\top\mbfs{x} + \mbfs{\beta}{}^\top \mbfs{y} \le 0 $ has an exposing sequence $((\mbfs{x}_t, \mbfs{y}_t))_{t=1}^\infty$ in $Q_g$ with respect to $C^I_\Gamma$.
\end{enumerate}
Then, $C^I_\Gamma \cap H$ is a maximal $Q_g$-free set.
\end{theorem}

\begin{proof}

Assume to the contrary that $C^I_\Gamma \cap H$ is not maximal. 
Thus, there exists a set $K'\subseteq H$ that is $Q_g$-free and $C^I_\Gamma \cap H \subsetneq K'$. 
Given that the inclusion is strict, there must exist some $\overline{\mbfs{\beta}}\in I$ such that the inequality
\(
-\Gamma(\overline{\mbfs{\beta}})^\top \mbfs{x} + \overline{\mbfs{\beta}}{}^\top \mbfs{y} \leq 0
\)
is not valid for $K'$.

By assumption, the inequality $-\Gamma(\overline{\mbfs{\beta}})^\top \mbfs{x} + \overline{\mbfs{\beta}}{}^\top \mbfs{y} \leq 0$ has an exposing sequence $((\mbfs{x}_t, \mbfs{y}_t))_{t=1}^\infty$ in $Q_g$ with respect to $C^I_\Gamma$.
The set $K'$ is $Q_g$-free.
Thus, for each $t \in \mbb{N}$, there exists a vector $(\mbfs{e}_t, \mbfs{f}_t) \in \mbb{R}^n \times \mbb{R}^m$ such that $(\mbfs{e}_t, \mbfs{f}_t)$ and $(\mbfs{a}, \mbfs{d})$ are linearly independent (in fact, orthogonal, because $(\mbfs{x}_t, \mbfs{y}_t)$ and $K'$ are in $H$), the inequality $\mbfs{e}_t{}^\top\mbfs{x}+\mbfs{f}_t{}^\top \mbfs{y} \le \tau_t$ is valid for $K'$, and $\mbfs{e}_t{}^\top\mbfs{x}_t+ \mbfs{f}_t{}^\top \mbfs{y}_t \ge \tau_t$, i.e., the inequality separates $(\mbfs{x}_t, \mbfs{y}_t)$ from $K'$.

The inequality $\mbfs{e}_t{}^\top\mbfs{x}+ \mbfs{f}_t{}^\top \mbfs{y} \le \tau_t$ is also valid for $C^I_\Gamma \cap H \subsetneq K'$.
We can use Lemma \ref{lem:Wavy2} with $(F, K, \mbfs{f}, \mbfs{e}, \tau) = (H, C^I_{\Gamma}, (\mbfs{a}, \mbfs{d}), (\mbfs{e}_t, \mbfs{f}_t), \tau_t)$ to write 
\[
(\mbfs{e}_t, \mbfs{f}_t) =  (\mbfs{g}_t, \mbfs{b}_t)- \gamma_t \cdot(\mbfs{a}, \mbfs{d}),
\]
where $\gamma_t \leq \tau_t $, $(\mbfs{g}_t, \mbfs{b}_t)$ is nonzero, and $\mbfs{g}_t{}^\top\mbfs{x}+\mbfs{b}_t{}^\top\mbfs{y} \leq 0$ is valid for $C^I_\Gamma$.
Given that $\mbfs{e}_t{}^\top\mbfs{x}_t+ \mbfs{f}_t{}^\top \mbfs{y}_t \ge \tau_t$ and $(\mbfs{x}_t, \mbfs{y}_t)\in Q_g \subseteq H$, we have
\begin{equation}\label{eqCr1}
\mbfs{g}_t{}^\top\mbfs{x}_t+\mbfs{b}_t{}^\top\mbfs{y}_t  =  (\mbfs{e}_t{}^\top\mbfs{x}_t+\mbfs{f}_t{}^\top\mbfs{y}_t)  + \gamma_t \cdot (\mbfs{a}{}^\top\mbfs{x}_t+\mbfs{d}{}^\top\mbfs{y}_t ) \ge \tau_t - \gamma_t \geq 0.
\end{equation}
As $\mbfs{g}_t{}^\top\mbfs{x}+\mbfs{b}_t{}^\top\mbfs{y} \leq 0$ is valid for $C^I_\Gamma$, it follows from~\eqref{eqCr1} that
\begin{equation}\label{eqCr2}
\text{$\mbfs{g}_t{}^\top\mbfs{x}+\mbfs{b}_t{}^\top\mbfs{y} \leq \mbfs{g}_t{}^\top\mbfs{x}_t+\mbfs{b}_t{}^\top\mbfs{y}_t$ is valid for $C^I_\Gamma$.}
\end{equation}
Note that~\eqref{eqCr1} and~\eqref{eqCr2} hold if we multiply $(\mbfs{g}_t, \mbfs{b}_t)$ by $\sqrt{2} / \|(\mbfs{g}_t, \mbfs{b}_t)\| > 0$; this scales $(\mbfs{g}_t, \mbfs{b}_t)$ to have the same norm as $(-\Gamma(\overline{\mbfs{\beta}}),\overline{\mbfs{\beta}})$, that is,
\begin{equation}\label{eqCr3}
\left\|\frac{\sqrt{2}}{\|(\mbfs{g}_t, \mbfs{b}_t)\|} \cdot (\mbfs{g}_t, \mbfs{b}_t) \right\| = \left\|(-\Gamma(\overline{\mbfs{\beta}}),\overline{\mbfs{\beta}}) \right\|.
\end{equation}

Using~\eqref{eqCr1}, \eqref{eqCr2}, and~\eqref{eqCr3}, we can apply the fact that $(\mbfs{x}_t,\mbfs{y}_t)$ is an exposing sequence for $-\Gamma(\overline{\mbfs{\beta}})^\top \mbfs{x} + \overline{\mbfs{\beta}}{}^\top \mbfs{y} \leq 0$ with respect to $C^I_\Gamma$; this implies that
\begin{align}\label{explim1}
(-\Gamma(\overline{\mbfs{\beta}}),\overline{\mbfs{\beta}}) = &  \lim_{t\to\infty} \frac{\sqrt{2}}{\|(\mbfs{g}_t, \mbfs{b}_t)\|} \cdot (\mbfs{g}_t, \mbfs{b}_t)\\
= & \lim_{t \to\infty} \frac{\sqrt{2}}{ \|(\mbfs{g}_t, \mbfs{b}_t)\| } \cdot \left( (\mbfs{e}_t, \mbfs{f}_t) + \gamma_t \cdot(\mbfs{a}, \mbfs{d}) \right),\nonumber
\end{align}
and
\begin{equation}\label{explim2}
\lim_{t\to\infty} \frac{\sqrt{2}}{ \|(\mbfs{g}_t, \mbfs{b}_t)\| }\cdot\left(\mbfs{g}_t{}^\top\mbfs{x}_t+\mbfs{b}_t{}^\top  \mbfs{y}_t\right)  = 0.
\end{equation}

We can use $\mbfs{e}_t{}^\top\mbfs{x}_t+ \mbfs{f}_t{}^\top \mbfs{y}_t \ge \tau_t$ and $(\mbfs{x}_t, \mbfs{y}_t)\in H$ to conclude that
\begin{align}\label{exp:lim1}
0\leq &\, \frac{\sqrt{2}}{ \|(\mbfs{g}_t, \mbfs{b}_t)\| } \cdot(\tau_t - \gamma_t) \nonumber\\
\leq &\, \frac{\sqrt{2}}{ \|(\mbfs{g}_t, \mbfs{b}_t)\| }\cdot( \mbfs{e}_t{}^\top\mbfs{x}_t+ \mbfs{f}_t{}^\top \mbfs{y}_t  - \gamma_t  ) \nonumber\\
= &\, \frac{\sqrt{2}}{ \|(\mbfs{g}_t, \mbfs{b}_t)\| }\cdot\left(( \mbfs{e}_t+\gamma_t \cdot \mbfs{a}){}^\top\mbfs{x}_t+ (\mbfs{f}_t+\gamma_t\cdot\mbfs{d}){}^\top \mbfs{y}_t   \right) \nonumber \\
= &\, \frac{\sqrt{2}}{ \|(\mbfs{g}_t, \mbfs{b}_t)\| }\cdot\left(\mbfs{g}_t{}^\top\mbfs{x}_t+\mbfs{b}_t{}^\top  \mbfs{y}_t\right) .
\end{align}

Thus, by~\eqref{explim2}, we have
\begin{equation}\label{lemma:rhslimit}
\lim_{t\to\infty}
\frac{\sqrt{2}}{ \|(\mbfs{g}_t, \mbfs{b}_t)\| } \cdot (\tau_t - \gamma_t) = 0.
\end{equation}

Finally, for each $t \in \mbb{N}$, the inequality
\begin{equation}\label{eqFinalLimit}
\frac{\sqrt{2}}{ \|(\mbfs{g}_t, \mbfs{b}_t)\| } \cdot\left( (\mbfs{e}_t+\gamma_t \cdot \mbfs{a}){}^\top\mbfs{x}_t+ (\mbfs{f}_t+\gamma_t\cdot\mbfs{d}){}^\top \mbfs{y}_t   \right)\leq \frac{\sqrt{2}}{ \|(\mbfs{g}_t, \mbfs{b}_t)\| } \cdot (\tau_t - \gamma_t)
\end{equation}
is valid for $K'$ because it is an aggregation of a nonnegative multiple of $\mbfs{e}_t{}^\top\mbfs{x}+\mbfs{f}_t{}^\top \mbfs{y} \le \tau_t$ and a multiple of $\mbfs{a}{}^\top\mbfs{x} + \mbfs{d}{}^\top\mbfs{y} = -1$. 
By \eqref{explim1} and \eqref{lemma:rhslimit}, the limit of \eqref{eqFinalLimit} as $t\to\infty$ is $-\Gamma(\overline{\mbfs{\beta}})^\top \mbfs{x} + \overline{\mbfs{\beta}}{}^\top \mbfs{y} \leq 0$; this inequality is valid for $K'$.
However, this contradicts that $-\Gamma(\overline{\mbfs{\beta}})^\top \mbfs{x} + \overline{\mbfs{\beta}}{}^\top \mbfs{y} \leq 0$ is not valid for $K'$.
\end{proof}

Theorem \ref{thm:maximalitycriterion} uses exposing sequences in $Q_g\subseteq H$ to expose inequalities of a relaxation of $C_\Gamma$.
This is a subtle analysis of the inhomogeneous $(n+m-1)$-dimensional slice within the homogeneous $(n+m)$-dimensional space. 

\subsection{Common exposing sequences}

Theorems \ref{thm:easycase} and \ref{thm:hardcase} consider two different cases that result in two different
characterizations involving $C^G_{\Gamma}$.
Nevertheless we can unify part of their proofs. In both cases we work with an alternative description of $C^G_{\Gamma}\cap H$ of the form $C_\Gamma^J \cap H$, for a compact set $J\subseteq D^m$.
The choices of $J$ differ in both cases, but they share a common subset for which we design an exposing sequence.

The next lemma reduces $J$ in the description of $C_\Gamma^J\cap H$ to the subset generating the exposed rays of $\cone(\{(-\Gamma(\mbfs{\beta}), \mbfs{\beta}) \,: \, \mbfs{\beta} \in J\})$.

\begin{lemma}\label{lemma:exposedrays}
Let $\Gamma:D^m \to D^n$ be such that $C_\Gamma$ is a maximal $Q_h$-free set.
Let $J\subseteq D^m$ a compact set, and let $I \subseteq J$ be the set of $\overline{\mbfs{\beta}}\in J$ such that $(-\Gamma(\overline{\mbfs{\beta}}), \overline{\mbfs{\beta}})$ generates an exposed ray of $\cone(\{(-\Gamma(\mbfs{\beta}), \mbfs{\beta}) \,: \, \mbfs{\beta} \in J\})$. 
Then, $C^J_\Gamma =  C^I_\Gamma$.
\end{lemma}
\begin{proof}
The inclusion $C^J_\Gamma \subseteq  C^I_\Gamma$ holds because $I \subseteq J$.
Set $B :=\{(-\Gamma(\mbfs{\beta}),\mbfs{\beta}) \,:\ \mbfs{\beta}\in J\}$ and $P:=\cone(B)$. 
The maximality of $C_\Gamma$ ensures $\mbfs{0}\not\in \conv(\{\Gamma(\mbfs{\beta}) \,:\ \mbfs{\beta}\in D^m\})$ by Theorem~\ref{thmHQuad}, which implies $(\mbfs{0},\mbfs{0})\not\in \conv(B)$.
Additionally, $B$ is compact. Hence, $P$ is a closed pointed cone. 
The set $I$ contains one representative of each exposed ray of $P$, so $P = \cl(\cone(\{(-\Gamma(\mbfs{\beta}), \mbfs{\beta}) \,: \, \mbfs{\beta} \in I\} ))$; see, e.g., Theorem 18.7 in \cite{R1970}.
Thus, each valid inequality for $C^J_{\Gamma}$ is implied by inequalities defining $C^I_{\Gamma}$.
Hence, $C^J_\Gamma \supseteq  C^I_\Gamma$.
\end{proof}

Recall $G^{\circ}$ defined in \eqref{defnGcirc}.
The next result contains the aforementioned family of inequalities appearing in both characterizations: those indexed by $\cl(G^\circ) \cap I $.
Importantly, the result does not assume an ordering of $\|\mbfs{a}\|$ and $\|\mbfs{d}\|$.

\begin{lemma}\label{lemma:exposing_common}
Let $\Gamma : D^m \to D^n$ be such that $C_{\Gamma}$ is a full-dimensional maximal $Q_h$-free set.
Let $J\subseteq G$ be compact and $I$ be the set of vectors $\overline{\mbfs{\beta}}\in J$ such that $(-\Gamma(\overline{\mbfs{\beta}}), \overline{\mbfs{\beta}})$ generates an exposed ray of $\cone(\{(-\Gamma(\mbfs{\beta}), \mbfs{\beta}) \,: \, \mbfs{\beta} \in J\})$.
The following holds: For each $\overline{\mbfs{\beta}} \in \cl(G^\circ) \cap I $, the inequality $-\Gamma(\overline{\mbfs{\beta}}){}^\top\mbfs{x} + \overline{\mbfs{\beta}}{}^\top \mbfs{y} \le 0$ has an exposing sequence in $Q_g$ with respect to $C_\Gamma^I$.
\end{lemma}

\begin{proof}
Let $\overline{\mbfs{\beta}} \in \cl(G^\circ) \cap I $.
The vector $(-\Gamma(\overline{\mbfs{\beta}}), \overline{\mbfs{\beta}})$ generates an exposed ray of of $\cone(\{(-\Gamma(\mbfs{\beta}), \mbfs{\beta}) \,: \, \mbfs{\beta} \in J\})$.
Thus, there exists a vector $(\overline{\mbfs{x}},\overline{\mbfs{y}}) \in \mbb{R}^n \times \mbb{R}^m$ such that
\[
-\Gamma(\overline{\mbfs{\beta}})^\top \overline{\mbfs{x}} + \overline{\mbfs{\beta}}{}^\top \overline{\mbfs{y}} = 0 > -\Gamma({\mbfs{\beta}})^\top \overline{\mbfs{x}} + {\mbfs{\beta}}^\top \overline{\mbfs{y}} \qquad \forall ~\mbfs{\beta} \in J \setminus \{\overline{\mbfs{\beta}}\}.
\]
As $\overline{\mbfs{\beta}}\in \cl(G^\circ)$, there exists a sequence $( {\mbfs{\beta}_t})_{t=1}^\infty$ in $G^{\circ}$ such that
\[
\lim_{t\to\infty} \mbfs{\beta}_t =  \overline{\mbfs{\beta}}.
\]
Set $\alpha_t :=  \Gamma(\mbfs{\beta}_t)^\top \overline{\mbfs{x}} - \mbfs{\beta}_t^\top \overline{\mbfs{y}}$ for each $t \in \mbb{N}$.
We have $\alpha_t \geq 0$ and $\lim_{t\to\infty} \alpha_t = 0$.

For each $t \in \mbb{N}$, define
\[
\varepsilon_t := \min\left\{- \left(\mbfs{a}{}^\top \Gamma({\mbfs{\beta}_t}) + \mbfs{d}{}^\top {\mbfs{\beta_t}}\right), \frac{1}{t}\right\}.
 \]
Using $\mbfs{\beta}_t \in G^{\circ}$, we see that $0 < \varepsilon_t$  and $\lim_{t\to\infty} \varepsilon_t = 0$.

For each $t \in \mbb{N}$, define the vectors
\[
\begin{array}{rcl}
\widehat{\mbfs{x}}_t & :=&\Gamma(\mbfs{\beta}_t)+\varepsilon_t^2\cdot\overline{\mbfs{x}} \\[.1cm]
\widehat{\mbfs{y}}_t &:=&\left(1+\alpha_t\varepsilon_t^2+\varepsilon_t^3\right)\cdot\mbfs{\beta}_t+\varepsilon_t^2\cdot \overline{\mbfs{y}}.
\end{array}
\]

Observe that
\begin{align}
\mbfs{a}{}^\top \widehat{\mbfs{x}}_t + \mbfs{d}{}^\top \widehat{\mbfs{y}}_t 
&= \mbfs{a}{}^\top \Gamma(\mbfs{\beta}_t) 
+ \varepsilon_t^2 \mbfs{a}^\top \overline{\mbfs{x}}  
+ \left(1+\alpha_t\varepsilon_t^2+\varepsilon_t^3\right) \mbfs{d}{}^\top \mbfs{\beta}_t  
+ \varepsilon_t^2 \mbfs{d}{}^\top\overline{\mbfs{y}} \nonumber\\
&\leq -\varepsilon_t 
+ \varepsilon_t^2 \mbfs{a}^\top \overline{\mbfs{x}}  
+ (\alpha_t\varepsilon_t^2+\varepsilon_t^3)\, \mbfs{d}{}^\top \mbfs{\beta}_t  
+ \varepsilon_t^2 \mbfs{d}{}^\top\overline{\mbfs{y}}. \label{ahatbound}
\end{align}
Hence, $(1/\varepsilon_t)(\mbfs{a}{}^\top \widehat{\mbfs{x}}_t + \mbfs{d}{}^\top \widehat{\mbfs{y}}_t )\leq -1 + o(1)$, where $o(1)$ is a term that goes to 0 when $t\to \infty$.
Thus, $\mbfs{a}{}^\top \widehat{\mbfs{x}}_t + \mbfs{d}{}^\top \widehat{\mbfs{y}}_t < 0$ for sufficiently large $t$.
After taking a subsequence, we assume $\mbfs{a}{}^\top \widehat{\mbfs{x}}_t + \mbfs{d}{}^\top \widehat{\mbfs{y}}_t < 0$ for all $t$.

We now verify that $(\widehat{\mbfs{x}}_t,\widehat{\mbfs{y}}_t)\in Q_h$. 
We have
\begin{align*}
\hspace*{2cm} \|\widehat{\mbfs{x}}_t\|^2
&=
1 + 2\varepsilon_t^2 \Gamma(\mbfs{\beta}_t)^\top \overline{\mbfs{x}}
+ \varepsilon_t^4 \|\overline{\mbfs{x}}\|^2,
\intertext{and}
\|\widehat{\mbfs{y}}_t\|^2
&=
\left(1+\alpha_t\varepsilon_t^2+\varepsilon_t^3\right)^2
+ 2\varepsilon_t^2 \left(1+\alpha_t\varepsilon_t^2+\varepsilon_t^3\right)\mbfs{\beta}_t^\top \overline{\mbfs{y}}
+ \varepsilon_t^4 \|\overline{\mbfs{y}}\|^2 \\
&=
1 
+ \alpha_t^2\varepsilon_t^4 + \varepsilon_t^6 + 2\alpha_t\varepsilon_t^2 + 2\varepsilon_t^3  + 2\alpha_t\varepsilon_t^5  \\
&\quad
+ 2\varepsilon_t^2 \mbfs{\beta}_t^\top \overline{\mbfs{y}}
+ 2\alpha_t\varepsilon_t^4 \mbfs{\beta}_t^\top \overline{\mbfs{y}}
+ 2\varepsilon_t^5 \mbfs{\beta}_t^\top \overline{\mbfs{y}}
+ \varepsilon_t^4 \|\overline{\mbfs{y}}\|^2.
\end{align*}
Using $\alpha_t =  \Gamma(\mbfs{\beta}_t)^\top \overline{\mbfs{x}} - \mbfs{\beta}_t^\top \overline{\mbfs{y}}$, we see that
\begin{align*}
\|\widehat{\mbfs{x}}_t\|^2 - \|\widehat{\mbfs{y}}_t\|^2
= &
2\varepsilon_t^2 \alpha_t
- \left(\alpha_t^2\varepsilon_t^4 + \varepsilon_t^6 + 2\alpha_t\varepsilon_t^2 + 2\varepsilon_t^3  + 2\alpha_t\varepsilon_t^5 \right) \\
&
- 2\alpha_t\varepsilon_t^4 \mbfs{\beta}_t^\top \overline{\mbfs{y}}
- 2\varepsilon_t^5 \mbfs{\beta}_t^\top \overline{\mbfs{y}}
+ \varepsilon_t^4 \left(\|\overline{\mbfs{x}}\|^2 - \|\overline{\mbfs{y}}\|^2\right)\\
= & -2\varepsilon_t^3 +O(\varepsilon_t^4).
\end{align*}

Hence, $(1/\varepsilon_t^3)(\|\widehat{\mbfs{x}}_t\|^2 - \|\widehat{\mbfs{y}}_t\|^2) = -2 + o(1)$.
Thus, $\|\widehat{\mbfs{x}}_t\|^2 - \|\widehat{\mbfs{y}}_t\|^2 < 0$ for all sufficiently large $t$.
After taking a subsequence, we assume $\|\widehat{\mbfs{x}}_t\|^2 - \|\widehat{\mbfs{y}}_t\|^2 < 0$ for all $t$.
Hence, $(\widehat{\mbfs{x}}_t,\widehat{\mbfs{y}}_t)\in Q_h$ for all $t$.

For each $t$, define the vector
\[
(\mbfs{x}_t,\mbfs{y}_t ) := \frac{-1}{\mbfs{a}{}^\top \widehat{\mbfs{x}}_t + \mbfs{d}{}^\top \widehat{\mbfs{y}}_t }\cdot ( \widehat{\mbfs{x}}_t,\widehat{\mbfs{y}}_t). 
\]
We have $(\mbfs{x}_t,\mbfs{y}_t )\in H \cap Q_h = Q_g$.
In what remains, we argue that $((\mbfs{x}_t, \mbfs{y}_t))_{t=1}^\infty $ is an exposing sequence for $-\Gamma(\overline{\mbfs{\beta}})^\top \mbfs{x} + \overline{\mbfs{\beta}}{}^\top \mbfs{y} \leq 0$ with respect to $C^I_\Gamma$.

Take any sequence $((-\mbfs{g}_t, \mbfs{b}_t, r_t))_{t=1}^{\infty}$ in $\mbb{R}^n\times \mbb{R}^m\times \mbb{R}$ such that, for each $t$, we have $\|(-\mbfs{g}_t, \mbfs{b}_t)\| = \|(-\Gamma(\overline{\mbfs{\beta}}), \overline{\mbfs{\beta}})\| = \sqrt{2}$, and $-\mbfs{g}_t^\top \mbfs{x} + \mbfs{b}_t^\top \mbfs{y} \leq r_t$ is valid for $C^I_\Gamma$ and separates $(\mbfs{x}_t, \mbfs{y}_t)$, i.e.,
\[
-\mbfs{g}_t^\top \mbfs{x}_t + \mbfs{b}_t^\top \mbfs{y}_t \geq r_t.
\]
We finish the proof of the lemma by showing $\lim_{t\to\infty}( \mbfs{g}_t, \mbfs{b}_t , r_t) =( \Gamma(\overline{\mbfs{ \beta}}), \overline{\mbfs{ \beta}},0)$.

For each $t$, we have $r_t \ge 0$ because $(\mbfs{0}, \mbfs{0}) \in C^I_{\Gamma}$.
The inequality $r_t \ge 0$ implies that each inequality $-\mbfs{g}_t^\top \mbfs{x} + \mbfs{b}_t^\top \mbfs{y} \leq 0$ is valid for  $C^I_\Gamma$.
The inequality $-\mbfs{g}_t^\top \mbfs{x} + \mbfs{b}_t^\top \mbfs{y} \leq 0$ is then also valid for  $C^J_\Gamma$ because $C^I_\Gamma = C^J_\Gamma$ by Lemma \ref{lemma:exposedrays}. 
As $J \subseteq D^m$ is closed, we can apply Lemma \ref{lemma:lambdabounded} and write 
\[
(\mbfs{g}_t, \mbfs{b}_t ) = \left(\sum_{i=1}^{n+m} \lambda_{[i,t]} \cdot \Gamma(\mbfs{\beta}_{[i,t]}),~~\sum_{i=1}^{n+m} \lambda_{[i,t]}\cdot \mbfs{\beta}_{[i,t]}\right),
\]
where each $\mbfs{\beta}_{[i,t]}\in  J$ and $\lambda_{[i,t]} \in [0, \tau]$ for some upper bound $\tau$ that only depends on $\Gamma$.

As $-\mbfs{g}_t^\top \mbfs{x}_t + \mbfs{b}_t^\top \mbfs{y}_t \geq r_t \ge 0$, we have
\begin{align}
0 \geq\,& \mbfs{g}_t^\top \widehat{\mbfs{x}}_t - \mbfs{b}_t^\top \widehat{\mbfs{y}}_t \nonumber \\
=\,& \mbfs{g}_t^\top \Gamma(\mbfs{\beta}_t) - \mbfs{b}_t^\top \mbfs{\beta}_t  + \varepsilon_t^2 \left(\mbfs{g}_t^\top \overline{\mbfs{x}} - \mbfs{b}_t^\top \overline{\mbfs{y}}\right)
- \left(\alpha_t \varepsilon_t^2 + \varepsilon_t^3\right)\mbfs{b}_t^\top \mbfs{\beta}_t.
\label{eq:lasteq_exposing}
\end{align}
In \eqref{eq:lasteq_exposing}, the term $\mbfs{g}_t^\top \Gamma({\mbfs{\beta}_t})  - \mbfs{b}_t^\top {\mbfs{\beta}_t}$ can be lower bounded using the assumption that $\Gamma$ is non-expansive:
\begin{equation}\label{gandGamma}
\mbfs{g}_t^\top \Gamma({\mbfs{\beta}_t})  - \mbfs{b}_t^\top {\mbfs{\beta}_t} = \sum_{i=1}^{n+m} \lambda_{[i,t]} \cdot \left( \Gamma(\mbfs{\beta}_{[i,t]}){}^\top \Gamma({\mbfs{\beta}_t}) - \mbfs{\beta}_{[i,t]}{}^\top {\mbfs{\beta}_t} \right) \geq 0.
\end{equation}
Similarly, for an arbitrary $k\in \{1, \dotsc, n+m\}$, the term $\mbfs{g}_t^\top \overline{\mbfs{x}} - \mbfs{b}_t^\top \overline{\mbfs{y}}$ in \eqref{eq:lasteq_exposing} can be lower bounded as follows: 
\begin{align}
 \mbfs{g}_t^\top \overline{\mbfs{x}}  - \mbfs{b}_t^\top \overline{\mbfs{y}} & = \sum_{i=1}^{n+m} \lambda_{[i,t]} \cdot \left( \Gamma(\mbfs{\beta}_{[i,t]}){}^\top \overline{\mbfs{x}} - \mbfs{\beta}_{[i,t]}{}^\top \overline{\mbfs{y}} \right)\nonumber\\
 & \ge \lambda_{[k,t]} \cdot \left( \Gamma(\mbfs{\beta}_{[k,t]}){}^\top \overline{\mbfs{x}} - \mbfs{\beta}_{[k,t]}{}^\top \overline{\mbfs{y}} \right).\label{eqSinglek}
\end{align}
Substitute \eqref{gandGamma} and \eqref{eqSinglek} into \eqref{eq:lasteq_exposing}, then multiply by $1/(\alpha_t \varepsilon_t^2 + \varepsilon_t^3) > 0$, to obtain
\begin{align}
0 \geq\,& \frac{\lambda_{[k,t]}}{\alpha_t + \varepsilon_t} \cdot \left(\Gamma(\mbfs{\beta}_{[k,t]}){}^\top \overline{\mbfs{x}} - \mbfs{\beta}_{[k,t]}{}^\top \overline{\mbfs{y}}\right)
- \mbfs{b}_t^\top \mbfs{\beta}_t.
\label{eq:lasteq_exposing2}
\end{align}

The term $\mbfs{b}_t^\top \mbfs{\beta}_t$ is bounded because $\|(-\mbfs{g}_t, \mbfs{b}_t)\| =\sqrt{2}$ and $\|\mbfs{\beta}_t\| = 1$.
For each $k$, the sequence $((\Gamma(\mbfs{\beta}_{[k,t]}), \mbfs{\beta}_{[k,t]}))_{t=1}^{\infty}$ is bounded because $\Gamma$ is continuous; thus, after possibly taking a subsequence, we assume this sequence converges. 
Note that the limit is of the form $(\Gamma(\mbfs{\beta}), \mbfs{\beta})$ for some $\mbfs{\beta} \in J$.
Similarly, the sequence $(\lambda_{[k,t]})_{t=1}^{\infty}$ is bounded by $\tau$; thus, after possibly taking a subsequence, we assume this sequence converges to some $\lambda_k$.
If $\lambda_k > 0$ and 
\[
\lim_{t\to \infty} \Gamma(\mbfs{\beta}_{[k,t]}){}^\top \overline{\mbfs{x}} - \mbfs{\beta}_{[k,t]}{}^\top \overline{\mbfs{y}} > 0,
\]
then the right hand side of \eqref{eq:lasteq_exposing2} goes to infinity, which is a contradiction.
Thus, if $\lambda_k > 0$ for some $k$, then the previous limit equals $0$.
 By the definition of $(\overline{\mbfs{x}},\overline{\mbfs{y}})$, it follows that
\[
\lim_{t\to \infty}  \left(\Gamma(\mbfs{\beta}_{[k,t]}),\mbfs{\beta}_{[k,t]}\right) =  \left(\Gamma(\overline{\mbfs{\beta}}), \overline{\mbfs{\beta}}\right).
\]
Hence,  $(\mbfs{g}_t, \mbfs{b}_t) \to (\Gamma(\overline{\mbfs{ \beta}}), \overline{\mbfs{ \beta}})$. 

It remains to prove that $r_t \to 0$.
We have
\begin{align*}
0\leq r_t \leq &\, -\mbfs{g}_t^\top \mbfs{x}_t + \mbfs{b}_t^\top \mbfs{y}_t \\
=&  \frac{-1}{\mbfs{a}{}^\top \widehat{\mbfs{x}}_t + \mbfs{d}{}^\top \widehat{\mbfs{y}}_t } \cdot (-\mbfs{g}_t^\top \widehat{\mbfs{x}}_t + \mbfs{b}_t^\top \widehat{\mbfs{y}}_t) && \text{by definition of $({\mbfs{x}}_t , {\mbfs{y}}_t )$} \\
\leq&  \frac{- \varepsilon_t^2 \left(\mbfs{g}_t^\top \overline{\mbfs{x}} - \mbfs{b}_t^\top \overline{\mbfs{y}}\right)
+ \left(\alpha_t \varepsilon_t^2 + \varepsilon_t^3\right)\mbfs{b}_t^\top \mbfs{\beta}_t }{-(\mbfs{a}{}^\top \widehat{\mbfs{x}}_t + \mbfs{d}{}^\top \widehat{\mbfs{y}}_t )} && \text{by \eqref{eq:lasteq_exposing} and \eqref{gandGamma}} \\
\leq &\, \frac{-\varepsilon_t^2 \left(\mbfs{g}_t^\top \overline{\mbfs{x}} - \mbfs{b}_t^\top \overline{\mbfs{y}}\right)
+ \left(\alpha_t \varepsilon_t^2 + \varepsilon_t^3\right)\mbfs{b}_t^\top \mbfs{\beta}_t}{\varepsilon_t 
- \varepsilon_t^2 \mbfs{a}^\top \overline{\mbfs{x}}  
- (\alpha_t\varepsilon_t^2+\varepsilon_t^3)\, \mbfs{d}{}^\top \mbfs{\beta}_t  
- \varepsilon_t^2 \mbfs{d}{}^\top\overline{\mbfs{y}}} && \text{by \eqref{ahatbound}}\\
= &\, \frac{-\varepsilon_t \left(\mbfs{g}_t^\top \overline{\mbfs{x}} - \mbfs{b}_t^\top \overline{\mbfs{y}}\right)
+ \left(\alpha_t \varepsilon_t + \varepsilon_t^2\right)\mbfs{b}_t^\top \mbfs{\beta}_t}{1 
- \varepsilon_t \mbfs{a}^\top \overline{\mbfs{x}}  
- (\alpha_t\varepsilon_t +\varepsilon_t^2)\, \mbfs{d}{}^\top \mbfs{\beta}_t  
- \varepsilon_t \mbfs{d}{}^\top\overline{\mbfs{y}}}.
\end{align*}
The latter expression goes to $0$ as $t\to\infty$.
Hence, $r_t \to 0$.
\end{proof}

\section{A proof of Theorem~\ref{thm:easycase}}\label{sec:Maximality_easycase}

Assume $\|\mbfs{a}\| \le \|\mbfs{d}\|\ = 1$, and let $\Gamma: D^m \to D^n$ be such that $C_{\Gamma}$ is a full-dimensional maximal $Q_h$-free set.
If $C^G_{\Gamma} \cap H$ is a full-dimensional $Q_g$-free set, then $\mbfs{a} \neq \Gamma(-\mbfs{d})$ according to Lemma~\ref{lemma:necessary_easy}.
It remains to prove that if $\mbfs{a} \neq \Gamma(-\mbfs{d})$, then $C^G_{\Gamma} \cap H$ is a full-dimensional maximal $Q_g$-free set.
Thus, assume that $\mbfs{a} \neq \Gamma(-\mbfs{d})$.

We will apply the maximality criterion of Theorem \ref{thm:maximalitycriterion}, for which we need three technical conditions to be met.

First, we need $C^G_\Gamma \cap H$ to be full-dimensional.
As $\Gamma$ is non-expansive, we have $(\Gamma(-\mbfs{d}), -\mbfs{d}) \in C_{\Gamma}$.
By the Cauchy-Schwarz Inequality, $\mbfs{a}{}^\top \Gamma(-\mbfs{d}) +(-\mbfs{d}){}^\top\overline{\mbfs{y}} \le 0$ with equality if and only if $\Gamma(-\mbfs{d}) = \mbfs{a}$, which we assume does not happen.
Hence, $C_\Gamma \cap H \subseteq C^G_{\Gamma} \cap H$ is full-dimensional by Lemma~\ref{lem:fulldim}.

Second, we need $C^G_{\Gamma} \cap H$ to be $Q_g$-free.
This follows from Lemma \ref{lemma:patching}. 

Third, we need a description $C^I_{\Gamma} \cap H$ of $C^G_{\Gamma} \cap H$ for which every inequality has an exposing sequence in $Q_g$ with respect to $C^I_{\Gamma}$.
By Lemma~\ref{lem:Implied0}, we have
\[
K = C^{G}_{\Gamma} \cap H=C^{G^{\circ}}_{\Gamma} \cap H = C^{\cl(G^{\circ})}_{\Gamma} \cap H.
\]
Apply Lemma \ref{lemma:exposedrays} with $J= \cl(G^{\circ})$ to obtain
\(
 K =  C^I_{\Gamma} \cap H,
\)
where $I \subseteq \cl(G^{\circ})$ is the set of $\overline{\mbfs{\beta}} \in \cl(G^{\circ})$ such that $(-\Gamma(\overline{\mbfs{\beta}}), \overline{\mbfs{\beta}})$ generates an exposed ray of $\cone(\{(-\Gamma(\mbfs{\beta}), \mbfs{\beta}) \,: \, \mbfs{\beta} \in \cl(G^{\circ})\})$.
By Lemma \ref{lemma:exposing_common}, the inequality $-\Gamma(\mbfs{\beta}){}^\top\mbfs{x}+\mbfs{\beta}{}^\top\mbfs{y} \le 0$ has an exposing sequence in $Q_g$ with respect to $C_\Gamma^I$ for each $\mbfs{\beta} \in I$.

Thus, the three conditions of Theorem \ref{thm:maximalitycriterion} are met. 
We conclude that $C^G_{\Gamma} \cap H$ is a maximal $Q_g$-free set.
\qed

\begin{rmk}\label{rm:GcircG}
According to Lemma~\ref{lem:Implied0}, $C^{G^{\circ}}_{\Gamma} \cap H = C^G_{\Gamma} \cap H$ in the setting of Theorem~\ref{thm:easycase}.
Thus, we can rewrite the theorem to use $C^{G^{\circ}}_{\Gamma}$ rather than $C^{G}_{\Gamma}$.
We choose to write the result in terms of $G$ because we find it better connects with Theorem~\ref{thm:hardcase}.  
\end{rmk}

\section{A proof of Theorem~\ref{thm:hardcase}}\label{sec:Maximality_hardcase}

Assume $\|\mbfs{d}\| < \|\mbfs{a}\|\ = 1$, and let $\Gamma: D^m \to D^n$ be such that $C_{\Gamma}$ is a full-dimensional maximal $Q_h$-free set.
If $C^G_{\Gamma} \cap H$ is a full-dimensional $Q_g$-free set, then $G = D^m$ according to Lemma~\ref{lemma:hardcase_necessary}.
It remains to prove that if $G = D^m$, then $C^G_{\Gamma} \cap H$ is a full-dimensional maximal $Q_g$-free set.
Thus, assume that $G = D^m$.
Hence, $C^G_{\Gamma} = C_{\Gamma}$.

We will apply the maximality criterion of Theorem \ref{thm:maximalitycriterion}, for which we need three technical conditions to be met.

First, we need $C^G_\Gamma \cap H = C_{\Gamma} \cap H$ to be full-dimensional.
The assumption $G = D^m$ implies that $-(-\mbfs{a}){}^\top \Gamma(\mbfs{\beta}) + \mbfs{d}{}^\top \mbfs{\beta}  \le 0$ for each $\mbfs{\beta} \in D^m$.
Thus, $(-\mbfs{a},\mbfs{d}) \in C_{\Gamma}$.
Furthermore, the inequality $\|\mbfs{d}\| < \|\mbfs{a}\| =1$ implies that $\alpha := -\mbfs{a}{}^\top \mbfs{a} + \mbfs{d}{}^\top \mbfs{d} < 0 $.
Apply Lemma \ref{lem:fulldim} with $(\overline{\mbfs{x}}, \overline{\mbfs{y}}) := \frac{1}{|\alpha|} \cdot (-\mbfs{a}, \mbfs{d})$ to conclude that $C^G_\Gamma \cap H$ is full-dimensional.

Second, we need $C^G_{\Gamma} \cap H$ to be $Q_g$-free.
This follows directly because $G = D^m$, and $C^{G}_{\Gamma} \cap H = C_{\Gamma} \cap H$ is $Q_g$-free.

Third, we need to show that, for each $\mbfs{\beta} \in G$, the inequality $-\Gamma(\mbfs{\beta}){}^\top\mbfs{x}+\mbfs{\beta}{}^\top\mbfs{y} \le 0$ has an exposing sequence in $Q_g$ with respect to $C^G_{\Gamma}$.
Proving this third step constitutes the remainder of our proof.

Let $I \subseteq G$ be the set of $\overline{\mbfs{\beta}} \in G$ such that $(-\Gamma(\overline{\mbfs{\beta}}), \overline{\mbfs{\beta}})$ generates an exposed ray of $\cone(\{(-\Gamma(\mbfs{\beta}), \mbfs{\beta}) \,: \, \mbfs{\beta} \in G\})$.
Apply Lemma \ref{lemma:exposedrays} with $J= G$ to conclude that $C^G_{\Gamma} = C^I_{\Gamma}$.

For each $\overline{\mbfs{\beta}} \in I \cap \cl(G^\circ)$, Lemma \ref{lemma:exposing_common} states that the inequality $-\Gamma(\overline{\mbfs{\beta}}){}^\top\mbfs{x}+\overline{\mbfs{\beta}}{}^\top\mbfs{y} \le 0$ has an exposing sequence in $Q_g$ with respect to $C_\Gamma^I $.
Thus, we only need to find an exposing sequence for each $\overline{\mbfs{\beta}} \in I \setminus \cl(G^\circ)$.

Let $\overline{\mbfs{\beta}} \in I \setminus \cl(G^\circ)$.
As $\overline{\mbfs{\beta}}$ defines an exposed ray of $\cone(\{(-\Gamma(\mbfs{\beta}), \mbfs{\beta}) \,: \, \mbfs{\beta} \in G\})$, and $G=D^m$, there exists a vector $(\overline{\mbfs{x}}, \overline{\mbfs{y}}) \in \mbb{R}^n\times \mbb{R}^m$ such that
\[
-\Gamma(\overline{\mbfs{\beta}})^\top \overline{\mbfs{x}} + \overline{\mbfs{\beta}}{}^\top \overline{\mbfs{y}}  = 0 > 
-\Gamma({\mbfs{\beta}})^\top \overline{\mbfs{x}} + \mbfs{\beta}{}^\top \overline{\mbfs{y}} \qquad \forall ~\mbfs{\beta}\in D^m \setminus\{  \overline{\mbfs{\beta}} \}.
\]
For $t \in \mbb{N}$, define the vectors 
\begin{equation}\label{eq:exposingsequencedef}
\begin{array}{rcl}
 \widetilde{\mbfs{x}}_t &:=&\displaystyle \Gamma(\overline{\mbfs{\beta}}) + \frac{1}{t^{3/2}}\cdot \overline{\mbfs{x}} - \frac{1}{t} \cdot \mbfs{a},\\[.2cm]
 \widetilde{\mbfs{y}}_t &:=&\displaystyle\left( 1 + \frac{1}{t^2} \right)\overline{\mbfs{\beta}} \;+\; \frac{1}{t^{3/2}}\cdot \overline{\mbfs{y}} + \frac{1}{t} \cdot \mbfs{d}.
\end{array}
\end{equation}
We have $\mbfs{a}{}^\top \Gamma(\overline{\mbfs{\beta}}) + \mbfs{d}{}^\top \overline{\mbfs{\beta}} = 0$ because $\overline{\mbfs{\beta}}\in I \subseteq G$ and $\overline{\mbfs{\beta}} \not\in \cl(G^\circ)$.
Thus,
\begin{align}
& \mbfs{a}{}^\top \widetilde{\mbfs{x}}_t  + \mbfs{d}{}^\top \widetilde{\mbfs{y}}_t\nonumber\\
  \;=&\; \mbfs{a}{}^\top \Gamma(\overline{\mbfs{\beta}}) + O\left(\frac{1}{t^{3/2}}\right) - \frac{1}{t} + \left( 1 + \frac{1}{t^2} \right) \mbfs{d}{}^\top \overline{\mbfs{\beta}} + O\left(\frac{1}{t^{3/2}}\right) + \frac{\|\mbfs{d}\|^2}{t}\nonumber\\
 \;=&\; \frac{\|\mbfs{d}\|^2 - 1}{t} + O\left(\frac{1}{t^{3/2}}\right)\label{negativeden}.
\end{align}
Hence, $\mbfs{a}{}^\top \widetilde{\mbfs{x}}_t  + \mbfs{d}{}^\top \widetilde{\mbfs{y}}_t < 0$ for large $t$ because $\|\mbfs{d}\| < 1$.
After taking a subsequence, we assume $\mbfs{a}{}^\top \widetilde{\mbfs{x}}_t + \mbfs{d}{}^\top \widetilde{\mbfs{y}}_t  < 0$ for all $t$.

For each $t$, define the following rescaled version of $(\widetilde{\mbfs{x}}_t , \widetilde{\mbfs{y}}_t)$:
\begin{equation}\label{eq:exposingsequencedef_rescaled}
(\mbfs{x}_t , \mbfs{y}_t) \; :=\; -\,\frac{1}{\,\mbfs{a}{}^\top \widetilde{\mbfs{x}}_t + \mbfs{d}{}^\top \widetilde{\mbfs{y}}_t\,}\, \cdot(\widetilde{\mbfs{x}}_t , \widetilde{\mbfs{y}}_t).
\end{equation}
Given that $\mbfs{a}{}^\top \widetilde{\mbfs{x}}_t + \mbfs{d}{}^\top \widetilde{\mbfs{y}}_t  < 0$, we have $(\mbfs{x}_t, \mbfs{y}_t)\in H$. 

We claim that if $t$ is large enough, then $(\mbfs{x}_t, \mbfs{y}_t)\in Q_h$.
For this, it suffices to show $\|\widetilde{\mbfs{x}}_t\|^2 - \|\widetilde{\mbfs{y}}_t\|^2 \leq 0$.
We have
\begin{align*}
\|\widetilde{\mbfs{x}}_t\|^2 & = 1 + O\left(\frac{1}{t^{3}}\right) + \frac{1}{t^2} + \frac{2}{t^{3/2}} \Gamma(\overline{\mbfs{\beta}})^\top \bar x + O\left(\frac{1}{t^{5/2}}\right) - \frac{2}{t} \Gamma(\overline{\mbfs{\beta}})^\top \mbfs{a}\\
& = 1 + \frac{1}{t^2} + \frac{2}{t^{3/2}} \Gamma(\overline{\mbfs{\beta}})^\top \overline{\mbfs{x}} - \frac{2}{t} \Gamma(\overline{\mbfs{\beta}})^\top \mbfs{a} + O\left(\frac{1}{t^{3}}\right)
\intertext{and}
\|\widetilde{\mbfs{y}}_t\|^2 & = 1+\frac{1}{t^4} + \frac{2}{t^2} + O\left(\frac{1}{t^{3}}\right) + \frac{\|\mbfs{d}\|^2}{t^2}  + O\left(\frac{1}{t^{5/2}}\right) \\
& \quad + \frac{2}{t^{3/2}}\overline{\mbfs{\beta}}{}^\top \overline{\mbfs{y}}  + \frac{2}{t^{7/2}} \overline{\mbfs{\beta}}{}^\top \overline{\mbfs{ y}} + \frac{2}{t} \mbfs{d}{}^\top \overline{\mbfs{ \beta}} + \frac{2}{t^3} \overline{\mbfs{\beta}}{}^\top \mbfs{d}.
\end{align*}
Therefore, 
\begin{align*}
\|\widetilde{\mbfs{x}}_t\|^2 - \|\widetilde{\mbfs{y}}_t\|^2
  \;= & \; \frac{-2}{t}\left( \Gamma(\overline{\mbfs{ \beta}})^\top \mbfs{a} + \mbfs{d}{}^\top \overline{\mbfs{ \beta}} \right) + \frac{2}{t^{3/2}} \left( \Gamma(\overline{\mbfs{ \beta}})^\top \overline{\mbfs{ x}} - \overline{\mbfs{\beta}}{}^\top \overline{\mbfs{ y}}\right) \\
& - \frac{1}{t^2} \left( 1 + \|\mbfs{d}\|^2\right) + O\left(\frac{1}{t^{5/2}}\right).
\end{align*}
By the definition of $\overline{\mbfs{ \beta}}$, we have $\Gamma(\overline{\mbfs{ \beta}})^\top \mbfs{a} + \mbfs{d}{}^\top \overline{\mbfs{ \beta}} = 0$.
Similarly, by the definition of $(\overline{\mbfs{ x}}, \overline{\mbfs{ y}})$, we have $\Gamma(\overline{\mbfs{ \beta}})^\top \overline{\mbfs{ x}} - \overline{\mbfs{\beta}}{}^\top \overline{\mbfs{ y}}= 0$.
Thus,
\[
t^2 \cdot (\|\widetilde{\mbfs{x}}_t\|^2 - \|\widetilde{\mbfs{y}}_t\|^2) =  -1 - \|\mbfs{d}\|^2 + o(1).
\]
Thus, $\|\widetilde{\mbfs{x}}_t\|^2 \leq \|\widetilde{\mbfs{y}}_t\|^2$ for large enough $t$.
After taking a subsequence, we assume $\|\widetilde{\mbfs{x}}_t\|^2 \leq \|\widetilde{\mbfs{y}}_t\|^2$ for each $t$.
Thus, $(\mbfs{x}_t, \mbfs{y}_t) \in Q_h$ for each $t$.
Consequently, $(\mbfs{x}_t, \mbfs{y}_t) \in Q_h \cap H = Q_g$ for each $t$.

The final step is to prove that $((\mbfs{x}_t, \mbfs{y}_t))_{t=1}^{\infty}$ is an exposing sequence for the inequality $-\Gamma(\overline{\mbfs{\beta}}){}^\top\mbfs{x}+\overline{\mbfs{\beta}}{}^\top\mbfs{y} \le 0$ with respect to $C^G_{\Gamma}$.

Take any sequence $((-\mbfs{g}_t, \mbfs{b}_t, r_t))_{t=1}^{\infty}$ in $\mbb{R}^n\times \mbb{R}^m\times \mbb{R}$ such that, for each $t$, we have $\|(-\mbfs{g}_t, \mbfs{b}_t)\| = \|(-\Gamma(\overline{\mbfs{\beta}}), \overline{\mbfs{\beta}})\| = \sqrt{2}$, and $-\mbfs{g}_t^\top \mbfs{x} + \mbfs{b}_t^\top \mbfs{y} \leq r_t$ is valid for $C_\Gamma$ and separates $(\mbfs{x}_t, \mbfs{y}_t)$, i.e.,
\[
-\mbfs{g}_t^\top \mbfs{x}_t + \mbfs{b}_t^\top \mbfs{y}_t \geq r_t\quad \forall~ t.
\]
We finish the proof of the theorem by showing $\lim_{t\to\infty}( \mbfs{g}_t, \mbfs{b}_t , r_t) =( \Gamma(\overline{\mbfs{ \beta}}), \overline{\mbfs{ \beta}},0)$.

For each $t$, we have $r_t \ge 0$ because $(\mbfs{0}, \mbfs{0}) \in C_{\Gamma}$.
Thus, $-\mbfs{g}_t^\top \mbfs{x} + \mbfs{b}_t^\top \mbfs{y} \leq 0$ is valid for $C_\Gamma$.
By Lemma \ref{lemma:lambdabounded} with $J = D^m$, we can write 
\[
\mbfs{g}_t = \sum_{i=1}^{n+m} \lambda_{[i,t]} \Gamma(\mbfs{\beta}_{[i,t]}) \quad \text{and}\quad
\mbfs{b}_t = \sum_{i=1}^{n+m} \lambda_{[i,t]} \mbfs{\beta}_{[i,t]},
\]
where each $\mbfs{\beta}_{[i,t]}\in  D^m$ and $\lambda_{[i,t]} \in [0,\tau]$ for some upper bound $\tau$ that only depends on $\Gamma$.

The inequality $r_t \ge 0$ also implies $-\mbfs{g}_t^\top \mbfs{x}_t + \mbfs{b}_t^\top \mbfs{y}_t \geq 0$.
Thus
\begin{align}
0 \geq& \, \mbfs{g}_t^\top \widetilde{\mbfs{x}}_t - \mbfs{b}_t^\top \widetilde{\mbfs{y}}_t \nonumber\\
=&\, \mbfs{g}_t^\top \Gamma(\overline{\mbfs{\beta}}) + \frac{1}{t^{3/2}} \mbfs{g}_t^\top \overline{\mbfs{x}}  - \frac{1}{t} \mbfs{g}_t^\top \mbfs{a}   - \left(1 + \frac{1}{t^2}\right)\mbfs{b}_t^\top \overline{\mbfs{\beta}} - \frac{1}{t^{3/2}} \mbfs{b}_t^\top \overline{\mbfs{y}} - \frac{1}{t}\mbfs{b}_t^\top \mbfs{d} \nonumber\\
=&\, \mbfs{g}_t^\top \Gamma(\overline{\mbfs{\beta}}) - \mbfs{b}_t^\top \overline{\mbfs{\beta}}  + \frac{1}{t^{3/2}}\left( \mbfs{g}_t^\top \overline{\mbfs{x}} -  \mbfs{b}_t^\top \overline{\mbfs{y}}\right) - \frac{1}{t} \left( \mbfs{g}_t^\top \mbfs{a}  + \mbfs{b}_t^\top \mbfs{d}\right)  -\frac{1}{t^2} \mbfs{b}_t^\top \overline{\mbfs{\beta}}.\label{eq:last}
\end{align}
In \eqref{eq:last}, the term $\mbfs{g}_t^\top \Gamma(\overline{\mbfs{\beta}}) - \mbfs{b}_t^\top \overline{\mbfs{\beta}}$ can be lower bounded as
\begin{equation}\label{eq:last2}
\mbfs{g}_t^\top \Gamma(\overline{\mbfs{\beta}}) - \mbfs{b}_t^\top \overline{\mbfs{\beta}} = \sum_{i=1}^{n+m} \lambda_{[i,t]} \left( \Gamma(\mbfs{\beta}_{[i,t]}){}^\top \Gamma(\overline{\mbfs{\beta}})  - \mbfs{\beta}_{[i,t]}{}^\top \overline{\mbfs{ \beta}} \right) \geq 0,
\end{equation}
where the inequality follow because $\Gamma$ is non-expansive. 
Similarly, in \eqref{eq:last}, the term $\mbfs{g}_t^\top \mbfs{a} + \mbfs{b}_t^\top \mbfs{d} $ can be upper bounded as
\begin{equation}\label{eq:last3}
\mbfs{g}_t^\top \mbfs{a} + \mbfs{b}_t^\top \mbfs{d} = \sum_{i=1}^{n+m} \lambda_{[i,t]} \left( \Gamma(\mbfs{\beta}_{[i,t]}){}^\top \mbfs{a}  + \mbfs{\beta}_{[i,t]}{}^\top \mbfs{d} \right) \leq 0
\end{equation}
because $\mbfs{a}{}^\top \Gamma(\mbfs{\beta}) + \mbfs{d}{}^\top \mbfs{\beta} \leq 0$ for all $\mbfs{\beta}  \in D^m$.
Finally, using an arbitrary $k\in \{1, \dotsc, n+m\}$, the term $ \mbfs{g}_t^\top \overline{\mbfs{x}}  - \mbfs{b}_t^\top \overline{\mbfs{y}}$ in \eqref{eq:last} can be lower bounded as
\begin{align*}
 \mbfs{g}_t^\top \overline{\mbfs{x}}  - \mbfs{b}_t^\top \overline{\mbfs{y}} =& \sum_{i=1}^{n+m} \lambda_{[i,t]} \left( \Gamma(\mbfs{\beta}_{[i,t]}){}^\top \overline{\mbfs{x}} - \mbfs{\beta}_{[i,t]}{}^\top \overline{\mbfs{y}} \right) \\
\geq& \lambda_{[k,t]} \left( \Gamma(\mbfs{\beta}_{[k,t]}){}^\top \overline{\mbfs{x}} - \mbfs{\beta}_{[k,t]}{}^\top \overline{\mbfs{y}} \right).
\end{align*}
Substituting these bounds into \eqref{eq:last}, we get
\[
0 \geq \frac{1}{t^{3/2}}\lambda_{[k,t]} \left( \Gamma(\mbfs{\beta}_{[k,t]}){}^\top \overline{\mbfs{x}} - \mbfs{\beta}_{[k,t]}{}^\top \overline{\mbfs{y}} \right)  -\frac{1}{t^2} \mbfs{b}_t^\top \overline{\mbfs{\beta}}.
\]
If we multiply this by $t^2$, then
\begin{equation}\label{eq:lastineq_exposing_hard}
0 \geq  \lambda_{[k,t]} \sqrt{t} \left( \Gamma(\mbfs{\beta}_{[k,t]}){}^\top \overline{\mbfs{x}} - \mbfs{\beta}_{[k,t]}{}^\top \overline{\mbfs{y}} \right)  - \mbfs{b}_t^\top \overline{\mbfs{\beta}} .
\end{equation}
The sequence $(\mbfs{b}_t^\top \overline{\mbfs{\beta}})_{t=1}^{\infty}$ is bounded because $(\mbfs{b}_t)_{t=1}^{\infty}$ is bounded.  
For each $k$, the sequence $((\Gamma(\mbfs{\beta}_{[k,t]}), \mbfs{\beta}_{[k,t]}))_{t=1}^{\infty}$ is bounded because $\Gamma$ is continuous; thus, after possibly taking a subsequence, we assume this sequence converges. 
Note that the limit is of the form $(\Gamma(\mbfs{\beta}), \mbfs{\beta})$ for some $\mbfs{\beta} \in D^m$.
Similarly, the sequence $(\lambda_{[k,t]})_{t=1}^{\infty}$ is bounded by $\tau$; thus, after possibly taking a subsequence, we assume this sequence converges to some $\lambda_k$.
If $\lambda_k > 0$ and 
\[
\lim_{t\to \infty} \Gamma(\mbfs{\beta}_{[k,t]}){}^\top \overline{\mbfs{x}} - \mbfs{\beta}_{[k,t]}{}^\top \overline{\mbfs{y}} > 0,
\]
then the right hand side of \eqref{eq:lastineq_exposing_hard} goes to infinity, which is a contradiction.
Thus, if $\lambda_k > 0$ for some $k$, then the previous limit equals $0$.
By the definition of $(\overline{\mbfs{x}},\overline{\mbfs{y}})$, it follows that
\[
\lim_{t\to \infty}  \left(\Gamma(\mbfs{\beta}_{[k,t]}),\mbfs{\beta}_{[k,t]}\right) =  \left(\Gamma(\overline{\mbfs{\beta}}), \overline{\mbfs{\beta}}\right).
\]
Hence,  $(\mbfs{g}_t, \mbfs{b}_t) \to (\Gamma(\overline{\mbfs{ \beta}}), \overline{\mbfs{ \beta}})$. 

It remains to prove that $r_t \to 0$.
We have
\begin{align*}
0\leq r_t \leq &\, -\mbfs{g}_t^\top \mbfs{x}_t + \mbfs{b}_t^\top \mbfs{y}_t \\
=&  \frac{-1}{\mbfs{a}{}^\top \widetilde{\mbfs{x}}_t + \mbfs{d}{}^\top \widetilde{\mbfs{y}}_t } (-\mbfs{g}_t^\top \widetilde{\mbfs{x}}_t + \mbfs{b}_t^\top \widetilde{\mbfs{y}}_t) && \text{by the definition of $({\mbfs{x}}_t , {\mbfs{y}}_t )$} \\
\leq&  \frac{ - \frac{1}{t^{3/2}}\left( \mbfs{g}_t^\top \overline{\mbfs{x}} +  \mbfs{b}_t^\top \overline{\mbfs{y}}\right)  -\frac{1}{t^2} \mbfs{b}_t^\top \overline{\mbfs{\beta}} }{-(\mbfs{a}{}^\top \widetilde{\mbfs{x}}_t + \mbfs{d}{}^\top \widetilde{\mbfs{y}}_t )}   &&\text{by \eqref{eq:last}, \eqref{eq:last2}, and \eqref{eq:last3}}\\
=&  \frac{ - \frac{1}{t^{1/2}}\left( \mbfs{g}_t^\top \overline{\mbfs{x}} +  \mbfs{b}_t^\top \overline{\mbfs{y}}\right)  -\frac{1}{t} \mbfs{b}_t^\top \overline{\mbfs{\beta}} }{-t (\mbfs{a}{}^\top \widetilde{\mbfs{x}}_t + \mbfs{d}{}^\top \widetilde{\mbfs{y}}_t )} .
\end{align*}
The last numerator goes to $0$ as $t\to\infty$, while the denominator goes to $1-\|\mbfs{d}\|^2 \neq 0$ by \eqref{negativeden}.
Hence, $r_t \to 0$.
\hfill \qed

\section*{Acknowledgements}

J. Paat was supported by a Natural Sciences and Engineering Research Council of Canada Discovery Grant [RGPIN-2021-02475].
G. Muñoz was supported by the Chilean Agency of Research and Development (ANID) through Fondecyt Grant 1231522.

\appendix

\bibliographystyle{plain}
\bibliography{../references}

\begin{thebibliography}{10}

\bibitem{AJ2013}
K.~Andersen and {A.N.} Jensen.
\newblock Intersection cuts for mixed integer conic quadratic sets.
\newblock In M.~Goemans and J.~Correa, editors, {\em Integer Programming and
  Combinatorial Optimization}, pages 37--48. Springer Berlin Heidelberg,
  Berlin, Heidelberg, 2013.

\bibitem{ALWW2007}
K.~Andersen, Q.~Louveaux, R.~Weismantel, and {L.A.} Wolsey.
\newblock Cutting planes from two rows of the simplex tableau.
\newblock In M.~Fischetti and D.~P. Williamson, editors, {\em Proceedings of
  Integer Programming and Combinatorial Optimization (IPCO)}, pages 1--15,
  Berlin, Heidelberg, 2007. Springer Berlin Heidelberg.

\bibitem{averkov2013maximal}
G.~Averkov.
\newblock On maximal {S}-free sets and the {H}elly number for the family of
  {S}-convex sets.
\newblock {\em SIAM Journal on Discrete Mathematics}, 27(3):1610--1624, 2013.

\bibitem{A2013}
G.~Averkov.
\newblock A proof of {L}ov\'{a}sz's theorem on maximal lattice-free sets.
\newblock {\em Contributions to Algebra and Geometry}, 54:105--109, 2013.

\bibitem{ABP2018}
G.~Averkov, A.~Basu, and J.~Paat.
\newblock Approximation of corner polyhedra with families of intersection cuts.
\newblock {\em {SIAM} Journal on Optimization}, 28(1):904--929, 2018.

\bibitem{BOW2016}
M.~Baes, T.~Oertel, and R.~Weismantel.
\newblock Duality for mixed-integer convex minimization.
\newblock {\em {Mathematical Programming}}, 158:547--564, 2016.

\bibitem{B1971}
E.~Balas.
\newblock Intersection cuts - a new type of cutting planes for integer
  programming.
\newblock {\em Operations Research}, 19(1):19--39, 1971.

\bibitem{BCCWW2017}
A.~Basu, M.~Conforti, G.~Cornu{\'e}jols, R.~Weismantel, and S.~Weltge.
\newblock Optimality certificates for convex minimization and {H}elly numbers.
\newblock {\em Operations Research Letters}, 45(6):671--674, 2017.

\bibitem{BCCZ2010}
A.~Basu, M.~Conforti, G.~Cornu\'{e}jols, and G.~Zambelli.
\newblock Maximal lattice-free convex sets in linear subspaces.
\newblock {\em Mathematics of Operations Research}, 35(3):704--720, 2010.

\bibitem{basu2010minimal}
A.~Basu, M.~Conforti, G.~Cornu{\'e}jols, and G.~Zambelli.
\newblock Minimal inequalities for an infinite relaxation of integer programs.
\newblock {\em SIAM Journal on Discrete Mathematics}, 24(1):158--168, 2010.

\bibitem{BDP2019}
A.~Basu, {S.S.} Dey, and J.~Paat.
\newblock Nonunique lifting of integer variables in minimal inequalities.
\newblock {\em {SIAM} Journal on Discrete Mathematics}, 33(2):755--783, 2019.

\bibitem{BCM2019}
D.~Bienstock, C.~Chen, and G.~Mu{\~n}oz.
\newblock Intersection cuts for polynomial optimization.
\newblock In A.~Lodi and V.~Nagarajan, editors, {\em Integer Programming and
  Combinatorial Optimization}, pages 72--87, Cham, 2019. Springer International
  Publishing.

\bibitem{BCM2020}
D.~Bienstock, C.~Chen, and G.~Mu{\~n}oz.
\newblock Outer-product-free sets for polynomial optimization and oracle-based
  cuts.
\newblock {\em Mathematical Programming}, 183:105--148, 2020.

\bibitem{chmiela2022implementation}
A.~Chmiela, G.~Mu{\~n}oz, and F.~Serrano.
\newblock On the implementation and strengthening of intersection cuts for
  {QCQP}s.
\newblock {\em Mathematical Programming}, 197:549--586, 2023.

\bibitem{chmiela2025monoidal}
A.~Chmiela, G.~Mu{\~n}oz, and F.~Serrano.
\newblock Monoidal strengthening and unique lifting in {MIQCP}s.
\newblock {\em Mathematical Programming}, 210(1):189--222, 2025.

\bibitem{CCDLM2014}
M.~Conforti, G.~Cornu\'{e}jols, A.~Daniilidis, C.~Lemar\'{e}chal, and
  J.~Malick.
\newblock Cut-generating functions and {S}-free sets.
\newblock {\em Mathematics of Operations Research}, 40(2):276--301, 2015.

\bibitem{CCZ2014}
M.~Conforti, G.~Cornu{\'e}jols, and G.~Zambelli.
\newblock {\em Integer Programming}.
\newblock Springer, 2014.

\bibitem{conforti2016maximal}
M.~Conforti and M.~{Di Summa}.
\newblock Maximal {S}-free convex sets and the {H}elly number.
\newblock {\em SIAM Journal on Discrete Mathematics}, 30(4):2206--2216, 2016.

\bibitem{fischetti2016intersection}
M.~Fischetti, I.~Ljubi{\'c}, M.~Monaci, and M.~Sinnl.
\newblock Intersection cuts for bilevel optimization.
\newblock In {\em International conference on integer programming and
  combinatorial optimization}, pages 77--88. Springer, 2016.

\bibitem{fischetti2020branch}
M.~Fischetti and M.~Monaci.
\newblock A branch-and-cut algorithm for mixed-integer bilinear programming.
\newblock {\em European Journal of Operational Research}, 282(2):506--514,
  2020.

\bibitem{L1989}
L.~Lov\'{a}sz.
\newblock Geometry of numbers and integer programming.
\newblock In M.Iri and K.~Tanabe, editors, {\em Mathematical Programming:
  Recent Developments and Applications}, pages 177 -- 201. Kluwer Academic
  Publishers, 1989.

\bibitem{MKV2016}
S.~Modaresi, {M.R.} K{\i}l{\i}n{\c{c}}, and {J.P.} Vielma.
\newblock {Intersection cuts for nonlinear integer programming: convexification
  techniques for structured sets}.
\newblock {\em {Mathematical Programming}}, 155:575--611, 2016.

\bibitem{munoz2025characterization}
G.~Mu{\~n}oz, J.~Paat, and F.~Serrano.
\newblock A characterization of maximal homogeneous-quadratic-free sets.
\newblock {\em Mathematical Programming}, 210(1):641--668, 2025.

\bibitem{MS2020}
G.~Mu{\~n}oz and F.~Serrano.
\newblock Maximal quadratic-free sets.
\newblock In {\em Integer Programming and Combinatorial Optimization: 21st
  International Conference, IPCO 2020, London, UK, June 8–10, 2020,
  Proceedings}, page 307–321, Berlin, Heidelberg, 2020. Springer-Verlag.

\bibitem{MS2022}
G.~Mu{\~n}oz and F.~Serrano.
\newblock Maximal quadratic-free sets.
\newblock {\em Mathematical Programming}, 192:229--270, 2022.

\bibitem{PSS2022}
J.~Paat, M.~Schl{\"o}ter, and E.~Speakman.
\newblock Constructing lattice-free gradient polyhedra in dimension two.
\newblock {\em Mathematical Programming}, 192(1):293--317, 2022.

\bibitem{R1970}
{R.T.} Rockafellar.
\newblock {\em Convex analysis}.
\newblock Princeton {U}niversity press, 1970.

\bibitem{serrano2019intersection}
F.~Serrano.
\newblock Intersection cuts for factorable minlp.
\newblock In {\em International Conference on Integer Programming and
  Combinatorial Optimization}, pages 385--398. Springer, 2019.

\bibitem{T1964}
H.~Tuy.
\newblock Concave minimization under linear constraints with special structure.
\newblock {\em {Doklady Akademii Nauk}}, 159:32--35, 1964.

\bibitem{xu2025cutting}
L.~Xu, C.~D'Ambrosio, L.~Liberti, and S.~Haddad-Vanier.
\newblock Cutting planes for signomial programming.
\newblock {\em SIAM Journal on Optimization}, 35(2):899--926, 2025.

\bibitem{xu2025submodular}
L.~Xu and L.~Liberti.
\newblock Submodular maximization and its generalization through an
  intersection cut lens.
\newblock {\em Mathematical Programming}, 211(1):341--377, 2025.

\end{thebibliography}

\end{document}